\documentclass[12pt]{amsart}
\usepackage{amsmath,amscd,amsbsy,amssymb,amsthm,amsfonts}
\usepackage{latexsym,url,bm}
\usepackage{cite,enumitem}
\usepackage{fullpage}
\usepackage{color}
\usepackage[colorlinks=true]{hyperref}
\usepackage{chngcntr}
\usepackage{apptools}

\AtAppendix{\counterwithin{lemma}{section}}

\newtheorem{theorem}{Theorem}[section]
\newtheorem{lemma}{Lemma}
\newtheorem{proposition}[lemma]{Proposition}

\newtheorem{definition}[lemma]{Definition}

\numberwithin{lemma}{section}

\allowdisplaybreaks

\numberwithin{equation}{section}

\newcommand{\R}{{\mathbb R}}

\renewcommand{\H}{{\mathcal H }}

\newcommand{\nP}{{\mathbf P}}

\newcommand{\Half}{{\frac{1}{2}}}

\newcommand{\CalAO}{{\mathcal{A}_1}}
\newcommand{\CalAT}{{\mathcal{A}_{\frac{3}{2}}}}
\newcommand{\CalAZ}{{\mathcal{A}_0}}

\newcommand{\maH}{{\mathcal H} }

\newcommand{\W}{{\mathbf W}}

\author{Lizhe Wan}
\address{Department of Mathematics, University of Wisconsin - Madison}
\email{lwan33@wisc.edu}

\begin{document}

\title{Low regularity well-posedness for two-dimensional deep water waves}

\begin{abstract}
The study of gravity-capillary water waves in two space dimensions has been an important question in mathematical fluid dynamics.
By implementing  the \emph{cubic modified energy method} of Ifrim-Tataru in the context of gravity-capillary waves, we show that for $s> 1$, the two-dimensional gravity-capillary water wave system is locally well-posed in $\mathcal{H}^{s}$.
\end{abstract}

\maketitle

\tableofcontents

\section{Introduction}
In this article, we consider the local well-posedness of  two-dimensional  inviscid incompressible gravity-capillary water waves in infinite depth. 
Mathematically, the water waves can be described using the incompressible Euler's equations with free boundary conditions on the upper surface. 

We  start by briefly discussing the free boundary Euler equations and the Zakharov-Craig-Sulem formulation for the system. 
Let the water domain at time $t$ be $\Omega_t$, and the boundary at time $t$ be $\Gamma_t$. 
The boundary $\Gamma_t$ is assumed to be asymptotically flat at infinity or periodic with zero mean. 
The fluid velocity is denoted by $\mathbf{u}(t, x,y)$, and the pressure is denoted by $p(t,x,y)$. 
Then $(\mathbf{u}, p)$ solve the following Euler equations inside $\Omega_t$, 
\begin{equation*}
\left\{
\begin{aligned}
& \mathbf{u}_t + \mathbf{u} \cdot \nabla \mathbf{u} = \nabla p - g \mathbf{e}_2 \\
& \text{div } \mathbf{u} = 0\\
& \mathbf{u}(0,x) = \mathbf{u}_0,
\end{aligned}
\right.
\end{equation*}
while on the boundary we have the kinematic boundary condition 
\begin{equation*}
\partial_t+ \mathbf{u} \cdot \nabla \text{ is tangent to } \bigcup \Gamma_t,
\end{equation*}
and also the dynamic boundary condition
\begin{equation*}
 p = -2 \sigma \mathbb{H}  \ \ \text{ on } \Gamma_t.
\end{equation*}
Here, $\mathbb{H}$ is the mean curvature of the boundary, $g\geq 0$ is the gravitational constant, and $\sigma>0$ is the coefficient of capillary force.

For the two-dimensional fluid, the vorticity $\omega$ satisfies the transport equation $(\partial_t+ \mathbf{u} \cdot \nabla)\omega = 0$, 
and it will remain zero if it is  zero initially.
One can thus assume  the irrotationality condition of the fluid: 
\begin{equation*}
\omega  = \text{curl } \mathbf{u} = 0.   
\end{equation*}
The velocity field $\mathbf{u}$ can be expressed as the gradient of some real-valued potential $\phi$:  $\mathbf{u} = \nabla \phi$, where the velocity potential $\phi$ is a harmonic function that either decays at infinity or is periodic with zero mean.
Therefore, at each time $t$, $\phi$ solves the Laplace equation inside $\Omega_t$.
By the theory of elliptic PDEs, $\phi(t,x,y)$ is determined by its trace on the free boundary $\Gamma_t$. 
Following the work of Zakharov \cite{zakharov1968stability}, and Craig-Sulem \cite{MR1158383}, let $\eta(t,x)$ be the height of the water surface, and $\psi =\psi (t,x)$  be the trace of the velocity potential $\phi$ on the boundary,
$\psi (t,x)=\phi (t,x,\eta(t,x))$. 
Then, the water waves can be formulated in terms of a one-dimensional evolution of  the pair of unknowns $(\eta,\psi)$; it solves the nonlinear system
\begin{equation*}
\left\{
\begin{aligned}
& \partial_t \eta - G(\eta) \psi = 0 \\
& \partial_t \psi +g\eta -\sigma \mathbb{H}(\eta)  +\frac12 |\nabla \psi|^2 - \frac12 
\frac{(\nabla \eta \cdot \nabla \psi +G(\eta) \psi)^2}{1 +|\nabla \eta|^2} = 0,
\end{aligned}
\right.
\end{equation*}
where $G(\eta)$ is the Dirichlet-Neumann operator
\begin{equation*}
    G(\eta)\psi : = (-\phi_x \eta_x + \phi_y )|_{y = \eta(x)}.
\end{equation*}
This is the Zakharov-Craig-Sulem formulation of the gravity-capillary  water waves, which we write it for convenience but never used it throughout the paper.

\subsection{Water waves in holomorphic coordinates}
In the rest of this article, instead of the Zakharov-Craig-Sulem formulation, we will use the holomorphic coordinates introduced first  by Hunter-Ifrim-Tataru \cite{MR3535894} to study water waves.
The holomorphic coordinates are extensively applied to various water waves models by Ifrim and Tataru, sometimes jointly with collaborators in the following works \cite{MR3667289,MR3499085, MR3625189,MR4483135, ai2023dimensional, MR3869381}.
Let $\mathbf{P}$ be the projection onto negative frequencies, namely
\begin{equation*}
    \mathbf{P} := \frac{1}{2}(\mathbf{I} - iH),
\end{equation*}
with $H$ being the Hilbert transform.
Holomorphic functions are defined to be the complex-valued functions whose Fourier transforms are supported on $(-\infty,0]$.
In other words, they satisfy the relation $\nP f = f$.
Similarly, we define $\bar{\nP}$ to be the projection onto positive frequencies:
\begin{equation*}
    \bar{\nP} := \frac{1}{2}(\mathbf{I} + iH) = \mathbf{I} -\nP.
\end{equation*}
Complexed-valued functions such that $\bar{\nP} f = f$ are called anti-holomorphic functions. 
Anti-holomorphic functions are complex conjugates of holomorphic functions.
We will write $\alpha$ in holomorphic coordinates for the holomorphic spacial variable.

Let $W$ be the holomorphic position and  $Q$ be the holomorphic velocity potential.
The water wave system can be reexpressed in terms of pair of unknowns $(W,Q)$. 
The derivation of the gravity-capillary water wave system in holomorphic coordinates can be found in the paper Ifrim-Tataru \cite{MR3667289} and the appendix of Hunter-Ifrim-Tataru \cite{MR3535894}. 
The water wave system is given by
\begin{equation}\label{e:CWW}
\left\{
\begin{aligned}
& W_t + F (1+W_\alpha) = 0 \\
& Q_t + F Q_\alpha -ig W + \nP\left[ \frac{|Q_\alpha|^2}{J}\right] +i\sigma \nP\left[ \frac{W_{\alpha \alpha}}{J^{\Half}(1+W_{\alpha})}-\frac{\bar{W}_{\alpha \alpha}}{J^{\Half}(1+\bar{W}_{\alpha})}\right]  = 0,\\
\end{aligned} 
\right.
\end{equation}
where $J := |1+W_\alpha|^2$ is the Jacobian, and
\begin{equation*}
 F = \nP\left[\frac{Q_\alpha - \bar Q_\alpha}{J}\right].
 \end{equation*}

The system \eqref{e:CWW} has a conserved energy
\begin{equation*}
\mathcal{E}(W, Q)=  \Re\int Q\bar{Q}_{\alpha}
+2\sigma (J^{\frac{1}{2}} - 1 -\Re W_\alpha) + g|W|^2(1+W_\alpha) \, d\alpha.
\end{equation*}
It also has a conserved horizontal momentum
\begin{equation*}
    \mathcal{P}(W,Q) = -i\int Q\bar{W}_\alpha-\bar{Q}W_\alpha \,d\alpha.
\end{equation*}
The linearization of \eqref{e:CWW} around the zero solution is
\begin{equation}
\left\{
             \begin{array}{lr}
             w_t + q_\alpha = 0 &  \\
             q_t-igw + i\sigma w_{\alpha\alpha} =0,&  
             \end{array}
\right.\label{e:ZeroLinear}
\end{equation}
restricted to holomorphic functions.
\eqref{e:ZeroLinear} can be written as a linear dispersive equation
\begin{equation*}
    w_{tt}+igw_\alpha- i\sigma w_{\alpha\alpha \alpha} =0.
\end{equation*}
Its dispersion relation is 
\begin{equation*}
    \tau^2 + g\xi + \sigma\xi^3 =0, \quad \xi\leq 0.
\end{equation*}
This simplified model suggests that the water waves \eqref{e:CWW} is a nonlinear dispersive system.

The conserved energy of \eqref{e:ZeroLinear} is given by
\begin{equation}
    \mathcal{E}_0(w,q) = \int \sigma|w_\alpha|^2 - iq\bar{q}_\alpha\, d\alpha = \sigma\|w\|_{\dot{H}^1}^2 + \|q\|_{\dot{H}^{\frac{1}{2}}}^2. \label{ConservedEnergy}
\end{equation}
The system \eqref{e:ZeroLinear} is well-posed in $ \dot{H}^1 \times \dot{H}^{\frac{1}{2}}$ space.
To study the regularity of nonlinear problem, for appropriate choices of $s$, we will then use the non-homogeneous product spaces $\mathcal{H}^{s}$ endowed with the norm
\begin{equation*}
    \| (w,q)\|_{\mathcal{H}^{s}} := \|(w,q) \|^2_{ H^{s+\frac{1}{2}}\times H^s}, \quad s\in \mathbb{R}.
\end{equation*}

Going back to the full water wave  system \eqref{e:CWW}, it is a nonlinear degenerate hyperbolic system.
By differentiation, it can be diagonalized and converted into a quasilinear system.
Following the setup in \cite{MR3667289}, we define the differentiated holomorphic unknown functions
\begin{equation*}
(\W, R): = \left(W_\alpha, \frac{Q_\alpha}{1+W_\alpha}\right).
\end{equation*}
The holomorphic function $R$ above has an intrinsic meaning; it is the complex velocity restricted on the water surface.
We also need to define three real-valued auxiliary functions.
The first one is the {\em frequency shift} $a$  given by
\begin{equation*}
a := i\left(\bar \nP \left[\bar{R} R_\alpha\right]- \nP\left[R\bar{R}_\alpha\right]\right).
\end{equation*}
The second one is the {\em advection velocity} $b$  given by 
\begin{equation*}
b := 2 \Re \nP\left[ \frac{Q_\alpha}{J}\right].
\end{equation*}
The third  auxiliary function $M$ is given by
\begin{equation} \label{DefM}
M :=  \frac{R_\alpha}{1+\bar \W}  + \frac{\bar R_\alpha}{1+ \W} -  b_\alpha =
\bar \nP [\bar R Y_\alpha- R_\alpha \bar Y]  + \nP[R \bar Y_\alpha - \bar R_\alpha Y].
\end{equation}
Here we use the notation $Y:=\dfrac{W_{\alpha}}{1+W_{\alpha}}$. 
Then, using above notations, the pair $(\W,R)$ diagonalizes the differentiated
system, and the differentiated water wave system can be written as
\begin{equation} \label{e:WR}
\left\{
\begin{aligned}
 & \W_{ t} + b \W_{ \alpha} + \frac{(1+\W) R_\alpha}{1+\bar \W}   =  (1+\W)M\\
& R_t + bR_\alpha  -i\frac{g\W - a}{1+\W} + \frac{ i\sigma}{1+\W}\nP\left[ \frac{\W_{ \alpha}}{J^\Half(1+\W)}\right]_{\alpha}  =  \frac{ i\sigma}{1+\W} \nP \left[\frac{\bar{\W}_{ \alpha}}{J^{\Half}(1+\bar{\W})}\right]_{\alpha}.
\end{aligned}
\right.
\end{equation}
In the following, we will be mostly working with the differentiated system \eqref{e:WR}.

The system \eqref{e:WR} does not have a scaling that keeps both the parameters $g$ and $\sigma$ unchanged. 
The scaling
\begin{equation}
\left(\W(t,\alpha), R(t,\alpha)\right) \to \left( \W(\lambda^\frac32 t,\lambda \alpha), 
\lambda^{\frac12} R(\lambda^\frac32 t,\lambda \alpha)\right)  \label{Scaling}
\end{equation}
keeps $\sigma$, but $g$ becomes $\frac{g}{\lambda^{2}}$.
This scaling suggests that the critical space of \eqref{e:WR} is $\dot{H}^{\Half}\times L^2$, and one may hope to prove the well-posedness for \eqref{e:WR} in $\mathcal{H}^s$ for $s>0$.
In the rest of this article, we will write $\sigma = 1$  for simplicity.

\subsection{Some historical results on 2D capillary water waves}
The literature of the two-dimensional capillary water waves is numerous.
Here we only mention some of the results on the local well-posedness.
Beginning with the works of Nalimov \cite{MR0609882} and Yosihara \cite{MR0660822, MR0728155} for small smooth data, there are also many results including the work of Agrawal \cite{MR4244258, MR4684336}, Ambrose-Masmoudi \cite{MR2162781}, Beyer-Gunther \cite{MR1637554},  Coutand-Shkoller \cite{MR2291920}, Iguchi \cite{MR1865389}, Lannes \cite{MR3060183},  Ming-Zhang \cite{MR2558419} and Shatah-Zeng \cite{MR2763036}.

Let us mention the work of Alazard-Burq-Zuily \cite{MR2805065} in the local well-posedness.
They worked in Zakharov-Craig-Sulem formulation, and used paradifferential calculus to reduce the water wave system to the following form:
\begin{equation*}
    \partial_t u + T_V \nabla u + iT_\gamma u = f.
\end{equation*}
As a result, they proved that  \eqref{e:WR} is locally well-posed  in $\H^s$, for $s> \frac{3}{2}$.

De Poyferr\'{e} and Nguyen followed the strategy in \cite{MR2805065} and derive a paradifferential reduction for water waves in \cite{MR3770970}.
They combine the energy estimate and the Strichartz estimate to show that \eqref{e:WR} is locally well-posed  in $\H^s(\mathbb{R})$, for $s> \frac{27}{20}$ in \cite{MR3487264}.

Nguyen in \cite{MR3724757} further used  an improved para-composition together with the energy and Strichartz estimates to show the following well-posedness result. 
\begin{theorem} \label{t:Previous}
The system \eqref{e:WR} is locally well-posed in $\H^s(\mathbb{R})$, for $s> \frac{5}{4}$.
\end{theorem}
To the author's knowledge, this is the lowest regularity result before.
See also the  works of Alazard-Burq-Zuily \cite{MR2931520}, de Poyferr\'{e}-Nguyen \cite{MR3487264}, Nguyen \cite{MR3864396} and Ai \cite{ai2023improved} for Strichartz estimate and its application in the local well-posedness of water waves.

The global well-posedness of two-dimensional gravity-capillary water waves is an open question.
In the pure capillary case, it was shown that the solution is globally well-posed for small localized initial data by Ifrim-Tataru \cite{MR3667289}. Shortly after  Ionescu-Pusateri \cite{MR3862598} provided an alternative proof of a similar result.
In the periodic setting, Berti and Delort showed in \cite{MR3839269} for almost all gravity-capillarity parameters, even, small and smooth enough initial data produces almost globally defined in time solutions.
Then Berti-Feola-Franzoi \cite{MR4246389} showed that the periodic gravity-capillary water waves have quadratic lifespans.

The use of holomorphic coordinates for two-dimensional water waves  originates in the early work on traveling waves \cite{MR1512238}.
Since then, it has been widely used to study a variety of water wave problems including \cite{MR3535894, MR3499085, MR3625189,  ai2023dimensional, MR4483135, MR4462478, rowan2023dimensional, rowan2024, wan2023low}.
In this article, the main system \eqref{e:WR} was first introduced in \cite{MR3667289} by Ifrim and Tataru.
The authors of \cite{MR3667289} not only show that the solution of pure capillary water waves has a cubic lifespan but also prove that small localized initial data of \eqref{e:CWW} leads to global solutions. 

One of the main ingredients in the proof of \cite{MR3667289} is the modified energy estimate.
The authors in  \cite{MR3667289} used the normal form transformation to construct the modified  energy that enables them to establish the modified energy estimate. However the main inspiration in this paper comes from the work of Ai, Ifrim and Tatatru \cite{ai2023dimensional} on pure gravity waves where they developed a new, better class of estimates called \emph{balanced  cubic energy estimates} relying on paradifferential form of the equations.  
In this article,  following \cite{MR3667289}, we will compute the normal form transformation at the paradifferential level for \eqref{e:WR}, to obtain the corresponding paradifferential energy estimate for the  capillary water wave system.

\subsection{The main results}
In this article, we use the paradifferential modified energy estimate to improve the well-posedness result Theorem \ref{t:Previous}.

To state the energy estimate result, we  define the control norm that will be used in the estimates.
Let $\epsilon>0$ be an arbitrarily small positive constant.
\begin{align*}
\CalAZ : = \|\W \|_{ C^{\epsilon}_{*}}+ \|R \|_{ C^{\epsilon}_{*}}, \quad 
\CalAO : = \|\W \|_{ C^{1+\epsilon}_{*}}+\|R\|_{C_{*}^{\frac{1}{2}+\epsilon}}.
\end{align*}
Here $C^s_{*}$ is the Zygmund space that will be defined in Section \ref{s:Norms}. 
For reference, we also define a larger control norm
\begin{equation*}
\CalAT : =  \|\W \|_{C_{*}^{\frac{3}{2}}} + \|R \|_{C_{*}^{1+\epsilon}}.
\end{equation*}
The subscript $s$ of control norm $\mathcal{A}_s$ represents roughly the difference in terms of derivatives between the
control norm and the scaling. 

\begin{theorem} \label{t:MainEnergyEstimate}
 For any $s>1$, there exists an energy functional $E_s$ which has the following properties:
 \begin{enumerate}
     \item Norm equivalence:
       \begin{equation}
           E_s(\mathbf{W},R) \approx_\CalAZ \|(\W,  R)\|^2_{\mathcal{H}^s}.\label{normEquivalence}
       \end{equation}
     \item Modified energy estimate:
      \begin{equation}
     \frac{d}{dt}E_s\lesssim_\CalAZ (1+\mathcal{A}_1^2) E_s.\label{strongcubicest}
 \end{equation} 
 \end{enumerate} 
\end{theorem}
For comparison, in \cite{MR3770970}, the energy estimate is the tame energy estimate
\begin{equation*}
    \frac{d}{dt}\|(\W, R) \|^2_{\mathcal{H}^s} \lesssim_\CalAO \left(1+\CalAT \right) \|(\W, R) \|^2_{\mathcal{H}^s}. 
\end{equation*} 
It involves a larger control norm $\CalAT$.
As a consequence, a higher regularity well-posedness result is proved in \cite{MR3724757}.

To advance the study of longtime properties of the flow Ifrim and Tataru \cite{MR3667289} proved a better cubic modified energy estimate of the form
\begin{equation*}
    \frac{d}{dt}\|(\W, R) \|^2_{\mathcal{H}^s} \lesssim_\CalAO \mathcal{A}_0\CalAT \|(\W, R) \|^2_{\mathcal{H}^s},
\end{equation*} 
also with the $\epsilon$'s removed from the above definition of the control norms $\mathcal{A}_j$'s. Following the idea of \emph{balanced cubic energy estimates} of Ai, Ifrim and Tataru \cite{ai2023dimensional} in our estimate \eqref{strongcubicest} we balance better the two control norms in the last estimate.

For the linearized system, one can also prove a similar linearized modified energy estimate.
This will be discussed in detail in Section \ref{s:LinearEstimate}.

Having established the energy estimates for the full system and the linearized system, we can get the following local well-posedness result:
\begin{theorem} \label{t:MainWellPosed}
Let $s> 1$, the system \eqref{e:WR} is locally well-posed for initial data $(\W_0, R_0 )$ in $\H^s(\R)$  (or $\mathbb{T}$) such that $\mathcal{A}_0(\W_0, R_0)$ is small.
Furthermore, the solution can be continued for as long as $\mathcal{A}_0$ remains bounded and $(1+\mathcal{A}_1^2) \in L^1_t$.
\end{theorem}
Here by local well-posedness, we mean that there exists a positive time $T>0$ such that the following properties hold:
\begin{enumerate}
\item \textit{Existence of the solution}: If the initial data is in $\mathcal{H}^s$, then there exists a solution $(\W, R)\in \mathcal{H}^s$ in $[0,T]$, such that
\begin{equation*}
    \|(\W, R)\|_{C[0,T; \mathcal{H}^s]} \lesssim \|(\W_0, R_0)\|_{C[0,T; \mathcal{H}^s]}.
\end{equation*}
\item \textit{Uniqueness of the solution}: When $s>\frac{3}{2}$, it is unique as is proved in \cite{MR2805065}.
When $\frac{3}{2}\geq s>1$,  the solution is unique in the sense that it is the unique  limit of regular solutions.

\item \textit{Continuous dependence on the  data for  solutions}: If a sequence of initial data $(\W_j, R_j)(0)$  converges to $(\W_0, R_0)$ in $\mathcal{H}^{s}$ topology, then the solutions $(\W_j, R_j)(t)$ also converges to $(\W, R)(t)$ in $\mathcal{H}^{s}$, for $t\in [0,T]$.
\end{enumerate}  

Our proof of the local well-posedness does not rely on the dispersive properties of the water wave system, such as the Strichartz estimate.
The result holds for both the real line and periodic cases. 
In the case of the real line, it achieves a lower $\frac{1}{4}$ Sobolev regularity threshold compared to the result in \cite{MR3724757}.
It also obtains a lower $\frac{1}{2}$ Sobolev regularity threshold compared to the result \cite{MR2805065} in the periodic case.
For simplicity, we will only discuss the case of the real line in this article.
For the periodic setting, the analysis is essentially the same; we refer the reader to the discussion in Appendix $A$ of \cite{MR3535894} for the minor changes.

Compared to the previous result of \cite{MR3667289}, which also uses the modified energy estimate, our result has the following improvements:
\begin{itemize}
\item We express the water waves as a paradifferential system and utilize the paradifferential structure of the equations, akin to the work in \cite{ai2023dimensional} for the case of gravity waves. 
Our paradifferential modified energy estimate no longer depends on the larger control norm $\CalAT$; it only depends on a smaller control norm $\CalAO$.
This effectively lowers the regularity for well-posedness.
\item The normal form analysis is also at the paradifferential level.
As a result, we can allow the Sobolev index $s$ to be non-integer, whereas the energy estimates of \cite{MR3667289} only use integer-valued Sobolev index.
\item Our energy estimate for the linearized system Theorem \ref{t:LinearizedWellposed} is  also more balanced with implicit constants depending only on the control norms $\mathcal{A}_1$ refined, which is better than the energy estimate for the linearized system in \cite{MR3667289}.
We prove a better estimate for the linearized system, and it depends on the  control norm $\CalAO$.  
\item We take into account the effect of gravity and work in the non-homogeneous Sobolev spaces.
The result in \cite{MR3667289} is only on the pure-capillary water waves. 
\end{itemize}

\subsection{The organization of the article}
Throughout the analysis of the low regularity well-posedness, paradifferential calculus and paraproducts type estimates for symbols are frequently used.
The full water wave system and the linearized system will also be rewritten in the paradifferential format.
In Appendix \ref{s:Norms}, we gather the necessary paradifferential estimates that will be needed in the main part of the article.

In Section \ref{s:Bounds}, we  apply the paraproduct estimates that were developed in \cite{ai2023dimensional} to derive Sobolev and $BMO$ estimates for auxiliary functions $a, b, Y$, and so on.
We also compute the leading terms of para-material derivatives of $\W, R, Y$ and $J^s$.

Moving on to Section \ref{s:ModifiedEnergy}, we prove the paradifferential modified energy estimate Theorem \ref{t:MainEnergyEstimate}.
We write the water waves \eqref{e:WR} as a system of paradifferential equations and compute balanced quadratic paradifferential normal forms to remove the balanced quadratic source terms.
Low-high normal form transformations are used to construct the cubic paradifferential modified energy.
In addition, we also construct the quartic paradifferential modified energy.

After obtaining the energy estimate for the full water wave system, we establish the energy estimate for the linearized water wave system in Section \ref{s:LinearEstimate}.
We show that the linearized water wave system \eqref{linearizedeqn} is locally well-posed in $\mathcal{H}^{\Half}$ by writing the linearized system \eqref{linearizedeqn} as a paradifferential system \eqref{ParadifferentialLinearEqn} and constructing the paradifferential modified linearized energy.
The construction approach will be discussed in detail  within this section.

Finally, gathering the paradifferential modified energy estimates for both the full and the linearized systems, we  outline the proof for the local well-posedness of the water wave system in Section \ref{s:Wellposed}. \\

\textbf{Acknowledgments.} The author would like to thank Mihaela Ifrim  for introduction and discussion of many important details in her related work \cite{ai2023dimensional,MR3667289, MR3535894}.

\section{Water waves related bounds} \label{s:Bounds}
In the first part of this section, we  consider estimates  related to the water wave system, following the lead of \cite{MR3535894,ai2023dimensional}.
We will first give some estimates on the frequency shift $a$ and the advection velocity $b$.
Recall that they are given by
\begin{equation*}
a = \Im \nP[R\bar{R}_\alpha], \qquad b = 2\Re \nP[(1-\bar{Y})R].
\end{equation*}
Next, we consider the estimates of three auxiliary functions, they are
\begin{equation*}
Y = \frac{\W}{1+\W}, \quad M = 2\Re \nP[R\bar{Y}_\alpha - \bar{R}_\alpha Y], \quad c= 2\Im J^{-\Half}(1-Y)\W_\alpha.
\end{equation*}
Then we give an estimate for the operator $L$, which is a self-adjoint operator defined by 
\begin{equation} \label{LwDef}
\begin{aligned}
L &= \partial_\alpha J^{-\frac{1}{2}}\partial_\alpha -ic\partial_\alpha -i\nP c_\alpha\\
&= \partial_\alpha J^{-\frac{1}{2}}\partial_\alpha -i\left(c\partial_\alpha + \frac{1}{2}c_\alpha \right)-\frac{i}{2}(\nP c_\alpha- \bar{\nP}c_\alpha). 
\end{aligned}    
\end{equation}
The auxiliary function $Y$ is often used as a part of  para-coefficients in many of our computations. 
$c$ and $L$ will appear in Section \ref{s:LinearEstimate} when we compute the linearized water waves system.

In the second part of this section, we compute the leading terms of para-material derivatives of $\W, R, Y$, $J^{s}$, and also the leading term of the commutator $[T_{D_t}, \mathcal{L}]$, where $\mathcal{L}$ is the leading paradifferential part of the operator $L$.

\subsection{Sobolev and Zygmund bounds}
We begin with the bounds for the frequency shift $a$, which are similar to those derived in \cite{ai2023dimensional}. 
\begin{lemma}
The  frequency-shift $a$ satisfies the estimate

\begin{equation}
\|a\|_{C_{*}^{\epsilon}} \lesssim \mathcal{A}_1^2, \label{ACHalf}
\end{equation}
as well as the Sobolev estimate
\begin{equation}
\| a\|_{H^{s}}  \lesssim \CalAO \|(\W,R)\|_{\mathcal{H}^{s+\frac{1}{2}}}, \quad s>0. \label{AHsEst}  
\end{equation}
\end{lemma}
\begin{proof}
For the holomorphic part of $a$, we write
\begin{equation*}
\nP[R\bar{R}_\alpha] = T_{\bar{R}_\alpha} R + \nP\Pi(\bar{R}_\alpha, R).
\end{equation*}
Using \eqref{CCCEstimate} and \eqref{CsLInfty}, 
\begin{equation*}
\|T_{\bar{R}_\alpha} R \|_{C_{*}^{\epsilon}}+ \|\nP\Pi(\bar{R}_\alpha, R) \|_{C_{*}^{\epsilon}} \lesssim \|R \|_{C_{*}^{\frac{1}{2}+\epsilon}}^2 \lesssim \mathcal{A}_1^2,
\end{equation*}
which gives the bound \eqref{ACHalf} for the frequency-shift $a$.
For the Sobolev estimate, we use \eqref{HCHEstimate}, \eqref{HsCmStar} to write
\begin{equation*}
\|T_{\bar{R}_\alpha} R \|_{H^{s}}+ \|\nP\Pi(\bar{R}_\alpha, R) \|_{H^s} \lesssim \|R \|_{C_{*}^{\Half}} \| R\|_{H^{s+\Half}} \lesssim \CalAO \| R\|_{H^{s+\Half}},    
\end{equation*}
which gives the Sobolev bound \eqref{AHsEst} for the frequency-shift $a$.
\end{proof}

We continue with the advection velocity $b$, where we again follow \cite{ai2023dimensional}.
\begin{lemma}
The advection velocity $b$ satisfies the estimate
\begin{equation}
\| b\|_{C^{\Half}_{*}} \lesssim_\CalAZ \CalAO, \label{BCOneStar}
\end{equation}
as well as the Sobolev estimate
\begin{equation}
\| b\|_{H^s}  \lesssim_{\CalAZ} \|R\|_{H^{s}}, \quad s>0. \label{BHsEst}  
\end{equation}
\end{lemma}
\begin{proof}
The bounds for $\Re R$ are obvious, it suffices to estimate the $\nP[R\bar{Y}]$ term.
Note that, 
\begin{equation*}
\nP[R\bar{Y}] = T_{\bar{Y}}R + \nP \Pi(R,\bar{Y}). 
\end{equation*}
Using the estimates \eqref{CCCEstimate} and \eqref{CsLInfty},
\begin{equation*}
\|\nP[R\bar{Y}] \|_{C^\Half_{*}} \leq \|T_{\bar{Y}}R \|_{C^\Half_{*}} + \|\nP \Pi(R,\bar{Y}) \|_{C^\Half_{*}} \lesssim \|Y\|_{ C^{\epsilon}_{*}}\|R\|_{C^\Half_{*}}  \lesssim \|\W\|_{ C^{\epsilon}_{*}}\|R\|_{C^\Half_{*}}\lesssim_\CalAZ \CalAO,
\end{equation*}
which gives the $C^\Half_{*}$ estimate \eqref{BCOneStar}.
For the Sobolev estimate, we use \eqref{HsLinfty} and \eqref{HCHEstimate},
\begin{equation*}
\|\nP[R\bar{Y}] \|_{H^s} \leq \|T_{\bar{Y}}R \|_{H^s} + \|\nP \Pi(R,\bar{Y}) \|_{H^s} \lesssim \|Y\|_{C^\epsilon_{*}}\|R\|_{H^s}  \lesssim_\CalAZ \|R\|_{H^s},
\end{equation*}
which gives the Sobolev bound \eqref{BHsEst}.
\end{proof}

Next, we consider the estimates for the auxiliary function $Y$, also in the spirit of \cite{ai2023dimensional}.
\begin{lemma}
For $s>0$, the auxiliary function $Y = \frac{\W}{1+\W}$ satisfies
\begin{equation}
\|Y\|_{H^s} \lesssim_\CalAZ \|\W\|_{H^s}, \quad \|Y\|_{C^s_{*}} \lesssim_\CalAZ \|\W\|_{C^s_{*}}. \label{YMoser}
\end{equation}
In particular, $\|Y\|_{C^1_{*}} \lesssim_\CalAZ \CalAO$.
Moreover, one can write
\begin{equation}
    Y = T_{(1-Y)^2}\W + E, \label{YWExpression}
\end{equation}
where the error $E$ satisfies the bound
\begin{equation*}
\|E\|_{H^{s-1}} \lesssim_\CalAZ \CalAO \|\W\|_{H^{s}}.
\end{equation*}
\end{lemma}
The proof of \eqref{YMoser} is a direct consequence of Moser type estimates \eqref{MoserOne} and \eqref{MoserTwo}. 
The proof of \eqref{YWExpression} follows from the paralinearization Lemma \ref{t:Paralinear}.

Finally, we derive estimates for auxiliary functions $M$, $c$, and the operator $L$.
\begin{lemma}
The auxiliary function $M$ satisfies the Zygmund bound
\begin{equation}
\| M\|_{C^\Half_{*}} \lesssim \mathcal{A}^2_1. \label{MBound}
\end{equation}
\end{lemma}
One should compare this with the $L^{\infty}$ bound for $M$ in \cite{ai2023dimensional}.
\begin{proof}
Recall that $M$ can be rewritten as
\begin{equation*}
  M = \bar{\nP}[\bar{R}Y_\alpha - R_\alpha \bar{Y}] + \nP[R\bar{Y}_\alpha - \bar{R}_\alpha Y],  
\end{equation*}
and it suffices to consider the holomorphic part.
We estimate
\begin{align*}
&\|\nP[R \bar{Y}_\alpha] \|_{C^\Half_{*}} \leq \|T_{\bar{Y}_\alpha}R \|_{C^\Half_{*}} + \|\nP(\bar{Y}_\alpha, R) \|_{C^\Half_{*}}\lesssim \| Y_\alpha\|_{C^\epsilon_{*}}\|R \|_{C^\Half_{*}}  \lesssim \mathcal{A}^2_1, \\
& \|\nP[\bar{R}_\alpha Y] \|_{C^\Half_{*}} \leq \|T_{\bar{R}_\alpha}Y \|_{C^\Half_{*}} + \|\nP(\bar{R}_\alpha, Y) \|_{C^\Half_{*}}\lesssim \| R_\alpha\|_{C^{-\frac{1}{2}}_{*}}\| Y\|_{C^1_{*}} \lesssim \mathcal{A}^2_1.
\end{align*}
These give the estimate for $M$ \eqref{MBound}.
\end{proof}

Recall that the real-valued function $c$ is defined by
\begin{equation}
c:= -\dfrac{i\W_\alpha}{J^{\frac{1}{2}}(1+\W)}+\dfrac{i\bar{\W}_\alpha}{J^{\frac{1}{2}}(1+\bar{\W})}. \label{CDef}
\end{equation}
Then we have the following estimate for $c$.
\begin{lemma}
The auxiliary function $c$ satisfies the Zygmund bound
\begin{equation}
\| c\|_{C^{\epsilon}_{*}} \lesssim_\CalAZ \CalAO. \label{CBound}
\end{equation}
Moreover, we have the following representation for $\nP c$:
\begin{equation}
    \nP c = -iT_{J^{-\Half} (1-Y)}\W_\alpha + K, \label{nPCRepresentation}
\end{equation}
where $K$ satisfies the estimate
\begin{equation*}
    \| K\|_{C_{*}^{1+\epsilon}} \lesssim_\CalAZ \mathcal{A}^2_1.
\end{equation*}
\end{lemma}
\begin{proof}
By the definition \eqref{CDef}, it suffices to estimate $\frac{\W_\alpha}{J^\Half (1+\W)}$. 
Its complex conjugate satisfies the similar estimate.
We decompose
\begin{equation*}
 \frac{\W_\alpha}{J^\Half (1+\W)} = T_{J^{-\Half} (1-Y)}\W_\alpha + T_{\W_\alpha} (J^{-\Half} (1-Y) -1) + \Pi(\W_\alpha, J^{-\Half} (1-Y) -1).  
\end{equation*}
Since $J^{-\Half} (1-Y) -1$ is a function of $Y$, $\bar{Y}$ and vanishes when $Y = 0$, by the Moser type estimate \eqref{MoserTwo},
\begin{equation*}
  \|J^{-\Half} (1-Y)-1 \|_{C^\epsilon_{*}}\lesssim \| Y\|_{C^\epsilon_{*}} \lesssim \CalAZ.  
\end{equation*}
Hence, we can estimate
\begin{align*}
&\|T_{J^{-\Half} (1-Y)}\W_\alpha \|_{C^\epsilon_{*}} \lesssim (1+\|J^{-\Half} (1-Y)\|_{C^\epsilon_{*}}) \| \W_\alpha\|_{C^\epsilon_{*}} \lesssim (1+\CalAZ) \CalAO, \\
&\|T_{\W_\alpha} (J^{-\Half} (1-Y) -1) \|_{C^\epsilon_{*}}  \lesssim \|J^{-\Half} (1-Y)-1 \|_{C^\epsilon_{*}} \| \W_\alpha\|_{C^\epsilon_{*}} \lesssim \CalAZ \CalAO, \\
 &\|\Pi(\W_\alpha, J^{-\Half} (1-Y) -1) \|_{C^\epsilon_{*}} \lesssim \|J^{-\Half} (1-Y)-1 \|_{C^\epsilon_{*}} \| \W_\alpha\|_{C^\epsilon_{*}} \lesssim \CalAZ \CalAO,
\end{align*}
and this leads to the Zygmund bound \eqref{CBound}.

To prove the representation for $\nP c$, we claim that
\begin{equation}
J^{-\Half}(1-Y) \W_{ \alpha} =  T_{J^{-\frac{1}{2}}(1-Y)}\W_{\alpha} +K,   \label{CapillaryTerm}
\end{equation}
Indeed, the factor $J^{-\frac{1}{2}}(1-Y)$ is a smooth function of $\W$ and $\bar{\W}$ that equals $1$ when $\W =0$. 
We can apply Moser type of estimate \eqref{MoserTwo}  to conclude that
\begin{align*}
 &\|T_{\W_\alpha}J^{-\frac{1}{2}}(1-Y)-1\|_{C_{*}^{1+\epsilon}} + \|\Pi(\W_\alpha, J^{-\frac{1}{2}}(1-Y)-1)\|_{C_{*}^{1+\epsilon}} \\
 \lesssim &\|\W_\alpha \|_{C_{*}^{\epsilon}}\|J^{-\frac{1}{2}}(1-Y)-1 \|_{C_{*}^{1+\epsilon}}\lesssim_\CalAZ \mathcal{A}^2_1,   
\end{align*}
so that the high-low and the balanced terms can be placed into $K$.
Doing the same analysis for the complex conjugate of $J^{-\Half}(1-Y) \W_{ \alpha}$,
\begin{equation*}
    c = -iT_{J^{-\Half} (1-Y)}\W_\alpha + iT_{J^{-\Half} (1-\bar{Y})}\bar{\W}_\alpha + K.
\end{equation*}
After applying the projection $\nP$ to eliminate the antiholomorphic term, we get the representation for $\nP c$ \eqref{nPCRepresentation}.
\end{proof}

For the operator $L$, we compute the leading term of $Lw$.
\begin{lemma} \label{t:LwLeading}
The leading term of $Lw$ is given by
\begin{equation*}
 Lw = T_{J^{-\Half}}w_{\alpha \alpha}  - T_w  T_{J^{-\Half}(1-Y)}\W_{\alpha \alpha} + E,
\end{equation*}
where the remainder term $E$ satisfies
\begin{equation*}
 \|E\|_{L^2} \lesssim_\CalAZ \CalAO \|w\|_{H^1}.
\end{equation*}
\end{lemma}

\begin{proof}
For the first term of $Lw$, we write
\begin{equation*}
    \partial_\alpha(J^{-\Half}w_\alpha) = J^{-\Half}w_{\alpha \alpha} + (J^{-\Half})_\alpha w_\alpha =  T_{J^{-\Half}}w_{\alpha \alpha} + E.
\end{equation*}
For the second term of $Lw$, we apply the bound for $c$ \eqref{CBound} to write
\begin{equation*}
    \|ic w_\alpha \|_{L^2} \lesssim \|c\|_{L^\infty} \|w_\alpha \|_{L^2} \lesssim \CalAO \|w\|_{H^1},
\end{equation*}
so that this term can be put into the remainder term $E$.

As for the third term of $Lw$, by \eqref{nPCRepresentation}, 
\begin{equation*}
    -i \nP c_\alpha = -\partial_\alpha T_{J^{-\Half}(1-Y)}\W_\alpha + K_1 = - T_{J^{-\Half}(1-Y)}\W_{\alpha \alpha} + K_1, \quad \|K_1\|_{C^\epsilon_{*}}\lesssim_\CalAZ \mathcal{A}_1.
\end{equation*}
Hence, we can write for the last term
\begin{equation*}
-i \nP c_\alpha w = - T_{J^{-\Half}(1-Y)}\W_{\alpha \alpha}w + E = - T_w  T_{J^{-\Half}(1-Y)}\W_{\alpha \alpha} + E.
\end{equation*}
Putting these three terms of $Lw$ together, we obtain the leading term of $Lw$.
\end{proof}

\subsection{Leading terms of para-material derivatives} 
The material derivative $D_t = \partial_t + b \partial_\alpha$ is very important in the water wave system.
At the paradifferential level, it is replaced by the para-material derivative $T_{D_t} = \partial_t + T_b \partial_\alpha$, see \cite{ai2023dimensional}.
In the following, we compute the para-material derivatives of  various functions. This are similar to the methodology developed in \cite{ai2023dimensional} with the difference that we estimate errors using  the control norm $\mathcal{A}_1$, as opposed to $\mathcal{A}_{\frac{1}{2}}$.

We first consider the expressions for para-material derivatives of $W$,  $\W$ and $Y$.

\begin{lemma} 
Let $s>0$.
\begin{enumerate}
\item The unknown $W$ satisfies the paradifferential equation
\begin{equation}
  T_{D_t}W = -T_{1+W_\alpha} \nP[(1-\bar{Y})R]-\nP\Pi(W_\alpha, b). \label{ParaW}
\end{equation}
\item The leading term of the para-material derivative of $W$ is given by
\begin{equation}
      T_{D_t} W +T_{1+\W}T_{1-\bar{Y}}R = G,  \label{WParaMaterial}
\end{equation}
where the source term $G$  satisfies the Sobolev bound
\begin{equation*}
    \|G\|_{H^s} \lesssim_\CalAZ \CalAO \|(\W,R)\|_{\mathcal{H}^{s-1}}, \quad \| G\|_{C_{*}^{\frac{3}{2}}} \lesssim_\CalAZ \mathcal{A}^2_1.
\end{equation*}
\item The unknown $\W$ satisfies the paradifferential equation
\begin{equation}
       T_{D_t}\W  + T_{(1+\W)(1-\bar{Y})}R_\alpha = G_1,  \label{WParaMat}
\end{equation}
where the source term $G_1$  satisfies bounds
\begin{equation*}
    \|G_1\|_{H^s} \lesssim_\CalAZ \CalAO \|(\W,R)\|_{\mathcal{H}^{s}}, \quad \|G_1\|_{C_{*}^\Half} \lesssim_\CalAZ \mathcal{A}^2_1.
\end{equation*}
\end{enumerate}
\end{lemma}

\begin{proof}
The first equation of \eqref{e:CWW} can be rewritten as
\begin{equation*}
 W_t + b(1+ W_\alpha)  = \bar{R}, 
\end{equation*}
so that it becomes the paradifferential equation
 \begin{equation*}
     W_t + T_b \partial_\alpha W = -T_{1+W_\alpha} b -\Pi(W_\alpha,b)+ \bar{R}.
 \end{equation*}
 After plugging in the expression of $b$ and applying the holomorphic projection $\nP$, we can eliminate the anti-holomorphic portion and obtain the paradifferential equation \eqref{ParaW}.   

 For the leading term of the para-material derivative of $W$, we write
\begin{equation}
T_{D_t }W  +T_{1+\W}T_{1-\bar{Y}}R = T_{1+\W}\nP\Pi(\bar{Y}, R) - \nP\Pi(\W, b) .      \label{WFormula}
\end{equation}
To estimate the right-hand side of \eqref{WFormula}, we compute
\begin{align*}
&\|T_{1+\W}\nP\Pi(\bar{Y}, R) \|_{H^s} \lesssim_\CalAZ \|\Pi(\bar{Y},R)\|_{H^s} \lesssim_\CalAZ \|Y \|_{C^{1}_{*}} \|R \|_{H^{s-1}}\lesssim_\CalAZ \CalAO \|R \|_{H^{s-1}}, \\
&\|\nP\Pi(\W, b) \|_{H^s} \lesssim \| \W\|_{H^{s-\frac{1}{2}}}\|b\|_{C^{\frac{1}{2}}_{*}} \lesssim_\CalAZ \CalAO \| \W\|_{H^{s-\frac{1}{2}}}, \\
&\|T_{1+\W}\nP\Pi(\bar{Y}, R) \|_{C_{*}^{\frac{3}{2}}} \lesssim_\CalAZ \|\Pi(\bar{Y},R)\|_{C_{*}^{\frac{3}{2}}} \lesssim_\CalAZ \|Y \|_{C^{1}_{*}} \|R \|_{C_{*}^{\frac{1}{2}}}\lesssim_\CalAZ \mathcal{A}^2_1, \\
&\|\nP\Pi(\W, b) \|_{C_{*}^{\frac{3}{2}}} \lesssim \| \W\|_{C_{*}^{1}}\|b\|_{C^{\frac{1}{2}}_{*}} \lesssim_\CalAZ \mathcal{A}^2_1.
\end{align*}
Hence, the right-hand side of \eqref{WFormula} can be absorbed into the source term $G$.

Differentiating the equation \eqref{WParaMaterial}, one gets
\begin{equation*}
T_{D_t}\W + T_{(1+\W)(1-\bar{Y})} R_\alpha =  (T_{(1+\W)(1-\bar{Y})}-T_{1+\W}T_{1-\bar{Y}})R_\alpha -T_{\W_\alpha}T_{1-\bar{Y}}R+ T_{1+\W}T_{\bar{Y}_\alpha}R -T_{b_\alpha}\W   + \partial_\alpha G.
\end{equation*}
We use \eqref{ParaProducts}, \eqref{HsCmStar}, and \eqref{BCOneStar} to estimate the Sobolev bounds of the right-hand side of above equation,
\begin{align*}
&\| (T_{1+\W}T_{1-\bar{Y}}-T_{(1+\W)(1-\bar{Y})})R_\alpha \|_{H^s} \lesssim_\CalAZ \CalAO \|R\|_{H^{s}},\\
&\|T_{\W_\alpha}T_{1-\bar{Y}}R \|_{H^s} + \|T_{1+\W}T_{\bar{Y}_\alpha}R \|_{H^s} \lesssim_\CalAZ \CalAO \|R \|_{H^s},\\
&\|T_{b_\alpha}\W \|_{H^s}\lesssim \|b\|_{C^{\Half}_{*}}\| \W\|_{H^{s+\Half}} \lesssim_\CalAZ \CalAO \| \W\|_{H^{s+\Half}},\\
&\|\partial_\alpha G \|_{H^s}\leq \|G \|_{H^{s+1}} \lesssim_\CalAZ \CalAO \| (\W,R)\|_{\mathcal{H}^{s}}.
\end{align*}

As for the Zygmund bounds, we use \eqref{CompositionTwo}, \eqref{BCOneStar} and \eqref{CsCmStar} to write
\begin{align*}
&\| (T_{1+\W}T_{1-\bar{Y}}-T_{(1+\W)(1-\bar{Y})})R_\alpha \|_{C_{*}^\Half} \lesssim_\CalAZ \CalAO \|R\|_{C_{*}^\Half}  \lesssim_\CalAZ \mathcal{A}^2_1,\\
&\|T_{\W_\alpha}T_{1-\bar{Y}}R \|_{C_{*}^\Half} + \|T_{1+\W}T_{\bar{Y}_\alpha}R \|_{C_{*}^\Half} \lesssim_\CalAZ \CalAO \|R \|_{C_{*}^\Half}  \lesssim_\CalAZ \mathcal{A}^2_1,\\
&\|T_{b_\alpha}W \|_{C_{*}^\Half}\lesssim \|b_\alpha\|_{C^{-1}_{*}}\| W\|_{C_{*}^{\frac{3}{2}}}\lesssim \|b\|_{C^{\Half}_{*}}\| \W\|_{C_{*}^\Half} \lesssim_\CalAZ \CalAO \| \W\|_{C_{*}^\Half} \lesssim_\CalAZ \mathcal{A}^2_1,\\
&\|\partial_\alpha G \|_{C_{*}^\Half}\leq \|G \|_{C_{*}^{\frac{3}{2}}} \lesssim_\CalAZ \mathcal{A}^2_1.
\end{align*}
Therefore, these terms can be absorbed into the source term $G_1$.
\end{proof}

\begin{lemma}
 Leading term of the para-material derivative of $Y$ is:
\begin{equation*}
    T_{D_t} Y = -T_{|1-Y|^2} R_\alpha + G_2 =: G_3 +G_2, 
\end{equation*}
where $G_3$ and $G_2$ satisfy the  Zygmund bounds
\begin{equation*}
  \|G_3\|_{C^{-\frac{1}{2}}_{*}} \lesssim_\CalAZ \CalAO, \quad  \|G_2\|_{C^{\frac{1}{2}}_{*}}\lesssim_\CalAZ \mathcal{A}^2_1.
\end{equation*}
\end{lemma}

\begin{proof}
 Observe that the first equation of \eqref{e:WR} can be expressed using $Y$ and $R$ as
\begin{equation*}
D_t Y + |1-Y|^2 R_\alpha = (1-Y)M. 
\end{equation*}
It can be rewritten as the paradifferential equation
\begin{equation*}
    T_{D_t}Y + T_{|1-Y|^2} R_\alpha = -T_{R_\alpha} (|1-Y|^2-1) - \Pi(R_\alpha, |1-Y|^2-1) - T_{Y_\alpha} b - \Pi(Y_\alpha, b) + (1-Y)M.
\end{equation*}
Using the bounds for the auxiliary functions $Y$ \eqref{YMoser}, $b$ \eqref{BCOneStar} and $M$ \eqref{MBound}, the right-hand side can be bounded by $\mathcal{A}^2_1$ in $C^{\frac{1}{2}}_{*}$ space, and these terms can be put into $G_2$.
The term $-T_{|1-Y|^2}R_\alpha$ can be put into $G_3$.
\end{proof}

We continue with the para-material derivative of $R$.
\begin{lemma}
The leading term of the para-material derivatives is given by
\begin{equation}
T_{D_t}R =  -iT_{J^{-\frac{1}{2}}(1-Y)^2}\W_{\alpha\alpha} +K, \label{RParaMat}
\end{equation}

where the error $K$ satisfies the estimate
\begin{equation*}
 \| K\|_{C_{*}^{\epsilon}} \lesssim_\CalAZ \CalAO+\mathcal{A}^2_1.
\end{equation*}
\end{lemma}
\begin{proof}
We rewrite the second equation of \eqref{e:WR} as
\begin{equation}
\begin{aligned}
T_{D_t}R &= igY-i\nP[a(1-Y)] - \nP T_{R_\alpha}b - \nP\Pi(R_\alpha, b) \\
&- i (1-Y)\nP\partial_\alpha\left[J^{-\Half}(1-Y) \W_{ \alpha}\right]    + i (1-Y) \nP\partial_\alpha\left[J^{-\Half}(1-\bar{Y}) \bar{\W}_{ \alpha}\right].
\end{aligned}    \label{RParaEqn}
\end{equation}
For the first and the second terms on the right-hand side of \eqref{RParaEqn}, using \eqref{ACHalf} and \eqref{YMoser},
\begin{equation*}
 \|gY \|_{C^\epsilon_{*}} + \| a\|_{C^\epsilon_{*}} + \| aY\|_{C^\epsilon_{*}} \lesssim \|Y \|_{C^\epsilon_{*}} +  \| a\|_{C^\epsilon_{*}} +  \| a\|_{C^\epsilon_{*}}\| Y\|_{C^\epsilon_{*}} \lesssim_\CalAZ \CalAO+\mathcal{A}^2_1.
\end{equation*}
For the third and the fourth terms on the right-hand side of \eqref{RParaEqn}, we use \eqref{CsCmStar} and \eqref{CCCEstimate}  to shift derivatives:
\begin{equation*}
\|T_{R_\alpha}b \|_{C_{*}^{\epsilon}} + \|\Pi(R_\alpha, b) \|_{C_{*}^{\epsilon}} \lesssim \|R_\alpha \|_{C^{\epsilon-\Half}_{*}}\|b\|_{C^{\Half}_{*}}\lesssim_\CalAZ \mathcal{A}^2_1.
\end{equation*}
Hence, the first four terms of \eqref{RParaEqn} can be put into $K$.

As for the two remaining capillary terms, using the relation \eqref{CapillaryTerm}, 
\begin{align*}
&\nP\partial_\alpha\left[J^{-\Half}(1-Y) \W_{ \alpha}\right] = \nP\partial_\alpha \left[T_{J^{-\frac{1}{2}}(1-Y)}\W_{\alpha} +K_1\right] =  T_{J^{-\frac{1}{2}}(1-Y)}\W_{\alpha\alpha} +K,  \\
&\nP\partial_\alpha\left[J^{-\Half}(1-\bar{Y}) \bar{\W}_{ \alpha}\right] = \nP\partial_\alpha \left[T_{J^{-\frac{1}{2}}(1-\bar{Y})}\bar{\W}_{\alpha} +K_1\right] = K,
\end{align*}
where on the first line we use the estimate 
\begin{equation*}
\left\|T_{\left(J^{-\frac{1}{2}}(1-Y)\right)_\alpha}\W_{\alpha} \right\|_{C_{*}^{\epsilon}} \lesssim \|J^{-\frac{1}{2}}(1-Y)-1 \|_{C_{*}^{1+\epsilon}} \|\W \|_{C_{*}^{1+\epsilon}} \lesssim_\CalAZ \mathcal{A}^2_1, 
\end{equation*}
 so that the term $T_{\left(J^{-\frac{1}{2}}(1-Y)\right)_\alpha}\W_{\alpha}$ belongs to $K$.

Finally, we have
\begin{equation*}
(1-Y)\nP\left[J^{-\Half}(1-Y) \W_{ \alpha}\right]_\alpha = (1-Y) T_{J^{-\frac{1}{2}}(1-Y)}\W_{\alpha\alpha} +K  = T_{1-Y}T_{J^{-\frac{1}{2}}(1-Y)}\W_{\alpha\alpha} +K,
\end{equation*}
since the balanced and high-low contributions can be put into $K$ using estimates \eqref{CCCEstimate} and \eqref{CsCmStar}. 
Furthermore, one can combine two paradifferential operators $T_{1-Y}$ and $T_{J^{-\frac{1}{2}}(1-Y)}$ using paraproduct \eqref{ParaProducts}.
\begin{equation*}
T_{1-Y}T_{J^{-\frac{1}{2}}(1-Y)}\W_{\alpha\alpha} =  T_{J^{-\frac{1}{2}}(1-Y)^2}\W_{\alpha\alpha} +K.
\end{equation*}
Combining the estimates for each term on the right-hand side of \eqref{RParaEqn}, we obtain the leading term of the para-material derivative of $R$ \eqref{RParaMat}.
\end{proof}

Next, we compute the para-material derivative of $J^s$ for $s\neq 0$.
\begin{lemma} \label{t:JsParaMat}
The para-material derivative of $J^s$ satisfies the relation 
\begin{equation}
T_{D_t}J^s = -sJ^s b_\alpha +E_1, \quad   \| E_1\|_{C_{*}^{\frac{1}{2}}} \lesssim_\CalAZ \mathcal{A}^2_1.\label{JsParaMat}
\end{equation}
The same expression also holds for the material derivative of $J^s$.
One can further write 
\begin{equation}
\partial_t J^s =  -s T_{J^s(1-\bar{Y})}R_\alpha -s T_{J^s (1-Y)}\bar{R}_\alpha +E_2, \quad   \| E_2\|_{L^\infty} \lesssim_\CalAZ 1+\mathcal{A}^2_1. \label{JsTimeDerivative}
\end{equation}
\end{lemma}
\begin{proof}
Direct computation gives
\begin{equation*}
T_{D_t}J^s = sJ^{s}(1-Y)T_{D_t}\W + sJ^{s}(1-\bar{Y})T_{D_t}\bar{\W}.
\end{equation*}
For $T_{D_t}\W$, we use the first equation of \eqref{e:WR} to write
\begin{align*}
T_{D_t}\W =& -(1+\W)(1-\bar{Y})R_\alpha -T_{\W_\alpha}b -\Pi(\W_\alpha,b)  +(1+\W)M \\
=& -(1+\W)(1-\bar{Y})R_\alpha +E_1.
\end{align*}
Here, as in the derivation of \eqref{WParaMat}, $-T_{\W_\alpha}b -\Pi(\W_\alpha,b)  +(1+\W)M$ can be placed into the error term $E_1$.
As a result of the fact $(1+\W)(1-Y)=1$, we have 
\begin{equation*}
T_{D_t}J^s = -sJ^{s}[(1-\bar{Y})R_\alpha + (1-Y)\bar{R}_\alpha]+E_1.
\end{equation*}
Using the identity \eqref{DefM} and the estimate for $M$ \eqref{MBound}, 
\begin{equation*}
 T_{D_t}J^s  =  -sJ^{s}(b_\alpha +M)+E_1 = -sJ^{s}b_\alpha +E_1.
\end{equation*}
This gives the para-material derivative of $J^s$ for $s\neq 0$.
Similarly, for the material derivative of $J^s$, 
\begin{equation*}
D_t J^s = sJ^{s}(1-Y)D_t\W + sJ^{s}(1-\bar{Y}) D_t\bar{\W}.    
\end{equation*}
Again applying \eqref{e:WR} and using the same estimates as above, we get that $D_t J^s$ has the same expression as $T_{D_t} J^s$.

The remainder term $E_1$ belongs to $E_2$, and $T_{b}\partial_\alpha J^s$ can also put into $E_2$ by Moser type estimate for $J^s$ \eqref{MoserTwo}:
\begin{equation*}
    \| T_b \partial_\alpha J^s\|_{L^\infty} \lesssim \|b\|_{C^\Half_{*}} \|J^s-1\|_{C^1_{*}} \lesssim_\CalAZ  \mathcal{A}_1^2.
\end{equation*}
We can also write
\begin{equation*}
 -sJ^{s}[(1-\bar{Y})R_\alpha + (1-Y)\bar{R}_\alpha] =  -s T_{J^s(1-\bar{Y})}R_\alpha -s T_{J^s (1-Y)}\bar{R}_\alpha + E_2, 
\end{equation*}
where the high-low and the balanced portion can be put into $E_2$. 
This gives the proof of \eqref{JsTimeDerivative} for the time derivative of $J^s$.
\end{proof}

Finally, we define the operator $\mathcal{L}: = \partial_\alpha T_{J^{-\Half}}\partial_\alpha$, which will appear later in the paradifferential linearized system \eqref{ParadifferentialLinearEqn}.
We then compute the leading term of the commutator $[T_{D_t}, \mathcal{L}]$.
\begin{lemma}
The leading part of the commutator $[T_{D_t}, \mathcal{L}]$ is given by
\begin{equation}
[T_{D_t}, \mathcal{L}] = -\Half \partial_\alpha T_{J^{-\Half}b_\alpha} \partial_\alpha + \partial_\alpha \mathcal{E} = -\partial_\alpha T_{\Re T_{ J^{-\Half}(1-\bar{Y})}R_\alpha}\partial_\alpha + \partial_\alpha \mathcal{E}, \label{CommutatorBAlpha}
\end{equation}
where the operator $\mathcal{E}$ satisfies the estimate
\begin{equation*}
 \|\mathcal{E}\|_{H^{s+1} \rightarrow H^s} \lesssim_\CalAZ \mathcal{A}_1^2.
\end{equation*}
\end{lemma}
\begin{proof}
We expand the commutator,
\begin{align*}
[T_{D_t}, \mathcal{L}]u &= T_{D_t} \partial_\alpha \left( T_{J^{-\Half}} u_\alpha \right) - \partial_\alpha \left(T_{J^{-\Half}}\partial_\alpha T_{D_t}u\right)= \partial_\alpha \left(T_{T_{D_t} J^{-\Half}} u_\alpha \right) - \partial_\alpha \left(T_{J^{-\Half}} T_{b_\alpha} u_\alpha \right)\\
&= \Half \partial_\alpha \left(T_{J^{-\Half}b_\alpha} u_\alpha \right) -\partial_\alpha \left(T_{J^{-\Half}b_\alpha} u_\alpha \right) + \partial_\alpha(T_E u_\alpha)+ \partial_\alpha \left( \left(T_{J^{-\Half}b_\alpha} - T_{J^{-\Half}}T_{b_\alpha}\right)u_\alpha\right) \\
&= -\Half \partial_\alpha \left(T_{J^{-\Half}b_\alpha} u_\alpha \right)+ \partial_\alpha(T_E u_\alpha)+ \partial_\alpha \left( \left(T_{J^{-\Half}b_\alpha} - T_{J^{-\Half}}T_{b_\alpha}\right)u_\alpha\right),
\end{align*}
where $E$ is the remainder term in \eqref{JsParaMat}.
The last two terms can be put into the remainder term $\partial_\alpha \mathcal{E} u$, since
\begin{align*}
\|T_E u_\alpha\|_{\dot{H}^s} &\lesssim \|E\|_{C^\Half_{*}} \|u_\alpha\|_{H^s} \lesssim \mathcal{A}^2_1 \|u\|_{H^{s+1}}, \\
\left\| \left(T_{J^{-\Half}b_\alpha} - T_{J^{-\Half}}T_{b_\alpha}\right)u_\alpha \right\|_{H^s} &\lesssim \|J^{-\frac{1}{2}}-1 \|_{C^{\frac{3}{4}}_{*}} \|b_\alpha\|_{C^{-\Half}} \|u_\alpha\|_{H^s} \lesssim_\CalAZ \mathcal{A}^2_1 \|u\|_{H^{s+1}}.
\end{align*}
Finally, using the definition of $b$, one can write
\begin{align*}
 -\Half  T_{J^{-\Half}b_\alpha} \partial_\alpha +  \mathcal{E}  & = -\Half  T_{J^{-\Half}}T_{b_\alpha} \partial_\alpha +  \mathcal{E} = -T_{J^{-\Half}} T_{\Re \partial_\alpha \nP[(1-\bar{Y})R)]}\partial_\alpha + \mathcal{E} \\
 & = -T_{J^{-\Half}} T_{\Re \partial_\alpha T_{1-\bar{Y}}R}\partial_\alpha+ T_{J^{-\Half}} T_{\Re \nP\partial_\alpha \Pi(\bar{Y},R)}\partial_\alpha  + \mathcal{E} \\
 & = -T_{J^{-\Half}} T_{\Re T_{1-\bar{Y}}R_\alpha}\partial_\alpha + T_{J^{-\Half}} T_{\Re T_{\bar{Y}_\alpha}R}\partial_\alpha + \mathcal{E} \\
 & = -T_{J^{-\Half}} T_{\Re T_{1-\bar{Y}}R_\alpha}\partial_\alpha  + \mathcal{E} = - T_{J^{-\Half}\Re T_{1-\bar{Y}}R_\alpha}\partial_\alpha  + \mathcal{E} \\
 & = - T_{T_{J^{-\Half}}\Re T_{1-\bar{Y}}R_\alpha}\partial_\alpha - T_{T_{\Re T_{1-\bar{Y}}R_\alpha}J^{-\Half}-1}\partial_\alpha - T_{\Pi(\Re T_{1-\bar{Y}}R_\alpha,J^{-\Half}-1)}\partial_\alpha + \mathcal{E} \\
 & = - T_{T_{J^{-\Half}}\Re T_{1-\bar{Y}}R_\alpha}\partial_\alpha + \mathcal{E} = T_{\Re T_{ J^{-\Half}(1-\bar{Y})}R_\alpha}\partial_\alpha + \mathcal{E}.
\end{align*}
Here at each step, we put the perturbative parts into $\mathcal{E}$.
\end{proof}

\section{Paradifferential modified energy estimate for the water waves system} \label{s:ModifiedEnergy}
The main goal of this section is to compute the paradifferential normal from transformations for the water waves system \eqref{e:WR} and use them to construct the paradifferential modified energy estimate Theorem \ref{t:MainEnergyEstimate}.

To start the analysis, it is helpful to recast \eqref{e:WR} as a paradifferential system with source terms.
Then, we compute the paradifferential normal form transformations that remove the non-perturbative quadratic terms on the right-hand side of \eqref{WRSystemParaLin}.
These computations are similar to Proposition 3.1 of \cite{MR3667289}, where the normal form transformations for the undifferentiated water wave system \eqref{e:CWW} are computed.
Here the normal form transformations for water waves \eqref{e:WR} are calculated at the paradifferential level, and are separated into the balanced and low-high parts.
The balanced quadratic normal form transformations are bounded, whereas low-high quadratic normal form transformations are not.
We will use balanced quadratic normal forms to construct the leading normal form variables and use low-high normal forms to construct the cubic modified energy corrections.
Finally, just as in \cite{MR3667289}, we will construct quartic modified energy corrections to eliminate the remaining non-perturbative quartic energy terms.

In this section, we write $(G,K)$ for  perturbative source terms that satisfy the bound
\begin{equation}
    \|(G,K)\|_{\mathcal{H}^s}\lesssim_\CalAZ (1+\mathcal{A}^2_1) \| (\W, R)\|^2_{\mathcal{H}^s}, \quad s> 1. \label{PerturbativeSource}
\end{equation}

\subsection{The paradifferential form of the Water Waves}
The first result in this section is that the water wave system \eqref{e:WR} can be reformulated as a paradifferential system with given source terms.
These non-perturbative source terms will later be removed by  paradifferential normal from transformations.
\begin{lemma}
For $s>1$, the water wave system \eqref{e:WR} can be written as
 \begin{equation}
 \left\{
    \begin{array}{lr}
     \W_t +T_{(1-\bar{Y})(1+\W)}R_\alpha = \mathcal{G}(\W,R)+G  &\\
    R_t+ i T_{J^{-\frac{1}{2}}(1-Y)^2}\W_{\alpha\alpha} = \mathcal{K}(\W, R)+K ,&
             \end{array}
\right.  \label{WRSystemParaLin}
\end{equation}
where $(G, K)$ are perturbative source terms that satisfy \eqref{PerturbativeSource}, and the non-perturbative part of source terms $(\mathcal{G}, \mathcal{K})$ are given by
\begin{align*}
 \mathcal{G}(\W, R) =& -T_b\W_\alpha - T_{(1-\bar{Y})\W_\alpha -(1-\bar{Y})^2(1+\W)\bar{\W}_\alpha}R-T_{T_{1-\bar{Y}}R_\alpha + T_{1-Y}\bar{R}_\alpha}\W \\
 &+ T_{(1+\W)(1-\bar{Y})^2}\nP\partial_\alpha\Pi(\bar{\W}, R) - \nP\partial_\alpha\Pi(T_{1-\bar{Y}}R + T_{1-Y}\bar{R}, \W), \\
 \mathcal{K}(\W, R) =& 3iT_{J^{-\Half}(1-Y)^3\W_\alpha}\W_\alpha +\frac{5}{2}iT_{T_{J^{-\Half}(1-Y)^3}\W_{\alpha\alpha}}\W -\frac{i}{2}T_{T_{J^{-\frac{3}{2}}(1-Y)}\bar{\W}_{\alpha\alpha}}\W\\
& +\frac{5}{2}iT_{J^{-\Half}(1-Y)^3}\Pi\left(\W, \W_{\alpha \alpha}\right) +\frac{3}{2}iT_{J^{-\Half}(1-Y)^3}\Pi\left(\W_{\alpha}, \W_{\alpha }\right)\\
&+ \frac{i}{2}T_{J^{-\frac{3}{2}}(1-Y)}\nP \Pi\left(\bar{\W}, \W_{\alpha \alpha}\right)-\frac{i}{2}T_{J^{-\frac{3}{2}}(1-Y)}\nP \Pi\left(\bar{\W}_{\alpha \alpha}, \W\right)-T_b R_\alpha \\
&-T_{T_{1-\bar{Y}}R_\alpha}R-T_{T_{1-Y}\bar{R}_\alpha}R- \Pi(R_\alpha, T_{1-\bar{Y}}R)-T_{1-Y}\nP\partial_\alpha\Pi(\bar{R}, R).
\end{align*}
\end{lemma}

\begin{proof}
We first consider the paradifferential equation for $\W$.
According to the para-material derivative of $W$ \eqref{WFormula}, 
\begin{equation*}
W_t + T_b W_\alpha  +T_{(1+\W)(1-\bar{Y})}R = T_{1+\W}\nP\Pi(\bar{Y}, R) - \nP\Pi(\W, b) + (T_{(1+\W)(1-\bar{Y})}-T_{1+\W}T_{1-\bar{Y}})R.    
\end{equation*}
Taking the $\alpha$-derivative and using the para-products rule \eqref{ParaProducts}, we get
\begin{align*}
&\W_t + T_{(1-\bar{Y})(1+\W)}R_\alpha + T_b\W_\alpha + T_{((1-\bar{Y})(1+\W))_\alpha}R \\
=& -T_{b_\alpha}\W + T_{\W_\alpha}\nP\Pi(\bar{Y}, R)+ T_{1+\W}\nP\partial_\alpha\Pi(\bar{Y}, R) - \nP\partial_\alpha\Pi(\W, b) + G\\
=& -T_{b_\alpha}\W + T_{1+\W}\nP\partial_\alpha\Pi(\bar{Y}, R) - \nP\partial_\alpha\Pi(\W, b) + G\\
=& -T_{T_{1-\bar{Y}}R_\alpha + T_{1-Y}\bar{R}_\alpha}\W + T_{(1+\W)(1-\bar{Y})^2}\nP\partial_\alpha\Pi(\bar{\W}, R) - \nP\partial_\alpha\Pi(\W, T_{1-\bar{Y}}R + T_{1-Y}\bar{R}) + G.
\end{align*}
Here, we use the definition of $b$ and place the perturbative terms in $G$ at each step.
This gives the first paradifferential equation of the water waves.

Then we consider the paradifferential equation for $R$.
Recall the equation \eqref{RParaEqn}, 
\begin{equation*}
\begin{aligned}
R_t + T_b R_\alpha &= igY - \nP T_{R_\alpha}b - \nP\Pi(R_\alpha, b)-i\nP[a(1-Y)] \\
&- i (1-Y)\nP\partial_\alpha\left[J^{-\Half}(1-Y) \W_{ \alpha}\right]    + i (1-Y) \nP\partial_\alpha\left[J^{-\Half}(1-\bar{Y}) \bar{\W}_{ \alpha}\right].
\end{aligned}    
\end{equation*}
For the first three terms on the right-hand side,
\begin{equation*}
igY - \nP T_{R_\alpha}b - \nP\Pi(R_\alpha, b) = -T_{T_{1-\bar{Y}}R_\alpha}R-\nP \Pi(R_\alpha, T_{1-\bar{Y}}R + T_{1-Y}\bar{R})+K.
\end{equation*}
For the fourth term on the right-hand side, we use the definition of the auxiliary function $a$,
\begin{align*}
&-i\nP[a(1-Y)] = -i\nP T_{1-Y}a + iT_{a}Y + i\nP\Pi(a,Y) \\
=&-T_{1-Y}T_{\bar{R}_\alpha}R-T_{1-Y}\nP\Pi(\bar{R}_\alpha, R)+K \\
=&-T_{T_{1-Y}\bar{R}_\alpha}R-\nP\Pi(\bar{R}_\alpha, T_{1-Y}R)+K.
\end{align*}
As for the first capillary term,
\begin{align*}
&- i (1-Y)\nP\partial_\alpha\left[J^{-\Half}(1-Y) \W_{ \alpha}\right]    \\
=& -i(1-Y)\partial_\alpha T_{J^{-\frac{1}{2}}(1-Y)}\W_{\alpha} - i (1-Y)\nP\partial_\alpha T_{\W_\alpha}(J^{-\Half}(1-Y)-1)\\
&-i(1-Y)\nP\partial_\alpha \Pi(J^{-\Half}(1-Y)
-1, \W_\alpha) \\
=& -iT_{J^{-\Half}(1-Y)^2}\W_{\alpha \alpha} + iT_{T_{J^{-\Half}(1-Y)}\W_{\alpha\alpha}}Y +i\Pi\left(Y, T_{J^{-\Half}(1-Y)}\W_{\alpha\alpha}\right)\\
&-iT_{1-Y}T_{(J^{-\Half}(1-Y))_\alpha}\W_\alpha +\frac{3}{2}i(1-Y) \partial_\alpha T_{\W_\alpha}T_{J^{-\Half}(1-Y)^2}\W\\
&-iT_{1-Y}\nP\partial_\alpha \Pi(J^{-\Half}(1-Y)-1, \W_\alpha)+K\\
=& -iT_{J^{-\Half}(1-Y)^2}\W_{\alpha \alpha} + iT_{T_{J^{-\Half}(1-Y)^3}\W_{\alpha\alpha}}\W +i\Pi\left(\W, T_{J^{-\Half}(1-Y)^3}\W_{\alpha\alpha}\right) \\
&+\frac{3}{2}iT_{J^{-\Half}(1-Y)^3\W_\alpha}\W_\alpha +\frac{i}{2}T_{J^{-\frac{3}{2}}(1-Y)\bar{\W}_\alpha}\W_\alpha+ \frac{3}{2}iT_{J^{-\Half}(1-Y)^3\W_\alpha}\W_\alpha\\
&+\frac{3}{2}iT_{T_{J^{-\Half}(1-Y)^3}\W_{\alpha\alpha}}\W+ \frac{i}{2}T_{1-Y}\nP\partial_\alpha \Pi\left(3T_{J^{-\Half}(1-Y)^2}\W+T_{J^{-\frac{3}{2}}}\bar{\W}, \W_\alpha\right) +K\\
=& -iT_{J^{-\Half}(1-Y)^2}\W_{\alpha \alpha} +\frac{5}{2}iT_{T_{J^{-\Half}(1-Y)^3}\W_{\alpha\alpha}}\W +\frac{i}{2}T_{J^{-\frac{3}{2}}(1-Y)\bar{\W}_\alpha}\W_\alpha +3iT_{J^{-\Half}(1-Y)^3\W_\alpha}\W_\alpha\\
& +iT_{J^{-\Half}(1-Y)^3}\Pi(\W, \W_{\alpha\alpha})+\frac{3}{2}iT_{J^{-\Half}(1-Y)^3}\partial_\alpha\Pi\left(\W, \W_{\alpha}\right) + \frac{i}{2}T_{J^{-\frac{3}{2}}(1-Y)}\nP\partial_\alpha \Pi\left(\bar{\W}, \W_\alpha\right) +K.
\end{align*}
Similarly, for the second capillary term,
\begin{align*}
&i (1-Y) \nP\partial_\alpha\left[J^{-\Half}(1-\bar{Y}) \bar{\W}_{ \alpha}\right]\\
=& i (1-Y) \nP\partial_\alpha T_{\bar{\W}_{ \alpha}} (J^{-\Half}(1-\bar{Y})-1) + i (1-Y) \nP\partial_\alpha \Pi(J^{-\Half}(1-\bar{Y})-1, \bar{\W}_{ \alpha})\\
=&  -\frac{i}{2} (1-Y) \partial_\alpha T_{\bar{\W}_{ \alpha}} T_{J^{-\frac{3}{2}}}\W -\frac{i}{2}(1-Y) \nP\partial_\alpha \Pi\left(T_{J^{-\frac{3}{2}}}\W+3T_{J^{-\Half}(1-\bar{Y})^2}\bar{\W}, \bar{\W}_{ \alpha}\right)+K\\
=&-\frac{i}{2} (1-Y)T_{\bar{\W}_{\alpha\alpha}J^{-\frac{3}{2}}}\W-\frac{i}{2} (1-Y)T_{\bar{\W}_{\alpha}J^{-\frac{3}{2}}}\W_\alpha \\
&-\frac{i}{2}T_{1-Y} \nP\partial_\alpha \Pi\left(T_{J^{-\frac{3}{2}}}\W+3T_{J^{-\Half}(1-\bar{Y})^2}\bar{\W}, \bar{\W}_{ \alpha}\right)+K \\
=&-\frac{i}{2}T_{1-Y}T_{\bar{\W}_{\alpha\alpha} J^{-\frac{3}{2}}}\W-\frac{i}{2} T_{1-Y}T_{\bar{\W}_{\alpha}J^{-\frac{3}{2}}}\W_\alpha -\frac{i}{2} T_{J^{-\frac{3}{2}}(1-Y)}\nP \partial_\alpha\Pi(\W, \bar{\W}_{ \alpha})+K\\
=&-\frac{i}{2}T_{T_{J^{-\frac{3}{2}}(1-Y)}\bar{\W}_{\alpha\alpha}}\W -\frac{i}{2} T_{\bar{\W}_{\alpha}J^{-\frac{3}{2}}(1-Y)}\W_\alpha -\frac{i}{2} T_{J^{-\frac{3}{2}}(1-Y)}\nP \partial_\alpha \Pi(\W, \bar{\W}_{\alpha })+K.
\end{align*}
Hence, adding two capillary terms, they are
\begin{align*}
&- i (1-Y)\nP\partial_\alpha\left[J^{-\Half}(1-Y) \W_{ \alpha}\right]    + i (1-Y) \nP\partial_\alpha\left[J^{-\Half}(1-\bar{Y}) \bar{\W}_{ \alpha}\right]\\
=& -iT_{J^{-\Half}(1-Y)^2}\W_{\alpha \alpha} +3iT_{J^{-\Half}(1-Y)^3\W_\alpha}\W_\alpha +\frac{5}{2}iT_{T_{J^{-\Half}(1-Y)^3}\W_{\alpha\alpha}}\W -\frac{i}{2}T_{T_{J^{-\frac{3}{2}}(1-Y)}\bar{\W}_{\alpha\alpha}}\W\\
& +\frac{5}{2}iT_{J^{-\Half}(1-Y)^3}\Pi\left(\W, \W_{\alpha \alpha}\right) +\frac{3}{2}iT_{J^{-\Half}(1-Y)^3}\Pi\left(\W_{\alpha}, \W_{\alpha }\right)\\
&+ \frac{i}{2}T_{J^{-\frac{3}{2}}(1-Y)}\nP \Pi\left(\bar{\W}, \W_{\alpha \alpha}\right)-\frac{i}{2}T_{J^{-\frac{3}{2}}(1-Y)}\nP \Pi\left(\bar{\W}_{\alpha \alpha}, \W\right) +K.
\end{align*}
Putting all the terms together, we obtain the paradifferential system \eqref{WRSystemParaLin}.
\end{proof}

\subsection{Computation of normal form transformations}
Having rewritten the water wave system as the paradifferential system \eqref{WRSystemParaLin}, we now compute the quadratic normal form transformations $(\W_{[2]}, R_{[2]})$ that remove the quadratic parts of $(\mathcal{G}, \mathcal{K})$.
Here and later when we write quadratic terms, the para-coefficients are usually not taken into account, and are simply viewed as constant coefficients.

We consider normal form transformations as the sum of bilinear forms of the following type:
\begin{align*}
\W_{[2]}   &= B^h(\W, T_{1-Y}\W) + C^h(R, T_{J^{\frac{1}{2}}(1+\W)^2(1-\bar{Y})}R) + B^a(\bar{\W}, T_{1-\bar{Y}} \W) + C^a(\bar{R}, T_{J^{\Half}(1+\W)}R) \\
R_{[2]} &= A^h(R, T_{1-Y}\W) + D^h(\W, T_{1-Y}R) + A^a(\bar{R}, T_{(1+\bar{\W})(1-Y)^2}\W) + D^a(\bar{\W}, T_{1-\bar{Y}} R),
\end{align*}
where $A^h, B^h, C^h, D^h$ are bilinear forms of the holomorphic type, and $A^a, B^a, C^a, D^a$ are bilinear forms of the mixed type.
For each bilinear form, we can then consider the low-high and the balanced portions.
Direct computation gives,
\begin{align*}
    &\partial_t \W_{[2]}+T_{(1-\bar{Y})(1+\W)} \partial_\alpha R_{[2]}+ \text{cubic and higher  terms}   \\
    =& \partial_\alpha A^h(R,  \W)-B^h(R_\alpha, \W)  -i C^h(R,  \W_{\alpha \alpha})    
    -B^h(\W, R_\alpha)-iC^h(\W_{\alpha \alpha}, R)+\partial_\alpha D^h(\W, R) \\
    &+\partial_\alpha A^a(\bar{R}, \W)-B^a(\bar{R}_\alpha, \W)- iC^a(\bar{R}, \W_{\alpha \alpha}) 
    -B^a(\bar{\W}, R_\alpha)+iC^a(\bar{\W}_{\alpha \alpha}, R)+\partial_\alpha D^a(\bar{\W}, R), \\
    &\partial_t R_{[2]} + i T_{(1-Y)^2J^{-\Half}} \partial_\alpha^2 \W_{[2]} + \text{cubic and higher  terms} \\
    =& -A^h(R, R_\alpha) +i\partial_\alpha^2 C^h(R, R) -D^h(R_\alpha, R) 
     -iA^h(\W_{\alpha \alpha}, \W)+i\partial_\alpha^2 B^h(\W, \W)-iD^h(\W,  \W_{\alpha \alpha}) \\
     & -A^a(\bar{R}, R_\alpha)+ i\partial_\alpha^2 C^a(\bar{R}, R) - D^a(\bar{R}_\alpha, R)
      +i A^a(\bar{\W}_{\alpha \alpha}, \W)+ i\partial_\alpha^2 B^a(\bar{\W}, \W) - iD^a(\bar{\W}, \W_{\alpha \alpha}).
\end{align*}
We classify the quadratic terms into the holomorphic/ mixed and low-high/ balanced types and compute the normal forms separately.
To eliminate the quadratic source terms, we take the Fourier transform and solve  linear systems for  symbols.
We write $\mathfrak{a}^h(\xi, \eta)$ for the symbol of $A^h(R, T_{1-Y}\W)$, and $\mathfrak{a}^a(\eta, \zeta)$ for the symbol of $A^a(\bar{R}, T_{(1+\bar{\W})(1-Y)^2}\W)$. 
The other symbols are defined similarly.

\textbf{(i) The low-high holomorphic case:}
In this case, we seek low-high holomorphic bilinear forms $A^h_{lh}, B^h_{lh}, C^h_{lh}, D^h_{lh}$ such that the low-high holomorphic parts of the quadratic normal form transformation $(\W_{[2]}^{h,lh}, R_{[2]}^{h,lh})$ satisfy
\begin{align*}
&\partial_t \W_{[2]}^{h,lh}+T_{(1-\bar{Y})(1+\W)} \partial_\alpha R_{[2]}^{h,lh} + \text{cubic and higher  terms} \\
=& T_{R(1-\bar{Y})}\W_\alpha + T_{(1-\bar{Y})\W_\alpha}R+ T_{R_\alpha(1-\bar{Y})}\W, \\
&\partial_t R_{[2]}^{h,lh} + i T_{(1-Y)^2J^{-\Half}} \partial_\alpha^2 \W_{[2]}^{h,lh}+ \text{cubic and higher terms} \\
=& -3iT_{J^{-\Half}(1-Y)^3\W_\alpha}\W_\alpha -\frac{5}{2}iT_{J^{-\Half}(1-Y)^3\W_{\alpha\alpha}}\W+  T_{R(1-\bar{Y})}R_\alpha + T_{(1-\bar{Y})R_\alpha}R.
\end{align*}
Define the symbol $\chi_{1}(\xi, \eta)$ that selects the low-high frequencies \eqref{ChiOnelh}, the symbols $\mathfrak{a}^h_{lh}$, $\mathfrak{b}^h_{lh}$, $\mathfrak{c}^h_{lh}$, $\mathfrak{d}^h_{lh}$ of low-high holomorphic bilinear forms $A^h_{lh}, B^h_{lh}, C^h_{lh}, D^h_{lh}$ then solve the system
\begin{equation*}
\left\{
    \begin{array}{lr}
    (\xi+\eta) \mathfrak{a}^h_{lh} - \xi \mathfrak{b}^h_{lh} + \eta^2 \mathfrak{c}^h_{lh} = (\xi +\eta) \chi_1(\xi, \eta) &\\
    \eta \mathfrak{b}^h_{lh} - \xi^2 \mathfrak{c}^h_{lh} - (\xi+\eta) \mathfrak{d}^h_{lh} = -\xi \chi_1(\xi, \eta)  &\\
    \eta \mathfrak{a}^h_{lh} + (\xi + \eta)^2 \mathfrak{c}^h_{lh} + \xi \mathfrak{d}^h_{lh} = -(\xi+\eta)\chi_1(\xi, \eta)  &\\
     \xi^2 \mathfrak{a}^h_{lh} -(\xi+ \eta)^2 \mathfrak{b}^h_{lh} + \eta^2 \mathfrak{d}^h_{lh} = 3\xi\eta \chi_1(\xi, \eta) +\frac{5}{2} \xi^2 \chi_1(\xi, \eta).&  
    \end{array}
\right.
\end{equation*}
This  system  has the solution
\begin{equation}
 \begin{aligned}
&\mathfrak{a}^h_{lh}(\xi, \eta) = \dfrac{(-9\xi^4- \xi^3\eta + 26\xi^2\eta^2 +28\xi \eta^3 + 12\eta^4) \chi_1(\xi, \eta)}{2\xi \eta(9\xi^2 + 14\xi \eta + 9 \eta^2)}, \\
&\mathfrak{b}^h_{lh}(\xi, \eta) = -\dfrac{(9\xi^3+ 28\xi^2\eta +27\xi\eta^2 +4 \eta^3)\chi_1(\xi, \eta)}{2 \eta(9\xi^2 + 14\xi \eta + 9 \eta^2)}, \\
&\mathfrak{c}^h_{lh}(\xi, \eta) = -\dfrac{ (3\xi^3 + 6\xi^2 \eta + 11\xi\eta^2+6\eta^3)\chi_1(\xi, \eta)}{\xi 
\eta(9\xi^2 + 14\xi \eta + 9 \eta^2)}, \\
&\mathfrak{d}^h_{lh}(\xi, \eta) = \dfrac{(6\xi^3+ 15\xi^2\eta +7\xi\eta^2 -4 \eta^3)\chi_1(\xi, \eta)}{2\eta(9\xi^2 + 14\xi \eta + 9 \eta^2)}. \label{BilinearHLH}
\end{aligned}   
\end{equation}

\textbf{(ii) The low-high mixed case:}
In this case, we seek low-high mixed bilinear forms $A^a_{lh}, B^a_{lh}, C^a_{lh}, D^a_{lh}$ such that the low-high mixed parts of the quadratic normal form transformation $(\W_{[2]}^{a,lh}, R_{[2]}^{a,lh})$ satisfy
\begin{align*}
&\partial_t \W_{[2]}^{a,lh}+T_{(1-\bar{Y})(1+\W)} \partial_\alpha R_{[2]}^{a,lh} + \text{cubic and higher  terms} \\
=& T_{\bar{R}(1-Y)}\W_\alpha - T_{(1-\bar{Y})^2(1+\W)\bar{\W}_\alpha}R+ T_{\bar{R}_\alpha(1-Y)}\W, \\
&\partial_t R_{[2]}^{a,lh} + i T_{(1-Y)^2J^{-\Half}} \partial_\alpha^2 \W_{[2]}^{a,lh}+ \text{cubic and higher terms} \\
=&  \frac{1}{2}iT_{J^{-\frac{3}{2}}(1-Y)\bar{\W}_{\alpha\alpha}}\W+  T_{\bar{R}(1-Y)}R_\alpha + T_{(1-Y)\bar{R}_\alpha}R.
\end{align*}
The symbols $\mathfrak{a}^a_{lh}$, $\mathfrak{b}^a_{lh}$, $\mathfrak{c}^a_{lh}$, $\mathfrak{d}^a_{lh}$ of low-high mixed bilinear forms $A^a_{lh}, B^a_{lh}, C^a_{lh}, D^a_{lh}$ solve the algebraic system
\begin{equation*}
\left\{
    \begin{array}{lr}
    (\zeta -\eta) \mathfrak{a}^a_{lh} + \eta \mathfrak{b}^a_{lh} + \zeta^2 \mathfrak{c}^a_{lh} = (\zeta -\eta) \chi_1(\eta, \zeta) &\\
    \zeta \mathfrak{b}^a_{lh} + \eta^2 \mathfrak{c}^a_{lh} - (\zeta-\eta) \mathfrak{d}^a_{lh} =  -\eta \chi_1(\eta, \zeta)   &\\
    \zeta \mathfrak{a}^a_{lh} + ( \zeta -\eta)^2 \mathfrak{c}^a_{lh} - \eta \mathfrak{d}^h_{lh} = -(\zeta-\eta)\chi_1(\eta, \zeta)  &\\
     \eta^2 \mathfrak{a}^a_{lh} +( \zeta -\eta)^2 \mathfrak{b}^a_{lh} - \zeta^2 \mathfrak{d}^a_{lh} =  \frac{1}{2}\eta^2 \chi_1(\eta, \zeta).&  
    \end{array}
\right.
\end{equation*}
The solution of this algebraic system is given by
\begin{equation}
\begin{aligned}
&\mathfrak{a}^a_{lh}(\eta, \zeta) = -\dfrac{(6\eta^4-15\eta^3\zeta + 20\eta^2\zeta^2 -20\eta \zeta^3 + 12\zeta^4) \chi_1(\eta, \zeta)}{2\eta (\zeta -\eta)(4\eta^2 -4\eta \zeta + 9 \zeta^2)}, \\
&\mathfrak{b}^a_{lh}(\eta, \zeta) = \dfrac{(2\eta^3-3\eta^2\zeta +7\eta\zeta^2 -14 \zeta^3)\chi_1(\eta, \zeta)}{2 (\zeta -\eta)(4\eta^2 -4\eta \zeta + 9 \zeta^2)}, \\
&\mathfrak{c}^a_{lh}(\eta, \zeta) = \dfrac{ \zeta(5\eta^2 - 7\eta \zeta + 6\zeta^2)\chi_1(\eta, \zeta)}{\eta (\zeta -\eta)(4\eta^2 -4\eta \zeta + 9 \zeta^2)}, \\
&\mathfrak{d}^a_{lh}(\eta, \zeta) = \dfrac{(8\eta^3 -20\eta^2\zeta +23\eta\zeta^2 - 14\zeta^3)\chi_1(\eta, \zeta)}{2 (\zeta -\eta)(4\eta^2 -4\eta \zeta + 9 \zeta^2)}. \label{BilinearALH}
\end{aligned}    
\end{equation}

\textbf{(iii) The balanced holomorphic case:}
In this case, we seek balanced holomorphic bilinear forms $A^h_{hh}, B^h_{hh}, C^h_{hh}$ such that the balanced holomorphic parts of the quadratic normal form transformation $(\W_{[2]}^{h,hh}, R_{[2]}^{h,hh})$ satisfy
\begin{align*}
&\partial_t \W_{[2]}^{h,hh}+T_{(1-\bar{Y})(1+\W)} \partial_\alpha R_{[2]}^{h,hh} + \text{cubic and higher terms} \\
=& \partial_\alpha\Pi(T_{1-\bar{Y}}R, \W), \\
&\partial_t R_{[2]}^{h,hh} + i T_{(1-Y)^2J^{-\Half}} \partial_\alpha^2 \W_{[2]}^{h,hh}+ \text{cubic and higher terms} \\
=& \Pi(R_\alpha, T_{1-\bar{Y}}R) -\frac{5}{2}iT_{J^{-\Half}(1-Y)^3}\Pi\left(\W, \W_{\alpha \alpha}\right) -\frac{3}{2}iT_{J^{-\Half}(1-Y)^3}\Pi\left(\W_{\alpha}, \W_{\alpha }\right).
\end{align*}
Define the symbol $\chi_{2}(\xi, \eta)$ that selects the balanced frequencies \eqref{ChiTwohh} and $m_{sym}$ be the symmetrization of the bilinear symbol $m$, the symbols $\mathfrak{a}^h_{hh}$, $\mathfrak{b}^h_{hh}$, $\mathfrak{c}^h_{hh}$ of bilinear forms  $A^h_{hh}, B^h_{hh}, C^h_{hh}$  then solve the system
\begin{equation*}
\left\{
    \begin{array}{lr}
    (\xi+\eta) \mathfrak{a}^h_{hh} - 2\xi \mathfrak{b}^h_{hh} + 2\eta^2 \mathfrak{c}^h_{hh} = (\xi +\eta) \chi_2(\xi, \eta) &\\
    (\eta\mathfrak{a}^h_{hh})_{sym} + (\xi + \eta)^2 \mathfrak{c}^h_{hh} = -\frac{1}{2}(\xi+\eta)\chi_2(\xi, \eta)  &\\
     (\xi^2 \mathfrak{a}^h_{hh})_{sym} -(\xi+ \eta)^2 \mathfrak{b}^h_{hh}  = \frac{1}{4}(5\xi^2+6\xi \eta + 5 \eta^2) \chi_2(\xi, \eta).&  
    \end{array}
\right.
\end{equation*}
From the first equation of the system, we get
\begin{equation*}
\mathfrak{a}^h_{hh} = 2\frac{\xi}{\xi + \eta}\mathfrak{b}^h_{hh}  -2\frac{\eta^2}{\xi+\eta}\mathfrak{c}^h_{hh} + \chi_2(\xi, \eta).  
\end{equation*}
Then we get two symmetrized symbol
\begin{align*}
&(\eta\mathfrak{a}^h_{hh})_{sym} = 2\frac{\xi \eta}{\xi + \eta}\mathfrak{b}^h_{hh}  -\frac{\xi^3+\eta^3}{\xi+\eta}\mathfrak{c}^h_{hh} + \frac{\xi + \eta}{2}\chi_2(\xi, \eta), \\
&(\xi^2 \mathfrak{a}^h_{hh})_{sym} = \frac{\xi^3+\eta^3}{\xi + \eta}\mathfrak{b}^h_{hh}  -2\frac{\xi^2\eta^2}{\xi+\eta}\mathfrak{c}^h_{hh} + \frac{\xi^2 + \eta^2}{2}\chi_2(\xi, \eta).   
\end{align*}
Substituting these into the second and the third equation of the system, 
\begin{equation*}
\left\{
    \begin{array}{lr}
     2\xi\eta \mathfrak{b}^h_{hh}+ 3(\xi + \eta)\xi \eta \mathfrak{c}^h_{hh} = -(\xi+\eta)^2\chi_2(\xi, \eta)  &\\
     3(\xi+ \eta)\xi \eta \mathfrak{b}^h_{hh}+ 2\xi^2 \eta^2 \mathfrak{c}^h_{hh}  = -\frac{3}{4}(\xi+\eta)^3 \chi_2(\xi, \eta).&  
    \end{array}
\right.
\end{equation*}
The solution is then given by 
\begin{equation}
 \begin{aligned}
&\mathfrak{a}^h_{hh}(\xi, \eta) = -\dfrac{(27\xi^4+39\xi^3\eta + 23\xi^2 \eta^2-3\xi \eta^3-6\eta^4)\chi_2(\xi, \eta)}{2\xi 
\eta(9\xi^2 + 14\xi \eta + 9 \eta^2)}, \\
&\mathfrak{b}^h_{hh}(\xi, \eta) = -\dfrac{3(\xi + \eta)^2(9\xi^2+10\xi\eta + 9\eta^2)\chi_2(\xi, \eta)}{4\xi 
\eta(9\xi^2 + 14\xi \eta + 9 \eta^2)}, \\
&\mathfrak{c}^h_{hh}(\xi, \eta) = -\dfrac{3(\xi + \eta)^3\chi_2(\xi, \eta)}{2\xi 
\eta(9\xi^2 + 14\xi \eta + 9 \eta^2)}. \label{BilinearHHH}
\end{aligned}   
\end{equation}

\textbf{(iv) The balanced mixed case:}
In this case, we seek balanced mixed bilinear forms $A^a_{hh}, B^a_{hh}, C^a_{hh}, D^a_{hh}$ such that the balanced mixed parts of the quadratic normal form transformation $(\W_{[2]}^{a,hh}, R_{[2]}^{a,hh})$ satisfy
\begin{align*}
&\partial_t \W_{[2]}^{a,hh}+T_{(1-\bar{Y})(1+\W)} \partial_\alpha R_{[2]}^{a,hh} + \text{cubic and higher  terms} \\
=& \nP\partial_\alpha\Pi(T_{1-Y}\bar{R}, \W) -T_{(1+\W)(1-\bar{Y})^2}\nP\partial_\alpha\Pi(\bar{\W}, R), \\
&\partial_t R_{[2]}^{a,hh} + i T_{(1-Y)^2J^{-\Half}} \partial_\alpha^2 \W_{[2]}^{a,hh}+ \text{cubic and higher  terms} \\
=&T_{1-Y}\nP\partial_\alpha\Pi(\bar{R}, R)-\frac{i}{2}T_{J^{-\frac{3}{2}}(1-Y)}\nP \Pi\left(\bar{\W}, \W_{\alpha \alpha}\right)+ \frac{i}{2}T_{J^{-\frac{3}{2}}(1-Y)}\nP \Pi\left(\bar{\W}_{\alpha \alpha}, \W\right) .
\end{align*}
Then for symbols $\mathfrak{a}^a_{hh}$, $\mathfrak{b}^a_{hh}$, $\mathfrak{c}^a_{hh}$, $\mathfrak{a}^a_{hh}$ of bilinear forms  $A^a_{hh}, B^a_{hh}, C^a_{hh},  D^a_{hh}$, we have the algebraic system
\begin{equation*}
\left\{
    \begin{array}{lr}
    (\zeta -\eta) \mathfrak{a}^a_{hh} + \eta \mathfrak{b}^a_{hh} + \zeta^2 \mathfrak{c}^a_{hh} = (\zeta -\eta) \chi_2(\eta, \zeta)1_{\zeta< \eta} &\\
    \zeta \mathfrak{b}^a_{hh} + \eta^2 \mathfrak{c}^a_{hh} - (\zeta-\eta) \mathfrak{d}^a_{hh} =  (\zeta -\eta) \chi_2(\eta, \zeta)1_{\zeta< \eta}   &\\
    \zeta \mathfrak{a}^a_{hh} + ( \zeta -\eta)^2 \mathfrak{c}^a_{hh} - \eta \mathfrak{d}^h_{hh} = -(\zeta-\eta)\chi_2(\eta, \zeta)1_{\zeta< \eta}  &\\
     \eta^2 \mathfrak{a}^a_{hh} +( \zeta -\eta)^2 \mathfrak{b}^a_{hh} - \zeta^2 \mathfrak{d}^a_{hh} =  \frac{1}{2}(\eta^2-\zeta^2) \chi_2(\eta, \zeta)1_{\zeta< \eta}.&  
    \end{array}
\right.
\end{equation*}
Here the indicator function $1_{\zeta< \eta}$ is the symbol for the frequency projection $\nP$. 
Then the solution of this algebraic system is 
\begin{equation}
\begin{aligned}
&\mathfrak{a}^a_{hh}(\eta, \zeta) = \dfrac{3(\eta -\zeta)(2\eta^2 -\eta \zeta + 2\zeta^2) \chi_2(\eta, \zeta)1_{\zeta< \eta}}{2\eta (4\eta^2 -4\eta \zeta + 9 \zeta^2)}, \\
&\mathfrak{b}^a_{hh}(\eta, \zeta) = -\dfrac{(\eta -\zeta)(2\eta^2 +\eta\zeta + 9\zeta^2)\chi_2(\eta, \zeta)1_{\zeta< \eta}}{2\eta (4\eta^2 -4\eta \zeta + 9 \zeta^2)}, \\
&\mathfrak{c}^a_{hh}(\eta, \zeta) = -\dfrac{ 3\zeta\chi_2(\eta, \zeta)1_{\zeta< \eta}}{\eta (4\eta^2 -4\eta \zeta + 9 \zeta^2)}, \\
&\mathfrak{d}^a_{hh}(\eta, \zeta) = -\dfrac{(\eta -\zeta)(8\eta^2 -8\eta \zeta +9\zeta^2)\chi_2(\eta, \zeta)1_{\zeta< \eta}}{2 \eta (4\eta^2 -4\eta \zeta + 9 \zeta^2)}. \label{BilinearAHH}
\end{aligned}    
\end{equation}

To conclude, we have shown the following result:
\begin{proposition} \label{t:SymbolWRTwo}
Suppose that $(\W, R)$ solve the water wave system \eqref{WRSystemParaLin}.
There exists normal form transformations
\begin{align*}
\W_{[2]}   &= B^h(\W, T_{1-Y}\W) + C^h(R, T_{J^{\frac{1}{2}}(1+\W)^2(1-\bar{Y})}R) + B^a(\bar{\W}, T_{1-\bar{Y}} \W) + C^a(\bar{R}, T_{J^{\Half}(1+\W)}R), \\
R_{[2]} &= A^h(R, T_{1-Y}\W) + D^h(\W, T_{1-Y}R) + A^a(\bar{R}, T_{(1+\bar{\W})(1-Y)^2}\W) + D^a(\bar{\W}, T_{1-\bar{Y}} R),
\end{align*}
such that the modified unknowns $(\tilde{\W}, \tilde{R}): = (\W+ \W_{[2]}, R+R_{[2]})$ solve the system 
 \begin{equation*}
 \left\{
    \begin{array}{lr}
     \tilde{\W}_t +T_{(1-\bar{Y})(1+\W)}\tilde{R}_\alpha = \tilde{\mathcal{G}}(\W,R)+G  &\\
    \tilde{R}_t+ i T_{J^{-\frac{1}{2}}(1-Y)^2}\tilde{\W}_{\alpha\alpha} = \tilde{\mathcal{K}}(\W, R)+K ,&
             \end{array}
\right.  
\end{equation*}
where the leading parts of the source terms $(\tilde{\mathcal{G}}, \tilde{\mathcal{K}})$ do not contain  quadratic terms.
The holomorphic bilinear forms $A^h, B^h, C^h, D^h$  and mixed bilinear forms $A^a, B^a, C^a, D^a$ can be further separated into the low-high and balanced portions.
Their symbols are given in \eqref{BilinearHLH}, \eqref{BilinearALH}, \eqref{BilinearHHH} and \eqref{BilinearAHH} respectively.
\end{proposition}

Taking a close look at the normal form transformations $(\W_{[2]}, R_{[2]})$, one may hope that they are bounded in the sense that
\begin{equation}
    \|(\W_{[2]}, R_{[2]}) \|_{\mathcal{H}^s} \lesssim_\CalAZ \|(\W, R) \|_{\mathcal{H}^s}, \quad s> 1. \label{BoundedNormalForm}
\end{equation}
One can check that the balanced parts of the normal form transformations
\begin{align*}
\W_{[2]}^{hh}   &= B^h_{hh}(\W, T_{1-Y}\W) + C^h_{hh}(R, T_{J^{\frac{1}{2}}(1+\W)^2(1-\bar{Y})}R) + B^a_{hh}(\bar{\W}, T_{1-\bar{Y}} \W) + C^a_{hh}(\bar{R}, T_{J^{\Half}(1+\W)}R), \\
R_{[2]}^{hh} &= A^h_{hh}(R, T_{1-Y}\W)  + A^a_{hh}(\bar{R}, T_{(1+\bar{\W})(1-Y)^2}\W) + D^a_{hh}(\bar{\W}, T_{1-\bar{Y}} R),
\end{align*}
satisfy the bound \eqref{BoundedNormalForm}, where the symbols are given in \eqref{BilinearHHH} and \eqref{BilinearAHH}.
The balanced parts  of the normal form transformations $(\W_{[2]}^{hh}, R_{[2]}^{hh})$ eliminate the balanced parts of the cubic source terms $(\mathcal{G}, \mathcal{K})$ in \eqref{WRSystemParaLin}.
Plugging $(\W + \W_{[2]}^{hh}, R+ R_{[2]}^{hh})$ into \eqref{WRSystemParaLin} also produces cubic terms, which are perturbative and satisfy \eqref{PerturbativeSource}.

As for the low-high parts of the normal form transformations $(\W_{[2]}^{lh}, R_{[2]}^{lh})$, they are
\begin{align*}
\W_{[2]}^{lh}   &= B^h_{lh}(\W, T_{1-Y}\W) + C^h_{lh}(R, T_{J^{\frac{1}{2}}(1+\W)^2(1-\bar{Y})}R) + B^a_{lh}(\bar{\W}, T_{1-\bar{Y}} \W) + C^a_{lh}(\bar{R}, T_{J^{\Half}(1+\W)}R), \\
R_{[2]}^{lh} &= A^h_{lh}(R, T_{1-Y}\W) + D^h_{lh}(\W, T_{1-Y}R) + A^a_{lh}(\bar{R}, T_{(1+\bar{\W})(1-Y)^2}\W) + D^a_{lh}(\bar{\W}, T_{1-\bar{Y}} R),
\end{align*}
whose symbols are given in \eqref{BilinearHLH} and \eqref{BilinearALH}.
Unfortunately, $(\W_{[2]}^{lh}, R_{[2]}^{lh})$ do not satisfy the bound \eqref{BoundedNormalForm}. 
One cannot use  $(\W_{[2]}^{lh}, R_{[2]}^{lh})$ directly to eliminate the low-high parts of the cubic source terms $(\mathcal{G}, \mathcal{K})$ in \eqref{WRSystemParaLin}.

\subsection{Construction of  almost balanced cubic modified energy}
Having computed the quadratic paradifferential normal form transformations, we are now in a position to mimic the \emph{balanced cubic energy method} of Ai, Ifrim and Tataru first developed in \cite{ai2023dimensional}, and  construct the corresponding  almost balanced cubic paradifferential modified energy for our model problem.

To achieve the norm equivalence property \eqref{normEquivalence}, our first idea is to consider the energy
\begin{equation*}
    E^1_s = \int T_{J^{-\frac{3}{2}}}\langle D \rangle^{s+\Half}\W \cdot \langle D \rangle^{s+\Half}\bar{\W}  + |\langle D \rangle^{s}R|^2 \,d\alpha.
\end{equation*}
This choice of energy, however, violates the modified energy estimate \eqref{strongcubicest}.
Indeed, we compute the time derivative of $E_s^1$ using \eqref{WRSystemParaLin},

\begin{equation} \label{EOneSIntegral}
\begin{aligned}
&\frac{d}{dt}E^1_s =  \int T_{\partial_t J^{-\frac{3}{2}}}\langle D \rangle^{s+\Half}\W \cdot \langle D \rangle^{s+\Half}\bar{\W} \,d\alpha \\
 & -2\Re \int T_{J^{-\frac{3}{2}}}\langle D \rangle^{s+\frac{1}{2}} \bar{\W}\cdot \langle D \rangle^{s+\frac{1}{2}} T_{(1-\bar{Y})(1+\W)}R_\alpha - i\langle D \rangle^{s} R\cdot \langle D \rangle^{s} T_{J^{-\Half}(1-\bar{Y})^2}\bar{\W}_{\alpha \alpha} \,d\alpha \\
& -2\Re \int \langle D \rangle^{s+\frac{1}{2}} \bar{\W}\cdot \langle D \rangle^{s+\frac{1}{2}}  G + \langle D \rangle^{s} R\cdot \langle D \rangle^{s} \bar{K} \,d\alpha\\
&  -2\Re \int \langle D \rangle^{s+\frac{1}{2}} \bar{\W}\cdot \langle D \rangle^{s+\frac{1}{2}} \mathcal{G}  + \langle D \rangle^{s} R\cdot \langle D \rangle^{s} \bar{\mathcal{K}}  \,d\alpha.
\end{aligned}   
\end{equation}
The first integral on the right-hand side of \eqref{EOneSIntegral} is non-perturbative due to \eqref{JsTimeDerivative}, we will later construct a cubic energy correction to eliminate it.
For the second integral of \eqref{EOneSIntegral}, we add the para-coefficient $T_{J^{-\frac{3}{2}}}$ to get $J^{-\frac{3}{2}}(1-\bar{Y})(1+\W) = J^{-\Half}(1-\bar{Y})^2$, so that the leading part of this integral cancels using integration by parts.
The third integral of $\frac{d}{dt}E_s^1$ is perturbative, whereas the last integral of \eqref{EOneSIntegral} is not, since source terms $(\mathcal{G}, \mathcal{K})$ do not satisfy the bound \eqref{PerturbativeSource}.

To eliminate the last non-perturbative integral of $\frac{d}{dt}E_s^1$, we then choose the modified unknowns $(\tilde{\W}, \tilde{R}) = (\W+ \W_{[2]}, R+R_{[2]})$, and consider the modified energy
\begin{equation*}
    \tilde{E}_s = \int T_{J^{-\frac{3}{2}}}\langle D \rangle^{s+\Half}\tilde{\W} \cdot \langle D \rangle^{s+\Half}\bar{\tilde{\W}} + |\langle D \rangle^{s}\tilde{R}|^2 \,d\alpha.
\end{equation*}
This new choice of energy, on the other hand, does not seem to satisfy the norm equivalence property \eqref{normEquivalence}, because low-high quadratic normal forms $(\W_{[2]}^{lh}, R_{[2]}^{lh})$ do not satisfy the bound \eqref{BoundedNormalForm}.
It also produces more quartic and higher integral terms that are not perturbative.
In the following, we will modify this energy $\tilde{E}_s$ and construct the desired modified energy that satisfies both conditions.

As discussed in the previous subsection, the balanced part of the modified unknowns $(\tilde{\W}^{hh}, \tilde{R}^{hh}) = (\W+ \W_{[2]}^{hh}, R+R_{[2]}^{hh})$ remove the balanced part of the non-perturbative source terms of $(\mathcal{G}, \mathcal{K})$ in \eqref{WRSystemParaLin}, and $(\W_{[2]}^{hh}, R_{[2]}^{hh})$ are bounded by \eqref{BoundedNormalForm}.
Hence, we can choose
\begin{equation*}
    \tilde{E}_s^1 = \int T_{J^{-\frac{3}{2}}}\langle D \rangle^{s+\Half}\tilde{\W}^{hh} \cdot \langle D \rangle^{s+\Half}\bar{\tilde{\W}}^{hh}   + |\langle D \rangle^{s}\tilde{R}^{hh}|^2 \,d\alpha.
\end{equation*}
Then we consider the unbalanced cubic part of $\tilde{E}_s$, which is given by
\begin{align*}
&\int T_{J^{-\frac{3}{2}}}\langle D \rangle^{s+\frac{1}{2}} \W \cdot \langle D \rangle^{s+\frac{1}{2}} \bar{\W}_{[2]}^{lh} + T_{J^{-\frac{3}{2}}}\langle D \rangle^{s+\frac{1}{2}} \bar{\W} \cdot \langle D \rangle^{s+\frac{1}{2}} \W_{[2]}^{lh} \, d\alpha \\
+& \int \langle D \rangle^{s} R \cdot \langle D \rangle^{s} \bar{R}_{[2]}^{lh} + \langle D\rangle^{s} \bar{R} \cdot \langle D \rangle^{s} R_{[2]}^{lh} \, d\alpha.
\end{align*}
Plugging in the symbols of $(\W_{[2]}^{lh}, R_{[2]}^{lh})$ computed in Proposition \ref{t:SymbolWRTwo} and applying the Plancherel theorem, we now successively compute each part of this cubic integral.
Note that when $\zeta = \xi + \eta$, 
\begin{equation*}
    4\eta^2 -4\eta \zeta + 9 \zeta^2 = 9\xi^2 + 14\xi \eta + 9 \eta^2.
\end{equation*}
In the following, when we write for example $T_{\widehat{\W}(\xi)}\widehat{\W}(\eta)$, we abuse the notation for low-high para-products, and this means the frequencies $|\xi| \ll |\eta|$.
\begin{enumerate}
\item The cubic integrals that only involve $\widehat{\W}$.
This part of integrals is given by
\begin{align*}
I_1^s =& \int_{\zeta = \xi + \eta} T_{J^{-\frac{3}{2}}}\widehat{\W}(\zeta)T_{\bar{\widehat{\W}}(\xi)}T_{1-\bar{Y}}\bar{\widehat{\W}}(\eta)\left(1+\xi^2+\eta^2\right)^{s+\frac{1}{2}}\bar{\mathfrak{b}}^h_{lh}(\xi, \eta) \, d\xi d\eta \\
&+ \int_{\zeta = \xi + \eta} T_{J^{-\frac{3}{2}}}\widehat{\W}(\xi)T_{\widehat{\W}(\eta)}T_{1-Y}\bar{\widehat{\W}}(\zeta)(1+\xi^2)^{s+\frac{1}{2}}\bar{\mathfrak{b}}^a_{lh}(\eta, \zeta)\, d\xi d\eta\\
&+ \int_{\zeta = \xi + \eta} T_{J^{-\frac{3}{2}}}\bar{\widehat{\W}}(\zeta)T_{\widehat{\W}(\xi)}T_{1-Y}\widehat{\W}(\eta)\left(1+\xi^2+\eta^2\right)^{s+\frac{1}{2}}\mathfrak{b}^h_{lh}(\xi, \eta) \, d\xi d\eta \\
&+ \int_{\zeta = \xi + \eta} T_{J^{-\frac{3}{2}}}\bar{\widehat{\W}}(\xi)T_{\bar{\widehat{\W}}(\eta)}T_{1-\bar{Y}}\widehat{\W}(\zeta)(1+\xi^2)^{s+\frac{1}{2}}\mathfrak{b}^a_{lh}(\eta, \zeta)\, d\xi d\eta\\
=& 2\Re \int T_{J^{-\frac{3}{2}}}\widehat{\W}(\xi+\eta)T_{\bar{\widehat{\W}}(\xi)}T_{1-\bar{Y}}\bar{\widehat{\W}}(\eta) \mathfrak{b}^\sharp_{lh}(\xi, \eta)\,d\xi d\eta + O\left((1+\mathcal{A}^2_1) \| (\W, R)\|^2_{\mathcal{H}^s}\right).
\end{align*}
By definition, $|\xi|\ll |\eta|$, $\chi_1(\xi, \eta)\approx \chi_1(\xi, \xi+\eta)$, and $\xi$, $\eta$ are negative.
One can compute the symbol $\mathfrak{b}^\sharp_{lh}$, which is
\begin{align*}
\mathfrak{b}^\sharp_{lh}(\xi, \eta) =& \left(1+\xi^2+\eta^2\right)^{s+\frac{1}{2}}\bar{\mathfrak{b}}^h_{lh}(\xi, \eta)+ (1+\eta^2)^{s+\frac{1}{2}}\mathfrak{b}^a_{lh}(\xi, \xi+\eta)\\
=& -(1+\xi^2+\eta^2)^{s+\frac{1}{2}}\dfrac{9\xi^3+ 28\xi^2\eta +27\xi\eta^2 +4 \eta^3}{2 \eta(9\xi^2 + 14\xi \eta + 9 \eta^2)}\\
&+ (1+\eta^2)^{s+\frac{1}{2}}\dfrac{2\xi^3-3\xi^2(\xi+\eta) +7\xi(\xi +\eta)^2 -14 (\xi+\eta)^3}{2 \eta(9\xi^2 + 14\xi \eta + 9 \eta^2)}\\
=& -(1+\xi^2+\eta^2)^{s+\frac{1}{2}}+ \text{lower order terms of }\eta.
\end{align*}
The leading part of the energy correction $I_1^s$ is
\begin{equation*}
\tilde{I}_1^s = -2\Re \int \langle D \rangle^{s+\Half} T_{J^{-\frac{3}{2}}} \W \cdot \langle D \rangle^{s+\Half} T_{\bar{\W}}T_{1-\bar{Y}}\bar{\W} \, d\alpha.
\end{equation*}

\item The cubic integrals that  involve $T_{\widehat{R}} \widehat{R}$ or $T_{\bar{\widehat{R}}} \bar{\widehat{R}}$.
This part of integrals is given by
\begin{align*}
 I_2^s =& \int_{\zeta = \xi + \eta} T_{J^{-\frac{3}{2}}}\widehat{\W}(\zeta) T_{\bar{\widehat{R}}(\xi)}T_{J^{\Half}(1+\bar{\W})^2(1-Y)}\bar{\widehat{R}}(\eta)(1+\xi^2+\eta^2)^{s+\Half}\bar{\mathfrak{c}}^h_{lh}(\xi,\eta) \,d\xi d\eta\\
 &+ \int_{\zeta = \xi + \eta} \widehat{R}(\xi)T_{\widehat{R}(\eta)}T_{(1-\bar{Y})^2(1+\W)}\bar{\widehat{\W}}(\zeta)(1+\xi^2)^s \bar{\mathfrak{a}}^a_{lh}(\eta, \zeta) \, d\xi d\eta \\
 &+\int_{\zeta = \xi + \eta} T_{J^{-\frac{3}{2}}}\bar{\widehat{\W}}(\zeta) T_{\widehat{R}(\xi)}T_{J^{\Half}(1+\W)^2(1-\bar{Y})}\widehat{R}(\eta)(1+\xi^2+\eta^2)^{s+\Half}\mathfrak{c}^h_{lh}(\xi,\eta) \,d\xi d\eta\\
 &+ \int_{\zeta = \xi + \eta} \bar{\widehat{R}}(\xi)T_{\bar{\widehat{R}}(\eta)}T_{(1-Y)^2(1+\bar{\W})}\widehat{\W}(\zeta)(1+\xi^2)^s \mathfrak{a}^a_{lh}(\eta, \zeta) \, d\xi d\eta \\
 =& 2\Re \int_{\zeta = \xi + \eta} \widehat{\W}(\zeta) T_{\bar{\widehat{R}}(\xi)}T_{(1+\bar{\W})(1-Y)^2}\bar{\widehat{R}}(\eta)\mathfrak{c}^\sharp_{lh}(\xi,\eta) \,d\xi d\eta + O\left((1+\mathcal{A}^2_1) \| (\W, R)\|^2_{\mathcal{H}^s}\right).
\end{align*}
Here, we use the fact $J^{-\frac{3}{2}}J^{\Half}(1+\bar{\W})^2(1-Y) = (1-Y)^2(1+\bar{\W})$ to simplify the para-coefficients.
The symbol $\mathfrak{c}^\sharp_{lh}$ is given by
\begin{align*}
\mathfrak{c}^\sharp_{lh} = &(1+\xi^2+\eta^2)^{s+\Half}\bar{\mathfrak{c}}^h_{lh}(\xi,\eta)+ (1+\eta^2)^{s}\mathfrak{a}^a_{lh}(\xi, \xi+\eta) \\
= & - (1+\xi^2+\eta^2)^{s+\Half}\dfrac{ 3\xi^3 + 6\xi^2 \eta + 11\xi\eta^2+6\eta^3}{\xi 
\eta(9\xi^2 + 14\xi \eta + 9 \eta^2)} \\
&-(1+\eta^2)^s \dfrac{6\xi^4-15\xi^3(\xi+\eta) + 20\xi^2(\xi+\eta)^2 -20\xi (\xi+\eta)^3 + 12(\xi+\eta)^4}{2\xi \eta(9\xi^2 + 14\xi \eta + 9 \eta^2)}\\
=& -\frac{1}{3} (1+\xi^2+\eta^2)^{s}+ \text{lower order terms of }\eta,
\end{align*}
since $\xi, \eta<0$, and $|\eta|^{2s+1}$ terms get canceled.
Hence, the leading part of the energy correction $I_2^s$ is
\begin{equation*}
\tilde{I}_2^s = -\frac{2}{3}\Re \int \langle D \rangle^{s+\Half}  \W \cdot \langle D \rangle^{s-\Half} T_{\bar{R}}T_{(1+\bar{\W})(1-Y)^2}\bar{R} \, d\alpha.
\end{equation*}

\item The cubic integrals that involve $T_{\bar{\widehat{R}}}\widehat{R}$ or $T_{\widehat{R}}\bar{\widehat{R}}$. 
This part of integral is given by
\begin{align*}
 I_3^s =& \int_{\zeta = \xi + \eta} T_{J^{-\frac{3}{2}}}\widehat{\W}(\xi) T_{\widehat{R}(\eta)}T_{J^{\Half}(1+\bar{\W})}\bar{\widehat{R}}(\zeta)(1+\xi^2)^{s+\Half}\bar{\mathfrak{c}}^a_{lh}(\eta,\zeta) \,d\xi d\eta\\
 &+ \int_{\zeta = \xi + \eta} \widehat{R}(\zeta)T_{\bar{\widehat{R}}(\xi)}T_{1-\bar{Y}}\bar{\widehat{\W}}(\eta)(1+\xi^2+\eta^2)^s \bar{\mathfrak{a}}^h_{lh}(\xi,\eta) \,d\xi d\eta \\
 &+  \int_{\zeta = \xi + \eta} T_{J^{-\frac{3}{2}}}\bar{\widehat{\W}}(\xi) T_{\bar{\widehat{R}}(\eta)}T_{J^{\Half}(1+\W)}R(\zeta)(1+\xi^2)^{s+\Half} \mathfrak{c}^a_{lh}(\eta,\zeta) \,d\xi d\eta\\
 &+ \int_{\zeta = \xi + \eta} \bar{\widehat{R}}(\zeta)T_{\widehat{R}(\xi)}T_{1-Y}\widehat{\W}(\eta)(1+\xi^2+\eta^2)^s \mathfrak{a}^h_{lh}(\xi,\eta) \,d\xi d\eta \\
 =& 2\Re \int_{\zeta = \xi + \eta} \widehat{R}(\zeta) T_{\bar{\widehat{R}}(\xi)}T_{1-\bar{Y}}\bar{\widehat{\W}}(\eta)\mathfrak{a}^\sharp_{lh}(\xi,\eta) \,d\xi d\eta + O\left((1+\mathcal{A}^2_1) \| (\W, R)\|^2_{\mathcal{H}^s}\right).
 \end{align*}
 Here we use the fact that $J^{-\frac{3}{2}}J^{\Half}(1+\bar{\W}) = 1-Y$.
The symbol $\mathfrak{a}^\sharp_{lh}$ is 
\begin{align*}
\mathfrak{a}^\sharp_{lh} = & (1+\xi^2+\eta^2)^{s}\bar{\mathfrak{a}}^h_{lh}(\xi,\eta) + (1+\eta^2)^{s+\Half} \mathfrak{c}^a_{lh}(\xi, \xi+\eta) \\
=& (1+\xi^2+\eta^2)^{s} \dfrac{-9\xi^4- \xi^3\eta + 26\xi^2\eta^2 +28\xi \eta^3 + 12\eta^4}{2\xi \eta(9\xi^2 + 14\xi \eta + 9 \eta^2)} \\
&+(1+\eta^2)^{s+\Half}  \dfrac{ (\xi+ \eta)(5\xi^2 - 7\xi(\xi+ \eta) + 6(\xi+\eta)^2)}{\xi \eta (9\xi^2 + 14\xi \eta + 9 \eta^2)}\\
=& \frac{1}{3}(1+\xi^2+\eta^2)^{s} + \text{lower order terms of }\eta.
\end{align*}
Hence, the leading part of the energy correction $I_3^s$ is
\begin{equation*}
\tilde{I}_3^s = \frac{2}{3}\Re \int \langle D \rangle^{s} 
R \cdot \langle D \rangle^{s} T_{\bar{R}}T_{1-\bar{Y}}\bar{\W} \, d\alpha.
\end{equation*}

\item The rest part of the cubic integrals, where $R$ or $\bar{R}$ does not have lowest frequency.
It is given by
\begin{align*}
 I_4^s =& \int_{\zeta = \xi + \eta} \widehat{R}(\zeta) T_{\bar{\widehat{\W}}(\xi)}T_{1-\bar{Y}}\bar{\widehat{R}}(\eta)(1+\xi^2+\eta^2)^s\bar{\mathfrak{d}}^h_{lh}(\xi,\eta) \,d\xi d\eta\\
 &+ \int_{\zeta = \xi + \eta} \widehat{R}(\xi) T_{\W(\eta)}T_{1-Y}\bar{\widehat{R}}(\zeta)(1+\xi^2)^s\bar{\mathfrak{d}}^a_{lh}(\eta,\zeta) \,d\xi d\eta \\
 &+  \int_{\zeta = \xi + \eta} \bar{\widehat{R}}(\zeta) T_{\widehat{\W}(\xi)}T_{1-Y}\widehat{R}(\eta)(1+\xi^2+\eta^2)^s\mathfrak{d}^h_{lh}(\xi,\eta) \,d\xi d\eta\\
 &+ \int_{\zeta = \xi + \eta} \bar{\widehat{R}}(\xi) T_{\bar{\widehat{\W}}(\eta)}T_{1-\bar{Y}}\widehat{R}(\zeta)(1+\xi^2)^s\mathfrak{d}^a_{lh}(\eta,\zeta) \,d\xi d\eta. \\
 =& 2\Re\int_{\zeta = \xi + \eta} \widehat{R}(\zeta) T_{\bar{\widehat{\W}}(\xi)}T_{1-\bar{Y}}\bar{\widehat{R}}(\eta)\mathfrak{d}^\sharp_{lh}(\xi,\eta) \,d\xi d\eta + O\left((1+\mathcal{A}^2_1) \| (\W, R)\|^2_{\mathcal{H}^s}\right).
 \end{align*}
 The symbol $\mathfrak{d}^\sharp_{lh}$ is
 \begin{align*}
 \mathfrak{d}^\sharp_{lh} = & (1+\xi^2+\eta^2)^s\bar{\mathfrak{d}}^h_{lh}(\xi,\eta) + (1+\eta^2)^s \mathfrak{d}^\sharp_{lh}(\xi, \xi+\eta) \\
 =& (1+\xi^2+\eta^2)^s \dfrac{6\xi^3+ 15\xi^2\eta +7\xi\eta^2 -4 \eta^3}{2\eta(9\xi^2 + 14\xi \eta + 9 \eta^2)}\\
 &+ (1+\eta^2)^s \dfrac{8\xi^3 -20\xi^2(\xi+\eta) +23\xi(\xi+\eta)^2 - 14(\xi+\eta)^3}{2 \eta(9\xi^2 + 14\xi \eta + 9 \eta^2)}\\
 =&-(1+\xi^2+\eta^2)^s + \text{lower order terms of }\eta.
 \end{align*}
The leading part of the energy correction $I_4^s$ is
\begin{equation*}
\tilde{I}_4^s = -2\Re \int \langle D \rangle^{s} 
R \cdot \langle D \rangle^{s} T_{\bar{\W}}T_{1-\bar{Y}}\bar{R} \, d\alpha.
\end{equation*}
\end{enumerate}

We still need to construct the paradifferential modified energy $I_5^s$ to remove the remaining non-perturbative cubic energy. 
This part of non-perturbative cubic integrals include the first term on the right-hand side of \eqref{EOneSIntegral} and the sub-leading part of the second integral of \eqref{EOneSIntegral}.
These non-perturbative cubic integrals need $\frac{1}{2}$ less derivative on the high-frequency factor to be bounded by the right hand side of \eqref{strongcubicest}.
One can try to find cubic corrections $I_5^s$ of the following type:
\begin{equation*}
 I^s_5 = \Re \int L^{hhl}_{2s+1}(\W, \bar{\W}, \W + \bar{\W}) + L^{hhl}_{2s+\frac{1}{2}}(\W, \bar{R}, R + \bar{R})\, d\alpha,
\end{equation*}
where $L^{hhl}_s$ is some trilinear form where each of its arguments has relatively  high, high, and low frequencies respectively, and $s$ derivatives in total are distributed to two variables with high frequencies.
Let the frequencies of the first and the third (or the second and the third) factor of $L^{hhl}_{2s+1}$ be $\xi$ and $\eta$.
Then similar to the computation for the symbols of $\mathfrak{a}^\sharp_{lh}$, the denominator of the symbols of these trilinear forms have a factor $9\xi^2 + 14\xi \eta + 9 \eta^2$, which cannot be zero unless $\xi = \eta = 0$.
This indicates that the symbols of $L^{hhl}_s$ is bounded.

When taking the time derivative of each factor of cubic energy corrections $I^s_5$, a leading order $\frac{3}{2}$ term is generated and cancels a part of the non-perturbative cubic integral terms. 
This process leaves behind sub-leading terms of order $1$, featuring $\frac{1}{2}$ fewer total derivatives compared to the non-perturbative cubic integral terms. 
This shows that the remaining cubic terms are perturbative after cancellation.

The five parts of paradifferential modified energy satisfy the energy equivalence property
\begin{equation*}
 |I_1^s|+  |I_2^s| +  |I_3^s| +  |I_4^s| +  |I_5^s| \lesssim \CalAZ \|(w,r)\|^2_{\mathcal{H}^s}.
\end{equation*}
Writing 
\begin{equation*}
    E^{(\leq 3)}_s = E^1_s(\W + \W_{[2]}^{hh}, R+ R_{[2]}^{hh}) + I_1^s+ I_2^s + I_3^s +I_4^s + I_5^s, 
\end{equation*}
then the energy $E^{(\leq 3)}_s$ satisfies the norm equivalence property \eqref{normEquivalence}, and its time derivative has no non-perturbative cubic integral terms,
\begin{equation*}
\frac{d}{dt} E^{(\leq 3)}_s = \text{non-perturbative quartic integral terms} + O\left((1+\mathcal{A}^2_1) \| (\W, R)\|^2_{\mathcal{H}^s}\right).
\end{equation*}

\subsection{Quartic modified energy} \label{s:QuarticEnergy}
Having constructed the almost balanced cubic modified energy, we now build the quartic modified energy to eliminate the non-perturbative quartic integral terms in $\frac{d}{dt} E^{(\leq 3)}_s$. A similar computation was carried out in \cite{MR3667289} without a paradifferential reduction.

The non-perturbative quartic integral terms can be separated into two parts:
\begin{itemize}
\item The time derivative acts on the para-coefficients, resulting in an additional  $T_{R_\alpha}$ or $T_{\bar{R}_\alpha}$.
When constructing the cubic paradifferential modified energy, para-coefficients such as $T_{1-Y}$, $T_{J^{-\frac{3}{2}}}$ and so on are treated as constants.
Upon taking the time derivatives on these para-coefficients, the leading terms of their time derivatives are para-coefficients $T_{R_\alpha}$ or $T_{\bar{R}_\alpha}$. 
\item The transport derivative $T_{b}\partial_\alpha$ acts on the main factors of low-high cubic energy, introducing a non-perturbative quartic integral.
When the time derivative acts on the main factors of low-high cubic energy, the leading parts of order $\frac{3}{2}$ interact with other terms to eliminate the non-perturbative cubic integrals.
The sub-leading factors of order $1$, namely, $T_b \W_\alpha$, $T_b R_\alpha$, $T_b \bar{\W}_\alpha$ or $T_b \bar{R}_\alpha$ produce new non-perturbative quartic integral terms.
\end{itemize}
For the second case, we can always use integration by parts to shift the $\alpha$-derivative from the high-frequency variables to para-coefficients of low frequencies.
For instance, consider one of the integral term 
\begin{align*}
 &\Re \int T_{J^{-\frac{3}{2}}(1-\bar{Y})} \langle D \rangle^{s+\Half}  \W T_{\bar{\W}}\langle D \rangle^{s+\Half} T_b\bar{\W}_\alpha \, d\alpha = -  \Re \int T_{J^{-\frac{3}{2}}(1-\bar{Y})} \langle D \rangle^{s+\Half}  \W_\alpha  T_{\bar{\W}}\langle D \rangle^{s+\Half} T_b\bar{\W} \, d\alpha \\
 &- \Re \int T_{J^{-\frac{3}{2}}(1-\bar{Y})} \langle D \rangle^{s+\Half}  \W \cdot T_{\bar{\W}_\alpha}\langle D \rangle^{s+\Half} T_b\bar{\W} + T_{J^{-\frac{3}{2}}(1-\bar{Y})} \langle D \rangle^{s+\Half}  \W  \cdot T_{\bar{\W}}\langle D \rangle^{s+\Half} T_{b_\alpha}\bar{\W} \, d\alpha \\
 & + \text{perturbative integrals},
\end{align*}
using integration by parts.
The first term on the right will cancel with another integral
\begin{equation*}
 \Re \int T_{J^{-\frac{3}{2}}(1-\bar{Y})} \langle D \rangle^{s+\Half}  T_b \W_\alpha \cdot  T_{\bar{\W}}\langle D \rangle^{s+\Half} \bar{\W} \, d\alpha   
\end{equation*}
at the leading order.
For the remaining two integral terms, the $\alpha$-derivatives are on the variables with low frequencies.

Another observation is that two variables at low frequencies cannot have comparable frequencies, and $\alpha$-derivatives always fall on variables with relatively high frequencies.
This is because if  $\alpha$-derivatives fall on variables with the lowest or comparable frequencies, one can shift half a derivative to the variables with medium frequencies and the quartic integral terms are perturbative.
Let the frequencies of four arguments of quartic integrals be $\xi_1, \xi_2, \xi_3, \xi_4$, then one can always arrange such that,
\begin{equation}
    0<|\xi_1| \ll |\xi_3|\ll |\xi_2| \approx |\xi_4|. \label{XiOnetoFour}
\end{equation}

Note that these non-perturbative quartic integral terms need $\frac{1}{2}$ less derivative to be perturbative.
More precisely, the quartic integral terms are of the following type:
\begin{align*}
\Re &\int L_{2s+1}^{lhmh}( \bar{\W}, \W,  R_\alpha,\bar{\W})+ L_{2s+1}^{lhmh}(\bar{\W}, \W, \bar{R}_\alpha, \bar{\W})  \\
+& L_{2s+1}^{lhmh}(R, \W, \bar{\W}_\alpha, \bar{\W})+ L_{2s+1}^{lhmh}(\bar{R}, \W, \bar{\W}_\alpha, \bar{\W})  \, d\alpha \\
+ \Re &\int L_{2s+\frac{1}{2}}^{lhmh}( \bar{R},\W, R_\alpha, \bar{R}) +  L_{2s+\frac{1}{2}}^{lhmh}( \bar{R}, \W,  \bar{R}_\alpha, \bar{R})  +  L_{2s+\frac{1}{2}}^{lhmh}(R, \W, \bar{R}_\alpha, \bar{R})  \\
+&  L_{2s+\frac{1}{2}}^{lhmh}(\bar{R}, R, \bar{R}_\alpha, \bar{\W})   + L_{2s+\frac{1}{2}}^{lhmh}(R, R, \bar{R}_\alpha, \bar{\W})   \,d\alpha \\
+ \Re &\int L_{2s}^{lhmh}( \bar{\W}, R,  \bar{R}_\alpha,\bar{R})  + L_{2s}^{lhmh}(R, R, \bar{\W}_\alpha, \bar{R})  + L_{2s}^{lhmh}(\bar{R}, R, \bar{\W}_\alpha, \bar{R}) \\
+& L_{2s}^{lhmh}( \bar{\W}, R,  R_\alpha,\bar{R}) \, d\alpha.
\end{align*}
Here $L^{lhmh}_{s}$ is some quadrilinear form where each of its arguments has relatively low, high, medium, and high frequencies respectively, and $s$ derivatives in total are distributed to two variables with high frequencies.
The $\alpha$-derivatives always fall on variables with medium frequencies because of the observation we made above.

To eliminate these non-perturbative quartic terms, we consider the following quartic integral corrections:
\begin{align*}
 \Re \int& L^{lhmh}(\bar{R}, \W, R, \bar{\W}) + L^{lhmh}(\bar{\W}, \bar{\W}, R, R) + L^{lhmh}(\bar{R}, \bar{\W}, \W, R) \\
 +& L^{lhmh}(\bar{R}, \W, R, \bar{R})+ L^{lhmh}(\bar{R}, \W, \bar{R}, \bar{\W}) + L^{lhmh}(\bar{\W}, \bar{\W}, \bar{R}, R) \\
 +&  L^{lhmh}(\bar{R}, \W, \bar{R}, \bar{R}) + L^{lhmh}(\W, \W, \bar{\W}, \bar{\W})+ L^{lhmh}(\bar{\W}, \W, \W, \bar{\W})\\
 +&  L^{lhmh}(R, \W, \bar{\W}, \bar{R}) + L^{lhmh}(\bar{\W}, \W, \bar{\W}, \bar{\W})  +  L^{lhmh}(\bar{R}, R, R, \bar{R}) \\
 +&L^{lhmh}(\bar{R}, \W, \bar{\W}, \bar{R}) +  L^{lhmh}(\bar{R}, R, \bar{R}, \bar{R}) + L^{lhmh}(R, R, \bar{R}, \bar{R}) \\
 + &L^{lhmh}(\W, \bar{\W}, \bar{R}, R) \,d\alpha.
\end{align*}
We need to solve two $8\times 8$ linear systems to determine $16$ symbols for quadrilinear forms.
The computations are long but are just solving linear systems.
The detailed analysis are thus omitted, and we refer interested readers to Section $3.2$ in \cite{MR3667289} for a similar computation for cubic normal form corrections.
Here we briefly discuss why these quartic energies are bounded.

Consider just the leading parts of the water wave system 
\begin{equation*}
\left\{
             \begin{array}{lr}
             \W_t  = -R_\alpha &\\
           R_t  = - i\W_{\alpha \alpha}.&  
             \end{array}
\right.
\end{equation*}
Eliminating the $R$ (or $\W$) variable yields the scalar equation
\begin{equation*}
    (\partial_t^2 + i \partial_\alpha^3) \W = 0  \left( \text{or } (\partial_t^2 + i \partial_\alpha^3) R = 0\right),
\end{equation*}
which then gives the linear dispersion relation
\begin{equation*}
    \tau = \pm |\xi|^{\frac{3}{2}}, \quad \xi <0.
\end{equation*}
The resonance condition
\begin{equation*}
    |\xi_1|^{\frac{3}{2}} \pm |\xi_2|^{\frac{3}{2}} \pm |\xi_3|^{\frac{3}{2}} \pm |\xi_4|^{\frac{3}{2}} = 0
\end{equation*}
is satisfied if and only if there are two pairs of equal frequencies and matching signs in the last relation, which is not possible since the frequencies satisfy \eqref{XiOnetoFour}.
This indicates that the interactions are non-resonant and symbols of quadrilinear forms cannot be unbounded.

The corresponding quartic energy corrections satisfy the norm equivalence property.
Upon applying the time derivative to each factor of quartic energy corrections, a leading order $\frac{3}{2}$ term is generated and cancels a part of the non-perturbative quartic integral terms.
 This process leaves behind sub-leading terms of order $1$, featuring $\frac{1}{2}$ fewer total derivatives compared to the non-perturbative quartic integral terms.
Since these non-perturbative quartic integral terms have $\frac{1}{2}$ more derivatives to be perturbative, the remaining quartic and higher integral terms after cancellation are perturbative, obviating the need for further quintic energy correction.

To conclude this section, let $E_s^{(4)}$ be the quartic energy correction constructed above, and
\begin{equation*}
    E_s: = E^{(\leq 3)}_s + E_s^{(4)} = E^1_s(\W + \W_{[2]}^{hh}, R+ R_{[2]}^{hh}) + I_1^s+ I_2^s + I_3^s +I_4^s + I_5^s+ E_s^{(4)}.
\end{equation*}
Then the paradifferential modified energy $E_s$ satisfies the norm equivalence \eqref{normEquivalence} and modified energy estiamte \eqref{strongcubicest}.
This finishes the proof of Theorem \ref{t:MainEnergyEstimate}.

\section{Estimate for the linearized equations} \label{s:LinearEstimate}
In this section, we derive the balanced cubic  modified energy estimate for the linearized water wave system.
Let the solutions for the linearized
water wave equations around a solution $(W, Q)$ to the system \eqref{e:CWW} by $(w, q)$ and $r: = q-Rw$.
It is computed in Section 6 of \cite{MR3667289}  that the linearized unknowns $(w,r)$ solve the system
\begin{equation}
\left\{
\begin{array}{lr}
(\partial_t + \mathfrak{M}_{b}\partial_\alpha)w +\mathbf{P}\left[\frac{r_\alpha}{1+\Bar{\mathbf{W}}}\right]+\mathbf{P}b_\alpha w   = \mathbf{P}(\mathcal{G}_0(w,r)- T_w [R\bar{Y}_\alpha]  - \Pi(w, \nP[R\bar{Y}_\alpha]))&\\
(\partial_t + \mathfrak{M}_{b}\partial_\alpha)r + i\nP\left[\frac{Lw}{1+\W}\right] -i\nP[\frac{gw}{1+\W}] =\mathbf{P}\mathcal{K}_0(w,r)  +i\nP[\frac{aw}{1+\W}]-iT_{1-Y}T_{a}w.&  \label{linearizedeqn}
\end{array}
\right.
\end{equation}
Here $\mathfrak{M}_b f = \nP[b f]$, and the self-adjoint operator $L$ is defined in \eqref{LwDef}.
The source terms $\mathcal{G}_0(w,r), \mathcal{K}_0(w,r)$ are given by
\begin{equation}
\begin{aligned}
 \mathcal{G}_0(w,r)& = (1+\mathbf{W})(\mathbf{P}\Bar{m}+\Bar{\mathbf{P}}m) + w\bar{\nP}[R_\alpha \bar{Y}]-T_{\nP[R\bar{Y}_\alpha]}w, \\ 
 \mathcal{K}_0(w,r)& = \Bar{\mathbf{P}}n - \mathbf{P}\Bar{n}+i\nP \bar{p} +iT_{1-Y}T_{a}w,    
\end{aligned}
\label{GKZeroDef}
\end{equation}
where
\begin{align*}
  &m := \frac{r_\alpha + R_\alpha w}{J}+\frac{\Bar{R}w_\alpha}{(1+\mathbf{W})^2}, \qquad n : = \frac{\Bar{R}(r_\alpha +R_\alpha w)}{1+\mathbf{W}}\\   
  &p := \dfrac{w_{\alpha\alpha}}{J^{\frac{1}{2}}(1+\W)} - \left(\dfrac{2\W_\alpha}{2J^{\frac{1}{2}}(1+\W)^2}-\dfrac{\bar{\W}_\alpha}{2J^{\frac{3}{2}}}\right)w_\alpha.
\end{align*}

In this section, we prove the following modified energy estimate for  the linearized system \eqref{linearizedeqn}.
\begin{theorem} \label{t:LinearizedWellposed}
Assume that for some fixed $s>1$, on the time interval $I = [0,T]$ for some time $T$, $N^s: = \|(\W, R)\|_{L^\infty(I; \mathcal{H}^s)}<\infty$, and $\mathcal{A}_1 \lesssim 1$.
If $(w,r)$ solve the linearized system \eqref{linearizedeqn} on $I = [0,T]$, there exists an energy functional $E^{\Half}_{lin}(w,r)$ such that on $I = [0,T]$, we have the following two properties:
\begin{enumerate}
    \item Norm equivalence:
    \begin{equation*}
        E^{\Half}_{lin}(w,r) \approx_{\CalAO} \| (w,r)\|^2_{\maH^\Half}.
    \end{equation*}
    \item Energy estimate:
    \begin{equation*}
        \frac{d}{dt} E^{\Half}_{lin}(w,r) \lesssim_{\CalAO, N^s} (1+ \mathcal{A}_1^2) E^{\Half}_{lin}(w,r).
    \end{equation*}
\end{enumerate} 
\end{theorem}

From the linearized equations \eqref{linearizedeqn} we obtain the corresponding paradifferential flow
\begin{equation}
\left\{
             \begin{array}{lr}
             T_{D_t}w+T_{1-\bar{Y}}\partial_\alpha r + \Half T_{b_\alpha }w  = 0 &\\
           T_{D_t}r  + i T_{1-Y}\mathcal{L}w -igT_{1-Y}w =0,&  \label{ParadifferentialFlow}
             \end{array}
\right.
\end{equation}
where $\mathcal{L}: = \partial_\alpha T_{J^{-\Half}}\partial_\alpha$ is the leading paradifferential term of the operator $L$.
$\mathcal{L}$ is also a self-adjoint operator.
Here in the first equation of the paradifferential flow we change $T_{\nP b_\alpha}w$ to $\Half T_{b_\alpha}w$ so that $T_{b_\alpha}$ is a real para-coefficient.

The linearized equations \eqref{linearizedeqn} can be rewritten in the paradifferential form
\begin{equation}
\left\{
             \begin{array}{lr}
             T_{D_t}w+T_{1-\bar{Y}}\partial_\alpha r + \Half T_{b_\alpha }w  = \mathcal{G}^{\sharp}(w,r) &\\
           T_{D_t}r + i T_{1-Y}\mathcal{L}w -igT_{1-Y}w =\mathcal{K}^{\sharp}(w,r),&  \label{ParadifferentialLinearEqn}
             \end{array}
\right.
\end{equation}
where the source terms $(\mathcal{G}^{\sharp}, \mathcal{K}^{\sharp})$ are given by
\begin{equation*}
   \mathcal{G}^{\sharp} = \nP(\mathcal{G}_0+\mathcal{G}_1), \qquad   \mathcal{K}^{\sharp} = \nP(\mathcal{K}_0+\mathcal{K}_1),
\end{equation*}
with $(\mathcal{G}_0, \mathcal{K}_0)$ are as \eqref{GKZeroDef} and 
\begin{align*}
  \mathcal{G}_1 = &\Pi(r_\alpha, \bar{Y}) -(T_{w_\alpha}b + \Pi(w_\alpha, b)) -T_w b_\alpha
 - \Pi(w, \nP b_\alpha ) -T_w (R\bar{Y}_\alpha)- \Pi(w, \nP[R\bar{Y}_\alpha])\\
 &+ \Half T_{\bar{\nP}b_\alpha - \nP b_\alpha}w,\\
 \mathcal{K}_1 = & -(T_{r_\alpha}b + \Pi(r_\alpha, b))+ i(T_{Lw}Y +\Pi(Lw, Y)) - iT_{1-Y}\partial_\alpha T_{w_\alpha} (J^{-\frac{1}{2}}-1) -T_{1-Y}\partial_\alpha T_w \nP c\\
  &  - iT_{1-Y}\partial_\alpha \Pi(w_\alpha, J^{-\Half}-1) -T_{1-Y}\Pi(w_\alpha, c)-T_{1-Y}\Pi(w, \nP c_\alpha)+iT_{1-Y}T_w a\\
  &  + iT_{1-Y}\Pi(w,a)-iT_{(g+a)w}Y-i\Pi((g+a)w, Y)-T_{1-Y}T_c w_\alpha -T_{1-Y}T_{\nP c_\alpha}w
\end{align*}
are the paradifferential truncations.

The rest of this section is devoted to the proof of Theorem \ref{t:LinearizedWellposed}, whose proof  is divided into five steps outlined below.

First, in Section \ref{s:ComputeParawr}, we simplify and find out the leading terms of the source terms.
Then we use these source term bounds together with the linear part of the paradifferential equation \eqref{ParadifferentialLinearEqn} to obtain the leading terms of the para-material derivatives of $(w,r)$.
In the following, we write the perturbative source terms for $(G,K)$ that satisfy
\begin{equation}
 \| (G,K)\|_{\maH^\Half}\lesssim_{\CalAZ}  (1+\mathcal{A}_1^2) \| (w,r)\|_{\maH^\Half}. \label{HalfGK}
\end{equation} 

For the second step, in Section \ref{s:HomoFlow}, we prove the energy estimate of the homogeneous paradifferential flow \eqref{ParadifferentialFlow}:
\begin{proposition} \label{t:wellposedflow}
Assume that $1+\mathcal{A}_1^2  \in L^1_t([0,T])$ for some time $T>0$, then if $(w,r)$ solve the homogeneous paradifferential system \eqref{ParadifferentialFlow} on $[0,T]$, there exists an  energy functional $E^{\Half,para}_{lin}(w, r)$ such that on $[0,T]$, we have the following two properties:
\begin{enumerate}
\item Norm equivalence:
\begin{equation*}
    E^{\Half,para}_{lin}(w, r) \approx_{\CalAZ} \|(w,r)\|_{\maH^\Half}^2.
\end{equation*}
\item The time derivative of $E^{\Half, para}_{lin}(w, r)$ is bounded by
\begin{equation*}
    \frac{d}{dt}  E^{\Half,para}_{lin}(w, r) \lesssim_{\CalAZ} (1+\mathcal{A}_1^2) \|(w,r)\|_{\maH^\Half}^2.
\end{equation*}
\end{enumerate}  
\end{proposition}
Note that here the implicit constant only depends on $\CalAZ$, which is sharper compared to the implicit constant in Theorem \ref{t:LinearizedWellposed}.
We will see that Theorem \ref{t:LinearizedWellposed}  follows  from Proposition \ref{t:wellposedflow} as long as the source terms $(\mathcal{G}^{\sharp}, \mathcal{K}^{\sharp})$ satisfy,
\begin{equation}
    \|(\mathcal{G}^{\sharp}, \mathcal{K}^{\sharp})\|_{\maH^\Half}\lesssim_{\CalAO, N^s} (1+\mathcal{A}_1^2) \|(w,r)\|_{\maH^\Half}. \label{GoodSourceTermBound}
\end{equation}

In Section \ref{s:NormalAnalysis} and Section \ref{s:PartialNormal}, we take into account the nonlinear source terms $(\mathcal{G}^{\sharp}, \mathcal{K}^{\sharp})$ on the right-hand side of \eqref{ParadifferentialLinearEqn}.
Unfortunately, the source terms $(\mathcal{G}^{\sharp}, \mathcal{K}^{\sharp})$ do not satisfy the bound \eqref{GoodSourceTermBound}, because of both the quadratic contributions and the 
non-perturbative cubic contributions.

To eliminate these unfavorable source terms, we use the paradifferential normal form analysis to construct the modified normal form linear variables $(w_{NF}, r_{NF})$. 
However, unlike the pure gravity case \cite{ai2023dimensional, wan2023low}, though the balanced parts of the corrections are bounded, the low-high parts of the modified normal form variables $(w_{NF}, r_{NF})$ do not satisfy the bound 
\begin{equation}
 \|(w_{NF}, r_{NF}) - (w,r) \|_{\mathcal{H}^\Half}\lesssim \CalAZ \|(w,r) \|_{\mathcal{H}^\Half}. \label{GoodQuadratic}
\end{equation}
For the third step, in Section  \ref{s:NormalAnalysis}, we compute the normal form transformations that remove the balanced portion of the quadratic source terms.
This part of balanced normal form transformations satisfy the estimate \eqref{GoodQuadratic}.
Then for the fourth step in Section \ref{s:PartialNormal}, we compute the normal form transformation to eliminate the low-high portion of the source terms.
This part of low-high normal form transformations only satisfy the estimate 
\begin{equation*}
 \|(w_{NF}, r_{NF}) - (w,r) \|_{\mathcal{H}^\Half}\lesssim_{N^s} \CalAO \|(w,r) \|_{\mathcal{H}^\Half}. 
\end{equation*}
Finally, in Section \ref{s:CubicNormal}, we construct the cubic normal forms to remove the remaining non-perturbative cubic source terms. 
Combining all normal form transformations together, we construct the paradifferential modified energy $ E^{\Half}_{lin}(w,r) := E^{\Half, para}_{lin}(w_{NF}, r_{NF})$ that finalizes the proof of Theorem \ref{t:LinearizedWellposed}.

\subsection{Para-material derivatives of $(w,r)$} \label{s:ComputeParawr}
In this subsection, we first compute the leading terms of the source terms $(\nP\mathcal{G}_0, \nP\mathcal{K}_0)$, $(\nP\mathcal{G}_1, \nP\mathcal{K}_1)$. 
Then we compute the leading terms of para-material derivatives of $(w,r)$.
We will also compute the non-perturbative quadratic terms of $(\nP\mathcal{G}_0, \nP\mathcal{K}_0)$, $(\nP\mathcal{G}_1, \nP\mathcal{K}_1)$. 

We begin with the leading terms of $(\nP\mathcal{G}_0, \nP\mathcal{K}_0)$.
\begin{lemma} \label{t:GKZeroZero}
The source terms $(\nP\mathcal{G}_0, \nP\mathcal{K}_0)$ have the decomposition
\begin{equation*}
\nP\mathcal{G}_0 = \mathcal{G}_{0,0} +G, \quad \nP\mathcal{K}_0 = \mathcal{K}_{0,0}+K,  
\end{equation*}
where $(\mathcal{G}_{0,0}, \mathcal{K}_{0,0})$ are given by
\begin{align*}
   &\mathcal{G}_{0,0} = -\nP[T_{1-\bar{Y}}(\bar{r}_\alpha + T_{\bar{w}}\bar{R}_\alpha) T_{1-Y}\W]+\nP[T_{(1-\bar{Y})^2}\bar{w}_\alpha T_{1+\W}R],\\
   &\mathcal{K}_{0,0} = -\nP[T_{1-\bar{Y}}\bar{r}_\alpha R]-\frac{i}{2}\nP[T_{J^{-\frac{3}{2}}} \W \bar{w}_{\alpha \alpha}] + \frac{i}{2}\nP[T_{J^{-\frac{3}{2}}}\W_\alpha \bar{w}_\alpha].
\end{align*}
$(\nP\mathcal{G}_0, \nP\mathcal{K}_0)$ satisfy the bound
\begin{equation}
    \| (\nP\mathcal{G}_0, \nP\mathcal{K}_0)\|_{\mathcal{H}^0} \lesssim_{\CalAZ} \CalAO \| (w,r)\|_{\maH^\Half}, \label{ZeroHalfGKZero}
\end{equation}
and the remainder terms $(G,K)$ are perturbative source terms that satisfy \eqref{HalfGK}.
\end{lemma}
This Lemma is similar to a corresponding computation obtained in the pure gravity water wave model \cite{ai2023dimensional}, though using a different control norm.
\begin{proof}
We first consider the estimate for the source term $\nP \mathcal{G}_0$.
Writing
\begin{equation*}
m = |1-Y|^2(r_\alpha +R_\alpha w)+(1-Y)^2\bar{R}w_\alpha,
\end{equation*}
we claim that 
\begin{equation}
    \|\bar{\nP} m\|_{L^2}\lesssim_\CalAZ \CalAO \| (w,r)\|_{\maH^\Half}. \label{nPmBound}
\end{equation}
Note that the derivative in $m$ is always  applied to one of the holomorphic variables, which has to be at or below the frequency of the undifferentiated antiholomorphic variable under the projection $\bar{\nP}$.
For instance, we can estimate 
\begin{align*}
 & \| \bar{\nP}(\bar{Y}r_\alpha)\|_{L^2} = \| [\bar{\nP}, \bar{Y}]r_\alpha \|_{L^2}\lesssim \| Y\|_{C^{\Half}_{*}}\|r\|_{H^{\Half}}\lesssim \CalAO \|r\|_{H^{\Half}}, \\
 & \| \bar{\nP}(\bar{R}w_\alpha)\|_{L^2}= \| [\bar{\nP}, \bar{R}]w_\alpha \|_{L^2}\lesssim \| R\|_{C^{\Half}_{*}}\|w\|_{H^\Half}\lesssim \CalAO \|w\|_{H^{1}},
\end{align*}
using the commutator estimate \eqref{CommutatorPBMO}. 
For other terms such as $\bar{\nP}[\bar{Y}R_\alpha w]$, some factors like $\bar{Y}R_\alpha$ must be antiholomorphic.
We again use the commutator estimates \eqref{CommutatorPBMO}, \eqref{CCCEstimate}, and \eqref{CsCmStar},
\begin{align*}
&\|\bar{\nP}[\bar{Y}R_\alpha w]\|_{L^2} = \|[\bar{\nP}, \bar{\nP}[\bar{Y}R_\alpha] ]w\|_{L^2} \lesssim \|\bar{\nP}[\bar{Y}R_\alpha] \|_{C^{\epsilon}_{*}}\| w\|_{L^2}\\
\lesssim &\left(\|T_{R_\alpha}\bar{Y} \|_{C^{\epsilon}_{*}} + \|\bar{\nP}\Pi(R_\alpha, \bar{Y}) \|_{C^{\epsilon}_{*}}\right) \| w\|_{L^2} \lesssim \|R \|_{C_{*}^\Half}\| Y\|_{C^{\Half + \epsilon}_{*}}\| w\|_{L^2} \lesssim \CalAZ \CalAO \| w\|_{H^1}.
\end{align*}
Hence, the bound for $\bar{\nP}m$ \eqref{nPmBound} is proved.
Using \eqref{nPmBound} and  commutator estimates \eqref{CommutatorPBMO}, \eqref{CommutatorPBMOTwo},
\begin{align*}
&\|\nP(\W \bar{\nP}m) \|_{H^1} = \|[\nP, \W]\bar{\nP}m \|_{H^1} \lesssim \|\W\|_{C^{1}_{*}}  \|\bar{\nP} m\|_{L^2}\lesssim \mathcal{A}_1^2 \| (w,r)\|_{\maH^\Half}, \\
& \|\nP[w\bar{\nP}[R_\alpha \bar{Y}]] \|_{H^1} = \|[\nP, w]\bar{\nP}[R_\alpha \bar{Y}] \|_{H^1} \lesssim \| w\|_{H^1} \|\bar{\nP}[R_\alpha \bar{Y}] \|_{C^\epsilon_{*}} \lesssim \mathcal{A}_1^2  \| w\|_{H^1}, \\
& \|T_{\nP[R\bar{Y}_\alpha]}w \|_{H^1} \leq \|w \|_{H^1}\| \nP[R\bar{Y}_\alpha]\|_{C^\epsilon_{*}} \lesssim \mathcal{A}_1^2  \| w\|_{H^1}.
\end{align*}
The terms $\nP((1+\W)\bar{\nP} m)$, $\nP[w\bar{\nP}[R_\alpha \bar{Y}]]$ and $T_{\nP[R\bar{Y}_\alpha]}w$ may be put into $G$.
Similarly, $\nP((1+\W)\nP \bar{m}) = T_{1+\W}\nP \bar{m} +G$.
We write
\begin{equation*}
\nP \bar{m} = -\nP[(1-\bar{Y})(\bar{r}_\alpha + \bar{R}_\alpha \bar{w})Y] + \nP[(1-\bar{Y})^2\bar{w}_\alpha R].    
\end{equation*}
Akin to the proof of Lemma $5.3$ in \cite{ai2023dimensional}, we can peel off frequency components to get
\begin{equation*}
\nP \bar{m} =  -\nP[T_{1-\bar{Y}}(\bar{r}_\alpha+ T_{\bar{w}}\bar{R}_\alpha) Y  ]+\nP[T_{(1-\bar{Y})^2}\bar{w}_\alpha R]+G.
\end{equation*}
Applying the para-associativity \eqref{ParaAssociateTwo}, $T_{1+\W}$ falls on the holomorphic function.
According to the relation \eqref{YWExpression}, one can replace $T_{1+\W}Y$ by $T_{1-Y}\W$.
Then, we get the expression of $\mathcal{G}_{0,0}$.

For the estimate for the second source term $\nP \mathcal{K}_0$, a similar analysis gives the bound
\begin{equation*}
    \|\bar{\nP} n\|_{L^2}\lesssim \CalAO \| (w,r)\|_{\maH^\Half}. 
\end{equation*}
Writing $n = (1-Y)\bar{R}(r_\alpha + R_\alpha w)$, so that by putting the perturbative components into the source term $K$,
\begin{equation*}
\nP \bar{n}= \nP[T_{1-\bar{Y}}\bar{r}_\alpha R] + K.
\end{equation*}
Due to the bound for $a$ \eqref{ACHalf}, 
\begin{equation*}
 \| iT_{1-Y}T_{a}w\|_{H^\Half} \lesssim_\CalAZ \|a \|_{L^\infty} \|w \|_{H^\Half} \lesssim_\CalAZ \mathcal{A}_1^2 \|w\|_{H^1}.
\end{equation*}
The term $iT_{1-Y}T_{a}w$ belongs to $K$.

If only suffices to consider $i\nP \bar{p}$.
We write
\begin{equation*}
\nP \bar{p} = \nP[(J^{-\frac{1}{2}}(1-\bar{Y})-1)\bar{w}_{\alpha \alpha}] - \frac{3}{2}\nP[J^{-\frac{1}{2}}(1-\bar{Y})^2\bar{\W}_\alpha \bar{w}_\alpha] +\frac{1}{2}\nP[J^{-\frac{3}{2}}\W_\alpha \bar{w}_\alpha].  
\end{equation*}
For the second term of $\nP \bar{p}$, we apply the commutator estimate \eqref{CommutatorPBMO} and the Moser type estimate \eqref{MoserTwo},
\begin{align*}
&\|\nP[J^{-\frac{1}{2}}(1-\bar{Y})^2\bar{\W}_\alpha \bar{w}_\alpha]\|_{H^\Half} \\
=& \|[\nP, J^{-\frac{1}{2}}(1-\bar{Y})^2\bar{\W}_\alpha] \bar{w}_\alpha\|_{H^\Half} 
\lesssim \|\nP[J^{-\frac{1}{2}}(1-\bar{Y})^2\bar{\W}_\alpha] \|_{C^{\Half}_{*}} \| w\|_{H^1} \\
\leq& \left(\|T_{(1-\bar{Y})^2\bar{\W}_\alpha} J^{-\Half}-1 \|_{C^{\Half}_{*}} + \|\nP \Pi((1-\bar{Y})^2\bar{\W}_\alpha, J^{-\Half}-1) \|_{C^{\Half}_{*}}\right)\| w\|_{H^1}
\lesssim \mathcal{A}^2_1 \| w\|_{H^1},
\end{align*}
so that it is perturbative and can be put into the error term $K$.
For the first and the third term of $\nP \bar{p}$,
\begin{align*}
 &\|\nP[(J^{-\frac{1}{2}}(1-\bar{Y})-1)\bar{w}_{\alpha \alpha}] \|_{L^2} = \|[\nP, J^{-\frac{1}{2}}(1-\bar{Y})-1] \bar{w}_{\alpha \alpha}\|_{L^2}\\
 \lesssim& \|J^{-\frac{1}{2}}(1-\bar{Y})-1 \|_{C^1_{*}}\| w\|_{H^1} \lesssim_{\mathcal{A}_0} \CalAO \| w\|_{H^1}, \\
 &\|\nP[J^{-\frac{3}{2}}\W_\alpha \bar{w}_\alpha] \|_{L^2} = \|[\nP, J^{-\frac{3}{2}}\W_\alpha]\bar{w}_\alpha \|_{L^2}
 \lesssim \| J^{-\frac{3}{2}}\W_\alpha \|_{C^\epsilon_{*}} \|w\|_{H^1}\lesssim_{\mathcal{A}_0} \CalAO \| w\|_{H^1}.
\end{align*}
These two terms satisfy the bound \eqref{ZeroHalfGKZero}. 
For the first term of $\nP\bar{p}$, we claim that
\begin{equation*}
J^{-\frac{1}{2}}(1-\bar{Y}) -1 = -\frac{1}{2}T_{J^{-\frac{3}{2}}} \W -\frac{3}{2} T_{J^{-\frac{1}{2}}(1-\bar{Y})^2} \bar{\W}+ E,
\end{equation*}
where the error term $E$ satisfies
\begin{equation*}
 \| E\|_{C_{*}^2} \lesssim_\CalAZ \|\W \|_{C^1_{*}}^2.
\end{equation*}
Indeed, we can write
\begin{equation*}
 J^{-\frac{1}{2}}(1-\bar{Y}) -1 = f(\W, \bar{\W}), \quad f(x,y) = (1+x)^{-\frac{1}{2}}(1+y)^{-\frac{3}{2}}-1.   
\end{equation*}
Direct computation of partial derivatives gives
\begin{equation*}
\frac{\partial f}{\partial x}(x,y)= -\Half (1+x)^{-\frac{3}{2}}(1+y)^{-\frac{3}{2}}, \quad  \frac{\partial f}{\partial y}(x,y)= -\frac{3}{2} (1+x)^{-\frac{1}{2}}(1+y)^{-\frac{5}{2}}.
\end{equation*}
The paralinearization in Zygmund spaces Lemma \ref{t:Paralinear} gives
\begin{equation*}
J^{-\frac{1}{2}}(1-\bar{Y}) -1 = -\frac{1}{2}T_{J^{-\frac{3}{2}}} \W -\frac{3}{2} T_{J^{-\frac{1}{2}}(1-\bar{Y})^2} \bar{\W} + E.
\end{equation*}
Hence,
\begin{equation*}
\nP[(J^{-\frac{1}{2}}(1-\bar{Y})-1)\bar{w}_{\alpha \alpha}] = -\Half \nP[T_{J^{-\frac{3}{2}}} \W \bar{w}_{\alpha \alpha}] + K,
\end{equation*}
since $T_{J^{-\frac{1}{2}}(1-\bar{Y})^2} \bar{\W}\bar{w}_{\alpha \alpha}$ is antiholomorphic, which vanishes after applying the projection $\nP$, and 
\begin{equation*}
 \|\nP[E \bar{w}_{\alpha\alpha}] \|_{H^\Half}  = \|[\nP, E] \bar{w}_{\alpha\alpha} \|_{H^\Half} \lesssim \|E \|_{C^{\frac{3}{2}}_{*}}\|w\|_{H^1} \lesssim_\CalAZ  \mathcal{A}^2_1 \|w\|_{H^1}. 
\end{equation*}

As for the last term of $\nP\bar{p}$, we write
\begin{equation*}
\nP[J^{-\frac{3}{2}}\W_\alpha \bar{w}_\alpha] = \nP[T_{J^{-\frac{3}{2}}}\W_\alpha \bar{w}_\alpha] + K.
\end{equation*}
Setting 
\begin{equation*}
\mathcal{K}_{0,0} = -\nP[T_{1-\bar{Y}}\bar{r}_\alpha R] -\frac{i}{2}\nP[T_{J^{-\frac{3}{2}}} \W \bar{w}_{\alpha \alpha}] + \frac{i}{2}\nP[T_{J^{-\frac{3}{2}}}\W_\alpha \bar{w}_\alpha],
\end{equation*}
it is the leading term
of $\nP \mathcal{K}_0$.
\end{proof}

Next we turn our attention to the source terms $(\nP\mathcal{G}_1, \nP\mathcal{K}_1)$.
\begin{lemma}
The source terms $(\nP\mathcal{G}_1, \nP\mathcal{K}_1)$ have the representation
\begin{equation}
\begin{aligned}
 &\nP\mathcal{G}_1 = - T_{1-\bar{Y}} T_wR_\alpha +\tilde{G}, \\
 &\nP\mathcal{K}_1 = i T_{J^{-\Half} (1-Y)^2}T_w\W_{\alpha \alpha} +\tilde{K},
\end{aligned} \label{PGKOneEqn}
\end{equation}
where $(\tilde{G}, \tilde{K})$ satisfies the estimate
\begin{equation}
    \| (\tilde{G}, \tilde{K})\|_{\mathcal{H}^0} \lesssim_{\CalAZ} (1+\mathcal{A}_1^2) \| (w,r)\|_{\maH^{\Half}}. \label{HalfGKZeroTwo}
\end{equation}
\end{lemma}

\begin{proof}
We begin with the analysis for $\nP\mathcal{G}_1$. 
We compute using \eqref{HCHEstimate}, \eqref{HsHmCStar}, 
\begin{align*}
&\|\nP \Pi(r_\alpha, \bar{Y}) \|_{H^\Half} \lesssim \|r\|_{\dot{H}^\Half} \|Y \|_{C^{1+\epsilon}}\lesssim \CalAO \|r\|_{H^\Half},\\
&\|\nP T_{w_\alpha} b \|_{H^\Half} + \|\nP \Pi(w_\alpha, b) \|_{H^\Half} \lesssim \|w_{\alpha} \|_{L^2} \|b \|_{C^{\Half}_{*}}\lesssim_\CalAZ \CalAO \|w\|_{H^1}, \\
&\|\nP \Pi(w, \nP b_\alpha) \|_{H^\Half} \lesssim \|w_{\alpha} \|_{L^2} \|b \|_{C^{\Half}_{*}}\lesssim_\CalAZ \CalAO \|w\|_{H^1} ,\\
&\|T_w \nP[R\bar{Y}_\alpha]\|_{H^\Half} + \|\nP \Pi(w, \nP[R\bar{Y}_\alpha]) \|_{H^\Half} \lesssim \|w \|_{L^2} \| \nP[R\bar{Y}_\alpha]\|_{C_{*}^\Half} \lesssim \mathcal{A}_1^2 \|w\|_{H^1}, \\
& \|T_{\bar{\nP}b_\alpha - \nP b_\alpha}w \|_{H^\Half} \lesssim \| w\|_{H^1}\|b\|_{C^\Half_{*}}\lesssim_\CalAZ \CalAO \|w\|_{H^1}.
\end{align*}
All terms except $- \nP T_w b_\alpha$ can be put into $\tilde{G}$. 
For this term, we write
\begin{equation*}
 \nP T_w b_\alpha = T_w \nP[(1-\bar{Y})R]_\alpha = T_w T_{1-\bar{Y}}R_\alpha +G = T_{1-\bar{Y}}T_w R_\alpha +\tilde{G}.
\end{equation*}
We continue with the analysis for $\nP\mathcal{K}_1$.
Similarly, 
\begin{align*}
&\|\nP \Pi(r_\alpha, b) \|_{L^2} + \|\nP T_{r_\alpha} b \|_{L^2} \lesssim \|r \|_{H^{\Half}} \|b\|_{C^\Half_{*}} \lesssim_\CalAZ \CalAO \|r \|_{\dot{H}^\Half}, \\
& \|\nP T_{1-Y}\partial_\alpha T_{w_\alpha}(J^{-\Half}-1)\|_{L^2} \lesssim (1+\| Y\|_{C^\epsilon_{*}})\|T_{w_\alpha}(J^{-\Half}-1) \|_{H^1} \lesssim_\CalAZ \CalAO \|w\|_{H^1}, \\
& \|\nP T_{1-Y}\Pi(w_\alpha, c) \|_{L^2} + \|\nP T_{1-Y}\Pi(w, \nP c_\alpha) \|_{L^2} \lesssim_\CalAZ \|w \|_{H^1} \| c\|_{C^\epsilon_{*}} \lesssim_\CalAZ \CalAO \|w \|_{H^1}, \\
& \|\nP T_{1-Y}T_w a\|_{L^2} + \|\nP T_{1-Y}\Pi(w, a) \|_{L^2} \lesssim \|w\|_{L^2}\|a\|_{C^\epsilon_{*}} \lesssim \mathcal{A}^2_1 \|w\|_{H^1}, \\
& \|\nP T_{g+a}w Y\|_{L^2} + \| \nP \Pi((g+ a)w, Y) \|_{L^2} \lesssim (g+\|a \|_{L^\infty}) \|w\|_{L^2} \|Y\|_{C^\epsilon_{*}} \lesssim_\CalAZ (1+ \mathcal{A}^2_1) \|w\|_{H^1}, \\
&\|T_{1-Y}T_c w_\alpha\|_{L^2}+ \|T_{1-Y}T_{\nP c_\alpha}w \|_{L^2}\lesssim_\CalAZ \|c\|_{C^\epsilon_{*}} \|w\|_{H^1}\lesssim_\CalAZ \CalAO \|w\|_{H^1}.
\end{align*}
For $T_{L w}Y$ and $\Pi(L w, Y)$ terms, we expand $L w$ and use \eqref{HCHEstimate},  \eqref{HsHmCStar}  to shift the derivative for the leading term $\partial_\alpha (J^{-\Half}w_\alpha)$.
\begin{align*}
 \| T_{L w}Y\|_{L^2} &\leq \|T_{\partial_\alpha (J^{-\Half}w_\alpha)} Y\|_{L^2} + \|T_{c w_\alpha}Y \|_{L^2} + \|T_{T_{\nP c_\alpha} w} Y\|_{L^2} +  \|T_{T_{w} \nP c_\alpha} Y\|_{L^2} + \|T_{ \Pi(w, \nP c_\alpha)} Y\|_{L^2}  \\
 & \lesssim \|J^{-\Half}w_\alpha \|_{L^2} \|Y_\alpha \|_{C^\epsilon_{*}} + \| c\|_{L^\infty} \|w_\alpha\|_{L^2} \|Y\|_{C^\epsilon_{*}} + \|w\|_{L^2} \|c\|_{C^{\epsilon}_{*}} \|Y\|_{C^1_{*}}\| \lesssim \mathcal{A}_1^2 \|w\|_{H^1}.
\end{align*}
The estimate for $\Pi(L w, Y)$ term is similar.
Hence, all terms except $-T_{1-Y}\partial_\alpha T_w \nP c$ can be absorbed into $\tilde{K}$.
For this term, we use the representation for $\nP c$ \eqref{nPCRepresentation} to write
\begin{align*}
 T_{1-Y}\partial_\alpha T_w \nP c &= T_{1-Y} T_w \nP c_\alpha + T_{1-Y}T_{w_\alpha} \nP c =  T_{1-Y} T_w \nP c_\alpha +\tilde{K} \\
 &= -iT_{1-Y} T_w (T_{J^{-\Half}(1-Y)}\W_\alpha)_\alpha +K = -iT_{1-Y}  T_w T_{J^{-\Half}(1-Y)} \W_{\alpha\alpha} +\tilde{K}\\
 &= -iT_{1-Y} T_{J^{-\Half}(1-Y)}T_w \W_{\alpha\alpha} +K =  -iT_{J^{-\Half} (1-Y)^2}T_w\W_{\alpha \alpha} +\tilde{K}.
\end{align*}
Here on the last line we use para-products \eqref{ParaProductsTwo} and para-commutator \eqref{ParaCommutatorThree} to simplify the para-coefficient.
\end{proof}

Putting the above two lemmas for the source terms together, we obtain the following formulas for the para-material derivatives of $(w,r)$:
\begin{lemma}
Suppose  $(w,r)$ solve the paradifferential linearized equations \eqref{ParadifferentialLinearEqn}, then we have the formulas 
\begin{equation}
\left\{
             \begin{array}{lr}
             T_{D_t}w  = -T_{1-\bar{Y}}(r_\alpha + T_w R_\alpha)+ G_2 =:G_1 + G_2 &\\
           T_{D_t}r  = - iT_{J^{-\frac{1}{2}}(1-Y)}(w_{\alpha \alpha}-T_{1-Y}T_w \W_{\alpha \alpha})+K_2 =:K_1 + K_2,&  \label{wrParaMaterial}
             \end{array}
\right.
\end{equation}
where $(G_2, K_2)$ satisfy the estimate
\begin{equation*}
    \| (G_2, K_2)\|_{\mathcal{H}^0} \lesssim_{\CalAZ} (1+\mathcal{A}_1^2) \| (w,r)\|_{\mathcal{H}^{\Half}},
\end{equation*}
and $(G_1, K_1)$ satisfy the linear bound
\begin{equation*}
\|(G_1, K_1) \|_{\mathcal{H}^{-1}} \lesssim_\CalAO \|(w, r) \|_{\mathcal{H}^{\Half}}. 
\end{equation*}
\end{lemma}

\begin{proof}
For the para-material derivative of $w$, according to the conclusion of the previous two lemmas,
\begin{equation*}
T_{D_t} w = -T_{1-\bar{Y}}r_\alpha - \Half T_{ b_\alpha} w - T_{1-\bar{Y}} T_wR_\alpha +G_2.
\end{equation*}
Using \eqref{BCOneStar} and \eqref{HsCmStar}, the term $- \Half T_{ b_\alpha} w$ can also be put into $G_2$. 
This gives the formula for $T_{D_t} w$.

As for the para-material derivative of $r$, similarly, we have
\begin{equation*}
 T_{D_t} r =  -i T_{1-Y}\mathcal{L}w + i T_{J^{-\Half} (1-Y)^2}T_w\W_{\alpha \alpha} +K_2.
\end{equation*}
It suffices to simplify the first term of $T_{D_t}r$.
We write,
\begin{equation*}
T_{1-Y}\mathcal{L}w = T_{1-Y}   \partial_\alpha \left(T_{J^{-\frac{1}{2}}}w_\alpha\right) = T_{1-Y}T_{J^{-\frac{1}{2}}}w_{\alpha \alpha}+ K_2 = T_{J^{-\frac{1}{2}}(1-Y)}w_{\alpha\alpha} + K_2.
\end{equation*}
Hence, we get the formula for $T_{D_t} r$ after applying the para-products \eqref{ParaProducts}.

The estimates for $(G_1, K_1)$ are straightforward.
\end{proof}

To facilitate the normal form computation later, we also compute here expressions for the quadratic terms of the source terms $(\nP\mathcal{G}_0, \nP\mathcal{K}_0)$ and $(\nP\mathcal{G}_1, \nP\mathcal{K}_1)$.

\begin{lemma} \label{t:QuadraticGKZero}
The quadratic terms of $(\nP\mathcal{G}_0, \nP\mathcal{K}_0)$ are given by
\begin{align*}
 \nP\mathcal{G}_0^{[2]} =&   -T_{J^{-1}}T_{\bar{r}_\alpha}\W +  T_{(1-\bar{Y})^2(1+\W)}T_{\bar{w}_\alpha} R - T_{J^{-1}} \nP \Pi(\bar{r}_\alpha, \W) +  T_{(1-\bar{Y})^2(1+\W)} \nP \Pi(\bar{w}_\alpha, R) + G,\\
  \nP\mathcal{K}_0^{[2]} =&  - T_{1-\bar{Y}}T_{\bar{r}_\alpha} R -\frac{i}{2} T_{J^{-\frac{3}{2}}} T_{\bar{w}_{\alpha \alpha}}\W + \frac{i}{2} T_{J^{-\frac{3}{2}}} T_{\bar{w}_{\alpha}}\W_\alpha - T_{1-\bar{Y}}\nP\Pi(\bar{r}_\alpha, R) \\
  &-\frac{i}{2} T_{J^{-\frac{3}{2}}} \nP \Pi(\bar{w}_{\alpha \alpha},\W)+ \frac{i}{2} T_{J^{-\frac{3}{2}}} \nP\Pi(\bar{w}_\alpha, \W_\alpha)+ K.
 \end{align*}
 where $(G,K)$ are perturbative terms that satisfy \eqref{HalfGK}.
\end{lemma}
\begin{proof}
 According to the computation in Lemma \ref{t:GKZeroZero}, it suffices to compute the quadratic terms of  $(\mathcal{G}_{0,0}, \mathcal{K}_{0,0})$.
 We write
 \begin{align*}
     -\nP[T_{1-\bar{Y}}\bar{r}_\alpha T_{1-Y}\W] &= -T_{J^{-1}}T_{\bar{r}_\alpha}\W - T_{J^{-1}} \nP \Pi(\bar{r}_\alpha, \W) + G, \\
    \nP[T_{(1-\bar{Y})^2}\bar{w}_\alpha T_{1+\W}R] &= T_{(1-\bar{Y})^2(1+\W)}T_{\bar{w}_\alpha} R +  T_{(1-\bar{Y})^2(1+\W)} \nP \Pi(\bar{w}_\alpha, R) + G, \\
    -\nP[T_{1-\bar{Y}}\bar{r}_\alpha R] &= - T_{1-\bar{Y}}T_{\bar{r}_\alpha} R - T_{1-\bar{Y}}\nP\Pi(\bar{r}_\alpha, R) + K, \\
    -\frac{i}{2}\nP[T_{J^{-\frac{3}{2}}} \W \bar{w}_{\alpha \alpha}]  &= -\frac{i}{2} T_{J^{-\frac{3}{2}}} T_{\bar{w}_{\alpha \alpha}}\W -\frac{i}{2} T_{J^{-\frac{3}{2}}} \nP \Pi(\bar{w}_{\alpha \alpha},\W) + K,   \\
    \frac{i}{2}\nP[T_{J^{-\frac{3}{2}}}\W_\alpha \bar{w}_\alpha] & =  \frac{i}{2} T_{J^{-\frac{3}{2}}} T_{\bar{w}_{\alpha}}\W_\alpha + \frac{i}{2} T_{J^{-\frac{3}{2}}} \nP\Pi(\bar{w}_\alpha, \W_\alpha) + K.
 \end{align*}
 The remaining term in $\mathcal{G}_{0,0}$ is cubic.
\end{proof}

\begin{lemma} \label{t:QuadraticGKOne}
 The quadratic terms of $(\nP\mathcal{G}_1, \nP\mathcal{K}_1)$ are given by
 \begin{align*}
 \nP\mathcal{G}_1^{[2]} =& - T_{w_\alpha}T_{1-\bar{Y}}R - T_{w}T_{1-\bar{Y}}R_\alpha + \Half T_{T_{1-Y}\bar{R}_\alpha}w-\frac{1}{2}T_{T_{1-\bar{Y}}R_\alpha}w  \\
-& T_{1-\bar{Y}} \partial_\alpha \Pi(w, R) - \nP \Pi(w_\alpha, T_{1-Y}\bar{R}) + \nP \Pi(r_\alpha, T_{(1-\bar{Y})^2}\bar{\W}) + G,\\
  \nP\mathcal{K}_1^{[2]} =&  -T_{r_\alpha}T_{1-\bar{Y}}R +\frac{3}{2}i T_{J^{-\Half}(1-Y)^2} \partial_\alpha T_{w_{\alpha}} \W +i T_{J^{-\Half}(1-Y)^2} T_w \W_{\alpha \alpha} + i T_{J^{-\Half}(1-Y)^2} \partial_\alpha T_{\W_{\alpha }} w\\
   -& iT_{J^{-\frac{3}{2}}}T_{\bar{\W}_\alpha} w_\alpha -\Pi(r_\alpha, T_{1-\bar{Y}}R) - \nP \Pi(r_\alpha, T_{1-Y}\bar{R})  + \frac{3}{2}i T_{J^{-\Half}(1-Y)^2}\partial_\alpha \Pi(w_{\alpha}, \W) \\
   + & i T_{J^{-\Half}(1-Y)^2}\partial_\alpha \Pi(w, \W_{\alpha}) +\frac{i}{2}T_{J^{-\frac{3}{2}}} \nP \Pi(w_{\alpha \alpha}, \bar{\W}) -\frac{i}{2}T_{J^{-\frac{3}{2}}} \nP \Pi(w_\alpha, \bar{\W}_\alpha) + K.
 \end{align*}
 where $(G,K)$ are perturbative terms that satisfy \eqref{HalfGK}.
\end{lemma}

\begin{proof}
We begin with the computation for $\nP\mathcal{G}_1^{[2]}$.
Using the identity for $Y$ \eqref{YWExpression},
\begin{equation*}
\nP \Pi(r_\alpha, \bar{Y}) = \nP \Pi(r_\alpha, T_{(1-\bar{Y})^2}\bar{\W}) + G.
\end{equation*}
We write for other terms
\begin{align*}
& -\nP T_{w_\alpha} b - \nP T_w b_\alpha = - T_{w_\alpha}T_{1-\bar{Y}}R - T_{w}T_{1-\bar{Y}}R_\alpha + G, \\
 & -\nP \Pi(w_\alpha, b) - \Pi(w, \nP b_\alpha) = -\nP \Pi(w_\alpha, T_{1-\bar{Y}}R + T_{1-Y}\bar{R}) - \Pi(w, T_{1-\bar{Y}}R_\alpha)+ G,\\
 & \frac{1}{2}T_{\bar{\nP}b_\alpha - \nP b_\alpha}w =  \Half T_{T_{1-Y}\bar{R}_\alpha}w-\frac{1}{2}T_{T_{1-\bar{Y}}R_\alpha}w + G. 
 \end{align*}
Other terms in $\nP \mathcal{G}_1$ are cubic.

Then we compute the expression for $ \nP\mathcal{K}_1^{[2]}$.
The first two terms are straightforward,
\begin{equation*}
    -\nP T_{r_\alpha}b - \nP \Pi(r_\alpha, b) = - T_{r_\alpha}T_{1-\bar{Y}}R - \nP \Pi(r_\alpha, T_{1-\bar{Y}}R) - \nP \Pi(r_\alpha, T_{1-Y}\bar{R}) +K.
\end{equation*}
Recall Lemma \ref{t:LwLeading} for the leading term of $Lw$, 
\begin{equation*}
iT_{Lw}Y + i\nP \Pi(Lw, Y) = i T_{J^{-\Half}(1-Y)^2}T_{ w_{\alpha \alpha}} \W + i T_{J^{-\Half}(1-Y)^2}\nP\Pi(w_{\alpha \alpha}, \W)+ \text{cubic terms} + K.
\end{equation*}
Using the paralinearization Lemma \ref{t:Paralinear},
\begin{equation*}
    J^{-\Half} - 1 = -\Half T_{J^{-\Half}(1-Y)} \W -\Half T_{J^{-\Half}(1-\bar{Y})} \bar{\W} + E, \quad \|E\|_{C^2_{*}} \lesssim \mathcal{A}^2_1.
\end{equation*}
We can then write
\begin{align*}
-iT_{1-Y}\partial_\alpha \nP T_{w_\alpha}(J^{-\Half}-1) &= \frac{i}{2} T_{J^{-\Half}(1-Y)^2}\partial_\alpha T_{w_\alpha} \W + K, \\
-iT_{1-Y}\partial_\alpha \nP \Pi(w_\alpha, J^{-\Half}-1) &= \frac{i}{2} T_{J^{-\Half}(1-Y)^2}\partial_\alpha \Pi(w_\alpha, \W)+\frac{i}{2}T_{J^{-\frac{3}{2}}}\partial_\alpha \nP \Pi(w_\alpha, \bar{\W}) + K.
\end{align*}
For the five terms with $c$, we use \eqref{CDef} and \eqref{nPCRepresentation} to compute
\begin{align*}
    -T_{1-Y} \partial_\alpha T_w \nP c &= i T_{J^{-\Half}(1-Y)^2}\partial_\alpha T_w \W_\alpha + K,\\
    -T_{1-Y}\nP \Pi(w_\alpha, c) &= i T_{J^{-\Half}(1-Y)^2}  \Pi(w_\alpha, \W_\alpha) -i T_{J^{-\frac{3}{2}}}\nP \Pi(w_\alpha, \bar{\W}_\alpha) + K, \\
    -T_{1-Y}\Pi(w, \nP c_\alpha) & = iT_{J^{-\Half}(1-Y)^2} \Pi(w, \W_{\alpha \alpha}) + K,\\
    -T_{1-Y}T_c w_\alpha & = i T_{J^{-\Half}(1-Y)^2} T_{\W_{\alpha}} w_\alpha - iT_{J^{-\frac{3}{2}}}T_{\bar{\W}_\alpha} w_\alpha + K,\\
    -T_{1-Y}T_{\nP c_\alpha} w & = i T_{J^{-\Half}(1-Y)^2} T_{\W_{\alpha \alpha}} w + K.
\end{align*}
The other terms of $\nP \mathcal{K}_1$ are perturbative and can be put into $K$.
\end{proof}

\subsection{$\maH^\Half$ energy estimate for the homogeneous paradifferential flow} \label{s:HomoFlow}
In this subsection we prove Proposition \ref{t:wellposedflow}.
Consider the paradifferential systems with general source terms $(G,K)$ that satisfy \eqref{HalfGK},
\begin{equation}
\left\{
             \begin{array}{lr}
             T_{D_t}w+T_{1-\bar{Y}} r_\alpha + \Half T_{ b_\alpha }w  = G &\\
           T_{D_t}r + i T_{1-Y}\mathcal{L}w -igT_{1-Y}w =K.&  \label{SourceParadifferential}
             \end{array}
\right.
\end{equation}
Inspired by the linearized energy defined in \cite{MR3667289}, we define the paradifferential linearized energy
\begin{equation*}
    E_{lin}(w,r) = \int -\Re(\mathcal{L}w \cdot \bar{w})+\Im \left(w\cdot T_{J^{-\Half}}\bar{w}_\alpha \right)+ gw\cdot \bar{w} + \Im(r\cdot\bar{r}_\alpha)+ \Re \left(r\cdot \bar{r}\right)\,d\alpha.
\end{equation*}
Clearly, $E_{lin}(w,r) \approx_\CalAZ \|(w,r)\|_{\maH^\Half}^2$.
We compute its time derivative
\begin{align*}
&\frac{d}{dt} E_{lin}(w,r) = -2\Re\int \mathcal{L} w\cdot \bar{w}_t\, d\alpha +2\Im \int T_{J^{-\Half}}\bar{w}_\alpha \cdot w_t \,d\alpha  + 2g\Re \int w\cdot \bar{w}_t \, d\alpha \\
+&  2\Im \int \bar{r}_\alpha \cdot r_t \,d\alpha + 2\Re \int r\cdot \bar{r}_t \, d\alpha - \Re\int [\partial_t, \mathcal{L}]w\cdot \bar{w} \,d\alpha + \Im \int T_{\partial_t J^{-\Half}}w\cdot \bar{w}_\alpha - w\cdot T_{(J^{-\Half})_\alpha}\bar{w}_t \,d\alpha.
\end{align*} 
Note that using integration by parts, 
\begin{align*}
&-2\Re \int T_b \mathcal{L}w \cdot \bar{w}_\alpha\,d\alpha = 2\Re \int T_{b_\alpha}\mathcal{L}w \cdot \bar{w}\,d\alpha + 2\Re \int T_b \partial_\alpha(\mathcal{L}w)\cdot \bar{w}\,d\alpha \\
=&  2\Re \int T_{b_\alpha}\mathcal{L}w \cdot \bar{w}\,d\alpha + 2\Re \int [T_b\partial_\alpha, \mathcal{L}]w\cdot \bar{w}\,d\alpha + 2\Re \int T_b \partial_\alpha w \cdot \mathcal{L}\bar{w}\,d\alpha,\\
&2\Re \int w\cdot T_b\bar{w}_\alpha \, d\alpha = -2\Re \int T_{b_\alpha}w\cdot \bar{w}\, d\alpha - 2\Re \int T_{b}w_\alpha \cdot \bar{w}\, d\alpha,
\end{align*}
 we have the following identities:
\begin{align*}
 &-2\Re \int \mathcal{L}w\cdot T_{b} \bar{w}_\alpha \,d\alpha    = \Re \int T_{b_\alpha}\mathcal{L}w\cdot \bar{w} + [T_b \partial_\alpha, \mathcal{L}]w\cdot \bar{w}\,d\alpha, \\
 &2\Im \int T_{J^{-\Half}}\bar{w}_\alpha\cdot T_b w_\alpha \,d\alpha  =0,\quad 2\Re \int w\cdot T_b\bar{w}_\alpha \, d\alpha = -\Re \int T_{b_\alpha}w\cdot \bar{w}\, d\alpha, \\
 &2\Im \int \bar{r}_\alpha\cdot T_b r_\alpha \,d\alpha  =0.
\end{align*}
We use these identities to replace the time derivatives by para-material derivatives.
Adding the above identities and applying \eqref{SourceParadifferential}, we get
\begin{align*}
&\frac{d}{dt} E_{lin}(w,r) = -2\Re\int \mathcal{L} w\cdot T_{D_t}\bar{w}\, d\alpha +2\Im \int T_{J^{-\Half}}\bar{w}_\alpha \cdot T_{D_t}w \,d\alpha + 2g\Re \int w\cdot T_{D_t}\bar{w} \, d\alpha \\
&+ 2\Im \int \bar{r}_\alpha T_{D_t}r \,d\alpha + 2\Re \int \bar{r}\cdot T_{D_t}r \, d\alpha -\Re\int T_{b_\alpha}\mathcal{L}w\cdot \bar{w}\, d\alpha - \Re\int [T_{D_t}, \mathcal{L}]w\cdot \bar{w} \,d\alpha\\
&+ g\Re \int T_{b_\alpha}w\cdot \bar{w}\, d\alpha - 2\Re \int \bar{r}\cdot T_{b}r_{\alpha} \, d\alpha + \Im \int T_{\partial_t J^{-\Half}}w\cdot \bar{w}_\alpha - w\cdot T_{(J^{-\Half})_\alpha}\bar{w}_t \,d\alpha\\
=& -2\Re\int \mathcal{L} w \cdot \bar{G}+i \bar{r}_\alpha \cdot 
K \,d\alpha +2\Re\int \mathcal{L}w \cdot T_{1-Y}\bar{r}_\alpha \,d\alpha-\Re\int [T_{D_t}, \mathcal{L}]w\cdot \bar{w} \,d\alpha \\
&-2\Im \int i \bar{r}_\alpha \cdot T_{1-Y}\mathcal{L}w \,d\alpha +\Re \int \mathcal{L}w\cdot T_{b_\alpha}\bar{w}-T_{b_\alpha}\mathcal{L}w\cdot \bar{w}\,d\alpha -2\Im \int T_{J^{-\Half}}\bar{w}_\alpha \cdot T_{1-\bar{Y}} r_\alpha \,d\alpha\\
&-\Im \int T_{J^{-\Half}}\bar{w}_\alpha\cdot  T_{b_\alpha}w \,d\alpha + 2g\Im i\int \bar{r}_\alpha \cdot T_{1-Y}w \,d\alpha-2g\Re \int w\cdot T_{1-Y}\bar{r}_\alpha \, d\alpha \\
&- g\Re \int w \cdot T_{b_\alpha} \bar{w} \, d\alpha + g\Re \int T_{b_\alpha}w\cdot \bar{w}\, d\alpha-2\Im \int r\cdot T_{1-\bar{Y}}\mathcal{L}\bar{w}\,d\alpha \\
& +2\Re \int ig\bar{r}\cdot T_{1-Y}w\,d\alpha+ 2g\Re \int w \cdot \bar{G}\, d\alpha+2\Im \int T_{J^{-\Half}}\bar{w}_\alpha\cdot G \,d\alpha\\
&-2\Re\int \bar{r} \cdot K \,d\alpha - 2\Re \int \bar{r}\cdot T_{b}r_{\alpha} \, d\alpha + \Im \int T_{\partial_t J^{-\Half}}w\cdot \bar{w}_\alpha\ - w\cdot T_{(J^{-\Half})_\alpha}\bar{w}_t,d\alpha\\
=& -2\Re\int \mathcal{L} w \cdot \bar{G} -gw\cdot \bar{G}+\bar{r} \cdot 
K \,d\alpha +2\Im \int \bar{r}_\alpha \cdot 
K + T_{J^{-\Half}}\bar{w}_\alpha \cdot G -\frac{1}{2}T_{J^{-\Half}}\bar{w}_\alpha \cdot T_{b_\alpha}w\,d\alpha\\
&+2\Re\int \mathcal{L}w \cdot T_{1-Y}\bar{r}_\alpha \,d\alpha -2\Im \int i \bar{r}_\alpha T_{1-Y}\mathcal{L}w \,d\alpha +\Re \int \mathcal{L}w\cdot T_{b_\alpha - b_\alpha}\bar{w}\,d\alpha \\
&- 2\Re \int \bar{r}\cdot T_{b}r_{\alpha} \, d\alpha + \Im \int T_{\partial_t J^{-\Half}}w\cdot \bar{w}_\alpha - w\cdot T_{(J^{-\Half})_\alpha}\bar{w}_t \,d\alpha -2g\Im \int \bar{r}\cdot T_{1-Y}w\,d\alpha\\
&-\Re\int [T_{D_t}, \mathcal{L}]w\cdot \bar{w} \,d\alpha -2\Im \int T_{J^{-\Half}}\bar{w}_\alpha T_{1-\bar{Y}}\cdot r_\alpha \,d\alpha +2\Im \int(T_{1-\bar{Y}}r)_\alpha \cdot T_{J^{-\Half}}\bar{w}_\alpha \,d\alpha\\
=&2\Re \int T_{J^{-\Half}}w_\alpha \cdot \bar{G}_\alpha + gw\cdot \bar{G} -\bar{r}\cdot K \,d\alpha +2\Im\int \bar{r}_\alpha \cdot 
K +T_{J^{-\Half}}\bar{w}_\alpha \cdot G -\frac{1}{2}T_{J^{-\Half}}\bar{w}_\alpha \cdot T_{b_\alpha}w \,d\alpha \\
&- 2\Re \int \bar{r}\cdot T_{b}r_{\alpha} \, d\alpha  -2g\Im \int \bar{r}\cdot T_{1-Y}w\,d\alpha + \Im \int T_{\partial_t J^{-\Half}}w\cdot \bar{w}_\alpha - w\cdot T_{(J^{-\Half})_\alpha}\bar{w}_t \,d\alpha\\
&-2\Im \int T_{J^{-\Half}}\bar{w}_\alpha \cdot T_{1-\bar{Y}} r_\alpha \,d\alpha +2\Im \int\left(T_{1-\bar{Y}}r\right)_\alpha \cdot T_{J^{-\Half}}\bar{w}_\alpha \,d\alpha -\Re\int [T_{D_t}, \mathcal{L}]w\cdot \bar{w} \,d\alpha.
\end{align*}
The first four integrals of $\frac{d}{dt} E_{lin}(w,r)$, namely,
\begin{align*}
&2\Re \int T_{J^{-\Half}}w_\alpha \cdot \bar{G}_\alpha + gw\cdot \bar{G} -\bar{r}\cdot K \,d\alpha +2\Im\int \bar{r}_\alpha \cdot 
K + T_{J^{-\Half}}\bar{w}_\alpha \cdot G -\frac{1}{2}T_{J^{-\Half}}\bar{w}_\alpha \cdot T_{b_\alpha}w \,d\alpha \\
-& 2\Re \int \bar{r}\cdot T_{b}r_{\alpha} \, d\alpha  -2g\Im \int \bar{r}\cdot T_{1-Y}w\,d\alpha,
\end{align*}
can be directly controlled by $(1+\mathcal{A}_1^2) \| (w,r)\|^2_{\maH^{\Half}}$.
For the fifth integral term on the right, we recall \eqref{JsTimeDerivative},
\begin{equation*}
   \partial_t J^{-\Half} =  \Half T_{J^s(1-\bar{Y})}R_\alpha +\Half T_{J^s (1-Y)}\bar{R}_\alpha +E, \quad   \| E\|_{L^\infty} \lesssim_\CalAZ 1+\mathcal{A}^2_1.
\end{equation*}
This integral  can also be  controlled by $(1+\mathcal{A}_1^2) \| (w,r)\|^2_{\maH^{\Half}}$.
The next two integrals have the cancellation, 
\begin{align*}
&-2\Im \int T_{J^{-\Half}}\bar{w}_\alpha T_{1-\bar{Y}}\cdot r_\alpha \,d\alpha +2\Im \int\left(T_{1-\bar{Y}}r\right)_\alpha \cdot T_{J^{-\Half}}\bar{w}_\alpha \,d\alpha \\
=& -2\Im \int T_{J^{-\Half}}\bar{w}_\alpha T_{1-\bar{Y}}\cdot r_\alpha \,d\alpha +2\Im \int T_{1-\bar{Y}}r_\alpha \cdot T_{J^{-\Half}}\bar{w}_\alpha \,d\alpha - 2\Im \int T_{\bar{Y}_\alpha}r \cdot \bar{w}_\alpha \,d\alpha\\
=&- 2\Im \int T_{\bar{Y}_\alpha}r \cdot \bar{w}_\alpha \,d\alpha.
\end{align*}
They can also be  controlled by $(1+\mathcal{A}_1^2) \| (w,r)\|^2_{\maH^{\Half}}$.
It suffices to find an energy correction to eliminate the last integral term $-\Re\int [T_{D_t}, \mathcal{L}]w\cdot \bar{w} \,d\alpha$.
In other words, using \eqref{CommutatorBAlpha}, we need to find an energy correction $E_{cor}$ such that $|E_{cor}|\lesssim_\CalAZ \|(w,r) \|^2_{\mathcal{H}^\Half}$ and its time derivative satisfies
\begin{equation*}
    \frac{d}{dt} E_{cor} = \int T_{\Re T_{ J^{-\Half}(1-\bar{Y})}R_\alpha}w_\alpha\cdot \bar{w}_\alpha \,d\alpha + O\left((1+\mathcal{A}_1^2) \| (w,r)\|_{\maH^{\Half}}^2 \right).
\end{equation*}

To construct this delicate energy correction $E_{cor}$, we first consider a simplified toy model and its linearized system
\begin{equation}
\left\{
             \begin{array}{lr}
             \W_t  = -R_\alpha &\\
           R_t  = - i\W_{\alpha \alpha},&  
             \end{array}
\right.
\quad 
\left\{
             \begin{array}{lr}
             w_t  = -r_\alpha &\\
           r_t  = - iw_{\alpha \alpha}.&  
             \end{array}
\right.  \label{ToyModel}
\end{equation}
Our aim is to find an energy $\tilde{E}_{cor}$ such that for some constant $C$ that depends on $\CalAZ$,
\begin{equation*}
    \frac{d}{dt} \tilde{E}_{cor} = \Re \int T_{ R_\alpha}w_\alpha\cdot \bar{w}_\alpha \,d\alpha + C(1+\mathcal{A}_1^2) \| (w,r)\|_{\maH^{\Half}}^2.
\end{equation*}
Our candidate for the  energy correction is of the form
\begin{equation*}
    \tilde{E}^3_{cor} = \int A(R, w, \bar{r}) + B(R, r, \bar{w}) + C(\W, w, \bar{w}) + D(\W, r, \bar{r}) \, d\alpha + \text{complex conjugates},
\end{equation*}
where $A,B,C,D$ are paradifferential cubic forms. 
Here, we  write $\mathfrak{a}(\xi, \eta, \zeta)$ for the symbol of $A(R, w, \bar{r})$, and the other three symbols are defined in the same way.
The time derivative of this cubic energy correction is given by
\begin{align*}
\frac{d}{dt}  \tilde{E}^3_{cor} &= \int A(R, w, i\bar{w}_{\alpha \alpha}) + B(R, -iw_{\alpha\alpha}, \bar{w})+ C(-R_\alpha, w, \bar{w}) \, d\alpha \\
&+ \int B(-i \W_{\alpha \alpha}, r, \bar{w}) + C(\W, -r_{\alpha}, \bar{w})+ D(\W, r, i\bar{w}_{\alpha\alpha}) \, d\alpha \\
&+ \int A(-i\W_{\alpha \alpha }, w, \bar{r}) + C(\W, w, -\bar{r}_\alpha)+ D(\W, -iw_{\alpha \alpha}, \bar{r}) \, d\alpha \\
&+ \int A(R, -r_\alpha, \bar{r}) + B(R, r, -\bar{r}_\alpha)+ D(-R_\alpha, r, \bar{r}) \, d\alpha + \text{complex conjugates}.
\end{align*}
Taking the Fourier transform of the integral and comparing the symbols for each cubic integral, the symbols  $\mathfrak{a}, \mathfrak{b}, \mathfrak{c}, \mathfrak{d}$ solve the algebraic linear system
\begin{equation*}
\left\{
    \begin{array}{lr}
    \zeta^2 \mathfrak{a} - \eta^2 \mathfrak{b} + \xi \mathfrak{c} = -\frac{1}{2}\xi \eta\zeta\chi_1(\xi, \eta) &\\
    \xi^2 \mathfrak{b} - \eta \mathfrak{c} - \zeta^2 \mathfrak{d} = 0  &\\
    \xi^2 \mathfrak{a} + \zeta \mathfrak{c} + \eta^2 \mathfrak{d} = 0  &\\
     -\eta \mathfrak{a} +\zeta \mathfrak{b} - \xi \mathfrak{d} =  0,&  
    \end{array}
\right.
\end{equation*}
where the $\chi_1(\theta_1, \theta_2)$ is a non-negative smooth bump function defined in \eqref{ChiOnelh}.
Using integration by parts, the derivative on the third factor is equal to the minus of derivative of the first two factors in the integral, which shows the symbolic relation $\zeta = \xi + \eta$.
This algebraic system has the solution
\begin{align*}
&\mathfrak{a} = \dfrac{-3\zeta^3 \chi_1(\xi, \eta)}{2(9\xi^2 + 14\xi \eta + 9 \eta^2)}, \quad \mathfrak{b} = \dfrac{-\zeta(2\xi^2 + 3\xi \eta + 3\eta^2)\chi_1(\xi, \eta)}{2(9\xi^2 + 14\xi \eta + 9 \eta^2)}, \\
&\mathfrak{c} = \dfrac{\xi \zeta (3\xi^2 + 3\xi \eta + 2\eta^2)\chi_1(\xi, \eta)}{2(9\xi^2 + 14\xi \eta + 9 \eta^2)}, \quad \mathfrak{d} = \dfrac{-\xi\zeta^2\chi_1(\xi, \eta)}{9\xi^2 + 14\xi \eta + 9 \eta^2}.
\end{align*}
Since $|\xi|\ll |\eta|$,  $\zeta \approx \eta$, so that at the leading order
\begin{equation*}
\mathfrak{a} \approx -\frac{1}{6}\eta \chi_1(\xi, \eta), \quad \mathfrak{b} \approx -\frac{1}{6}\eta \chi_1(\xi, \eta), \quad \mathfrak{c} \approx \frac{1}{9}\xi\zeta \chi_1(\xi, \eta), \quad \mathfrak{d} \approx -\frac{1}{9}\xi \chi_1(\xi, \eta).
\end{equation*}
Hence, adding the complex conjugates, the leading part of the cubic energy correction is given by
\begin{equation*}
 \tilde{E}^3_{high} = -\frac{2}{3}\Im \int T_{\Re R}w_\alpha \cdot\bar{r}- \frac{1}{3}T_{\Im \W_\alpha}w \cdot\bar{w}_\alpha +\frac{1}{3}T_{\Im \W_\alpha}r\cdot \bar{r} \, d\alpha.  
\end{equation*}
And the remainder part has the form
\begin{equation*}
 E^3_{low} = \Re \int L_0^{lhh}(R_\alpha, w, \bar{r}) +  L_0^{lhh}(R_\alpha, r, \bar{w}) + L_0^{lhh}(\W_\alpha, w, \bar{w}) + L_{-1}^{lhh}(\W_\alpha, r, \bar{r})\,d\alpha, 
\end{equation*}
where $L_k^{lhh}$ is some cubic forms of order $k$ such that the first component is at low frequency compared to the second component.
The energy correction $\tilde{E}^3_{cor} = \tilde{E}^3_{high}+\tilde{E}^3_{low}$ removes the $\Re \int T_{ R_\alpha}w_\alpha\cdot \bar{w}_\alpha \,d\alpha$ integral for the model problem.

We now go back to the paradifferential system \eqref{SourceParadifferential}, which can be rewritten as
\begin{equation*}
\left\{
             \begin{array}{lr}
             w_t= -T_{1-\bar{Y}} r_\alpha -T_{b}w_\alpha - \Half T_{ b_\alpha }w  + G &\\
           r_t =- i T_{(1-Y)J^{-\Half}}w_{\alpha\alpha} -T_{b}r_\alpha -iT_{(1-Y)(J^{-\Half})_\alpha}w_\alpha+ K.&  
             \end{array}
\right.
\end{equation*}
The original unknowns $(\W, R)$ satisfy \eqref{WParaMat} and \eqref{RParaMat}, which we recall here
\begin{equation*}
\left\{
             \begin{array}{lr}
             \W_t= -T_{(1+\W)(1-\bar{Y})} R_\alpha -T_{b}\W_\alpha  + \tilde{G} &\\
           R_t =- i T_{(1-Y)^2J^{-\Half}}\W_{\alpha\alpha} -T_{b}R_\alpha + \tilde{K},&  
             \end{array}
\right.
\end{equation*}
where $(\tilde{G}, \tilde{K})$ satisfies
\begin{equation*}
    \|(\tilde{G}, \tilde{K}) \|_{C^\Half_{*}\times C^\epsilon_{*}} \lesssim_\CalAZ 1+ \mathcal{A}^2_1.
\end{equation*}
Compared to the model systems \eqref{ToyModel}, the full systems are quasilinear rather than linear, and they also have lower order terms.

We define the cubic part of the energy correction to be $E^3_{cor}: = E^3_{high}+ E^3_{low}$, where 
\begin{equation*}
 E^3_{high} = -\frac{2}{3}\Im \int T_{\Re R}w_\alpha \cdot\bar{r}- \frac{1}{3}T_{\Im \W_\alpha}w \cdot T_{(1-Y)J^{-\Half}}\bar{w}_\alpha +\frac{1}{3}T_{\Im \W_\alpha}r\cdot T_{1-Y}\bar{r} \, d\alpha. 
 \end{equation*}
$E^3_{high}$ is the quasilinear modification of $\tilde{E}^3_{high}$ so that it takes into account the quasilinear nature of the system.
For the bound for the energy, clearly we have
\begin{equation*}
    |E^3_{cor}| \lesssim_\CalAZ \| (w,r)\|_{\maH^{\Half}}.
\end{equation*}
Then we consider the time derivative of $E^3_{cor}$.
When the time derivative applies to a para-coefficient, the term is perturbative and is controlled by $(1+\mathcal{A}_1^2) \| (w,r)\|^2_{\maH^{\Half}}$.
When the time derivative acts on $w$, $r$ terms, we can freely  combine or switch the order of para-coefficients using the para-commutators or para-products rules in Lemma \ref{t:ParaCoefficient}.
As a consequence, the time derivative of $E^3_{cor}$ removes the desired cubic terms.
Due to the lower order terms in the systems,  the time derivative of $E^3_{cor}$ also introduces extra quartic integral terms.
Fortunately, these quartic terms are perturbative after the cancellation.
For instance, using integration by parts,
\begin{align*}
 &\left|\frac{2}{3}\Im \int T_{\Re R} (T_b w_\alpha)_\alpha \cdot\bar{r} + T_{\Re R}w_\alpha \cdot T_b \bar{r}_\alpha  \, d\alpha \right|\\
 =& \left|\frac{2}{3}\Im \int T_{\Re R} T_b w_\alpha \cdot\bar{r}_\alpha - T_{\Re R}w_\alpha \cdot T_b \bar{r}_\alpha + T_{\Re R_\alpha} T_b w_\alpha \cdot\bar{r} \, d\alpha\right| \\
 = & \left|\frac{2}{3}\Im \int (T_{\Re R} T_b -T_b T_{\Re R})\bar{r}_\alpha \cdot w_\alpha+  T_b w_\alpha \cdot T_{\Re R_\alpha}\bar{r} \, d\alpha \right| \\
 \lesssim& (1+\mathcal{A}_1^2) \| (w,r)\|_{\maH^{\Half}}^2.
\end{align*}
In other words, for some constant $C$ that depends on $\CalAZ$,
\begin{equation*}
    \frac{d}{dt}E^3_{cor} = \int T_{\Re T_{ J^{-\Half}(1-\bar{Y})}R_\alpha}w_\alpha\cdot \bar{w}_\alpha \,d\alpha + C(1+\mathcal{A}_1^2) \| (w,r)\|_{\maH^{\Half}}^2.
\end{equation*}
Therefore, choosing energy functional
\begin{equation*}
    E^{\Half, para}_{lin}(w,r): = E_{lin}(w,r)+ E^3_{cor}(w,r)
\end{equation*}
gives the desired norm equivalence and energy estimates.
This finishes the proof of Proposition \ref{t:wellposedflow}.

\subsection{Normal form analysis for the balanced source terms} \label{s:NormalAnalysis}
In this subsection, we compute the normal form transformations that eliminate the balanced part of the non-perturbative quadratic source terms of  $(\mathcal{G}^\sharp, \mathcal{K}^\sharp)$ in \eqref{ParadifferentialLinearEqn}, and show that these balanced normal form  transformations are bounded.
We will first compute the normal form transformations that remove the non-perturbative quadratic parts of $(\nP\mathcal{G}_0, \nP\mathcal{K}_0)$.
Then we will perform the normal form analysis for $(\nP\mathcal{G}_1, \nP\mathcal{K}_1)$.

\subsubsection{Balanced quadratic normal form analysis for $(\nP\mathcal{G}_0, \nP\mathcal{K}_0)$} 
On account of the quadratic terms computed in Lemma \ref{t:QuadraticGKZero},  we seek to show that there exist balanced quadratic normal form variables $(w_{1}^{bal}, r_{1}^{bal})$ such that
\begin{align*}
&\partial_t w_{1}^{bal}+T_{1-\bar{Y}} \partial_\alpha r_{1}^{bal} + \text{cubic and higher terms} \\
=&  T_{J^{-1}} \nP \Pi(\bar{r}_\alpha, \W) -  T_{(1-\bar{Y})^2(1+\W)} \nP \Pi(\bar{w}_\alpha, R) + G, \\
&\partial_t r_{1}^{bal} + i T_{(1-Y)J^{-\Half}} \partial_\alpha^2 w_{1}^{bal}+ \text{cubic and higher terms} \\
=& T_{1-\bar{Y}}\nP\Pi(\bar{r}_\alpha, R) 
  +\frac{i}{2} T_{J^{-\frac{3}{2}}} \nP \Pi(\bar{w}_{\alpha \alpha},\W)- \frac{i}{2} T_{J^{-\frac{3}{2}}} \nP\Pi(\bar{w}_\alpha, \W_\alpha)+K,
\end{align*}
where $(G,K)$ are perturbative source terms that satisfy \eqref{HalfGK}.
We consider normal form transformations as the sum of bilinear forms of the following type:
\begin{equation} \label{wrBalOne}
\begin{aligned}
w_{1}^{bal} &= B^{bal}_1\left(\bar{w}, T_{1-\bar{Y}}\W \right) + C^{bal}_1\left(\bar{r}, T_{J^{-\frac{1}{2}}(1+\W)^2}R \right),\\
r_{1}^{bal} &= A^{bal}_1 \left(\bar{r}, T_{1-Y}\W \right) + D^{bal}_1\left(\bar{w}, T_{(1-\bar{Y})(1+\W)}R\right).
\end{aligned}
\end{equation}
Direct computation gives 
\begin{align*}
&\partial_t w_{1}^{bal}+T_{1-\bar{Y}} \partial_\alpha r_{1}^{bal} + \text{cubic and higher terms} \\
=& T_{J^{-1}}\partial_\alpha A^{bal}_1(\bar{r}, \W) -T_{J^{-1}}B^{bal}_1(\bar{r}_\alpha, \W) -  T_{J^{-1}}C^{bal}_1(\bar{r}, i\W_{\alpha \alpha}) \\
 -&T_{(1-\bar{Y})^2(1+\W)}B^{bal}_1(\bar{w}, R_\alpha)+ T_{(1-\bar{Y})^2(1+\W)}C^{bal}_1(i\bar{w}_{\alpha \alpha}, R) + T_{(1-\bar{Y})^2(1+\W)}\partial_\alpha D^{bal}_1(\bar{w}, R),\\
&\partial_t r_{1}^{bal} + i T_{(1-Y)J^{-\Half}} \partial_\alpha^2 w_{1}^{bal}+ \text{cubic and higher  terms} \\
=& -T_{1-\bar{Y}}A^{bal}_1(\bar{r}, R_\alpha) + iT_{1-\bar{Y}}\partial_\alpha^2 C^{bal}_1(\bar{r}, R) - T_{1-\bar{Y}}D^{bal}_1(\bar{r}_\alpha, R)\\
+& iT_{J^{-\frac{3}{2}}}A^{bal}_1(\bar{w}_{\alpha \alpha}, \W) +iT_{J^{-\frac{3}{2}}}\partial_\alpha^2 B^{bal}_1(\bar{w}, \W) - iT_{J^{-\frac{3}{2}}}D^{bal}_1(\bar{w}, \W_{\alpha \alpha}).
\end{align*}
We write $\mathfrak{a}^{bal}_1(\eta, \zeta)$ for the symbol of $A^{bal}_1(\bar{r}, T_{1-Y}\W)$, and similarly for other balanced bilinear forms.
In order to match the balanced source terms, one has the following algebraic system for the symbols: 
\begin{equation*}
\left\{
    \begin{array}{lr}
    (\zeta -\eta) \mathfrak{a}^{bal}_{1} + \eta \mathfrak{b}^{bal}_{1} + \zeta^2 \mathfrak{c}^{bal}_{1} = -\eta \chi_2(\eta, \zeta)1_{\zeta<\eta} &\\
    \zeta \mathfrak{b}^{bal}_{1} + \eta^2 \mathfrak{c}^{bal}_{1} - (\zeta-\eta) \mathfrak{d}^{bal}_{1} =  -\eta \chi_2(\eta, \zeta) 1_{\zeta<\eta}   &\\
    \zeta \mathfrak{a}^{bal}_{1} + ( \zeta -\eta)^2 \mathfrak{c}^{bal}_{1} - \eta \mathfrak{d}^{bal}_{1} = \eta\chi_2(\eta, \zeta)1_{\zeta<\eta}  &\\
     \eta^2 \mathfrak{a}^{bal}_{1} +( \zeta -\eta)^2 \mathfrak{b}^{bal}_{1} - \zeta^2 \mathfrak{d}^{bal}_{1} =  \frac{1}{2}\eta(\eta +\zeta) \chi_2(\eta, \zeta)1_{\zeta<\eta},&  
    \end{array}
\right.
\end{equation*}
where the symbol $\chi_{2}(\eta, \zeta)$ is defined in \eqref{ChiTwohh} to select the balanced frequencies, and $1_{\zeta<\eta}$ is the symbol for holomorphic cutoff.
The solutions of this system are given by
\begin{align*}
&\mathfrak{a}^{bal}_1(\eta, \zeta) = \dfrac{3(2\eta^2-\eta \zeta +2\zeta^2) \chi_2(\eta, \zeta)1_{\zeta<\eta}}{2(4\eta^2 - 4\eta\zeta + 9 \zeta^2)}, \quad \mathfrak{b}^{bal}_1(\eta, \zeta) = -\dfrac{(2\eta^2 +\eta \zeta + 9\zeta^2)\chi_2(\eta, \zeta)1_{\zeta<\eta}}{2(4\eta^2 - 4\eta\zeta + 9 \zeta^2)}, \\
&\mathfrak{c}^{bal}_1(\eta, \zeta) = -\dfrac{ 3\zeta \chi_2(\eta, \zeta)1_{\zeta<\eta}}{(4\eta^2 - 4\eta\zeta + 9 \zeta^2)}, \quad \mathfrak{d}^{bal}_1(\eta, \zeta) = -\dfrac{(8\eta^2-8\eta \zeta + 9\zeta^2)\chi_2(\eta, \zeta)1_{\zeta<\eta}}{2(4\eta^2 - 4\eta\zeta + 9 \zeta^2)}.
\end{align*}
The leading terms of these symbols are
\begin{equation*}
\mathfrak{a}^{bal}_1 \approx \frac{1}{3}\chi_{2}(\eta, \zeta)1_{\zeta<\eta}, \quad \mathfrak{b}^{bal}_1 \approx -\frac{1}{2}\chi_{2}(\eta, \zeta)1_{\zeta<\eta},  \quad \mathfrak{c}^{bal}_1 \approx -\frac{1}{3\zeta}\chi_{2}(\eta, \zeta)1_{\zeta<\eta}, \quad \mathfrak{d}^{bal}_1 \approx -\frac{1}{2}\chi_{2}(\eta, \zeta)1_{\zeta<\eta}.   
\end{equation*}
From the leading terms of these balanced symbols, we get that using Lemma \ref{t:SymbolPara}, balanced quadratic normal form variables $(w_{1}^{bal}, r_{1}^{bal})$ are bounded in the sense that
\begin{align*}
 \|(w_{1}^{bal}, r_{1}^{bal}) \|_{\mathcal{H}^\Half} \lesssim \CalAZ \|(w,r) \|_{\mathcal{H}^\Half}.
\end{align*}

\subsubsection{Balanced quadratic normal form analysis for $(\nP\mathcal{G}_1, \nP\mathcal{K}_1)$}
In view of the computations for quadratic terms  Lemma \ref{t:QuadraticGKOne}, we aim to show that there exist balanced quadratic normal form variables $(w_{2}^{bal}, r_{2}^{bal})$ such that
\begin{align*}
&\partial_t w_{2}^{bal}+T_{1-\bar{Y}} \partial_\alpha r_{2}^{bal} + \text{cubic higher terms} \\
=&  T_{1-\bar{Y}} \partial_\alpha \Pi(R, w) + \nP \Pi(T_{1-Y}\bar{R},w_\alpha) - \nP \Pi(T_{(1-\bar{Y})^2}\bar{\W},r_\alpha) + G, \\
&\partial_t r_{2}^{bal} + i T_{J^{-\Half}(1-Y)} \partial_\alpha^2 w_{2}^{bal}+ \text{cubic and higher terms} \\
=&    \Pi(T_{1-\bar{Y}}R, r_\alpha)   - \frac{3}{2}i T_{J^{-\Half}(1-Y)^2}\partial_\alpha \Pi(\W, w_{\alpha}) -  i T_{J^{-\Half}(1-Y)^2}\partial_\alpha \Pi(\W_{\alpha}, w)\\
&+ \nP \Pi(T_{1-Y}\bar{R}, r_\alpha) -\frac{i}{2}T_{J^{-\frac{3}{2}}} \nP \Pi(\bar{\W},w_{\alpha \alpha}) +\frac{i}{2}T_{J^{-\frac{3}{2}}} \nP \Pi(\bar{\W}_\alpha, w_\alpha)+K,
\end{align*}
where $(G,K)$ are perturbative source terms that satisfy \eqref{HalfGK}.
We choose normal form transformations as the sum of bilinear forms of the following type:
\begin{equation} \label{wrBalTwo}
\begin{aligned}
w_{2}^{bal} &= B^{bal}_{2, h}(T_{1-Y}\W, w) + C^{bal}_{2, h}(T_{J^{\frac{3}{2}}(1-\bar{Y})^2}R, r)+ B^{bal}_{2, a}(T_{1-\bar{Y}}\bar{\W}, w) + C^{bal}_{2, a}( T_{J^{\Half}}\bar{R}, r),\\
r_{2}^{bal} &= A^{bal}_{2,h}(T_{1-Y}\W, r) + D^{bal}_{2,h}(R, w)+ A^{bal}_{2,a}(T_{1-\bar{Y}}\bar{\W}, r) + D^{bal}_{2,a}(T_{(1+\bar{\W})(1-Y)}\bar{R}, w).
\end{aligned}
\end{equation}
Then one can compute
\begin{align*}
&\partial_t w_{2}^{bal}+T_{1-\bar{Y}} \partial_\alpha r_{2}^{bal} + \text{cubic and higher  terms} \\
=& -T_{1-\bar{Y}}B^{bal}_{2, h}(R_\alpha, w) -i T_{1-\bar{Y}}C^{bal}_{2, h}(R, w_{\alpha \alpha}) + T_{1-\bar{Y}}\partial_\alpha D^{bal}_{2, h}(R, w)\\
&+ T_{J^{-1}}\partial_\alpha A^{bal}_{2, h}(\W, r) - T_{J^{-1}} B^{bal}_{2, h}(\W, r_\alpha) -i T_{J^{-1}} C^{bal}_{2, h}(\W_{\alpha \alpha}, r)\\
& -T_{1-Y}B^{bal}_{2, a}(\bar{R}_\alpha, w) -i T_{1-Y}C^{bal}_{2, a}(\bar{R}, w_{\alpha \alpha}) + T_{1-Y}\partial_\alpha D^{bal}_{2, a}(\bar{R}, w)\\
&+ T_{(1-\bar{Y})^2}\partial_\alpha A^{bal}_{2, a}(\bar{\W}, r) - T_{(1-\bar{Y})^2} B^{bal}_{2, a}(\bar{\W}, r_\alpha) +i T_{(1-\bar{Y})^2} C^{bal}_{2, a}(\bar{\W}_{\alpha \alpha}, r),\\
&\partial_t r_{2}^{bal} + i T_{(1-Y)J^{-\Half}} \partial_\alpha^2 w_{2}^{bal}+ \text{cubic and higher  terms} \\
=& -T_{1-\bar{Y}} A^{bal}_{2, h}(R_\alpha, r) + i\partial_\alpha^2 T_{1-\bar{Y}} C^{bal}_{2, h}(R, r) - T_{1-\bar{Y}} D^{bal}_{2, h}(R, r_\alpha)\\
& -iT_{J^{-\Half}(1-Y)^2} A^{bal}_{2, h}(\W, w_{\alpha \alpha}) + i T_{J^{-\Half}(1-Y)^2} \partial_\alpha^2 B^{bal}_{2, h}(\W, w) - iT_{J^{-\Half}(1-Y)^2} D^{bal}_{2, h}(\W_{\alpha \alpha}, w) \\
& -T_{1-Y} A^{bal}_{2, a}(\bar{R}_\alpha, r) + i\partial_\alpha^2 T_{1-Y} C^{bal}_{2, a}(\bar{R}, r) - T_{1-Y} D^{bal}_{2, a}(\bar{R}, r_\alpha)\\
& -iT_{J^{-\frac{3}{2}}} A^{bal}_{2, a}(\bar{\W}, w_{\alpha \alpha}) + i T_{J^{-\frac{3}{2}}} \partial_\alpha^2 B^{bal}_{2, a}(\bar{\W}, w) + iT_{J^{-\frac{3}{2}}} D^{bal}_{2, a}(\bar{\W}_{\alpha \alpha}, w).
\end{align*}
We write $\mathfrak{a}^{bal}_{2,h}(\xi, \eta)$ for the symbol of $A^{bal}_{2,h}(T_{1-Y}\W, r)$, and $\mathfrak{a}^{bal}_{2,a}(\eta, \zeta)$ for the symbol of $A^{bal}_{2,a}(T_{1-\bar{Y}}\bar{\W}, r)$.
Other symbols for balanced bilinear forms are defined similarly.
Taking the Fourier transform and matching the quadratic source terms on the Fourier side, the balanced holomorphic symbols $\mathfrak{a}^{bal}_{2,h}, \mathfrak{b}^{bal}_{2,h}, \mathfrak{c}^{bal}_{2,h}, \mathfrak{d}^{bal}_{2,h}$ solve the linear system:
\begin{equation*}
\left\{
    \begin{array}{lr}
    (\xi+\eta) \mathfrak{a}^{bal}_{2,h} - \eta \mathfrak{b}^{bal}_{2,h} + \xi^2 \mathfrak{c}^{bal}_{2,h} = 0 &\\
    \xi \mathfrak{b}^{bal}_{2,h} - \eta^2 \mathfrak{c}^{bal}_{2,h} - (\xi+\eta) \mathfrak{d}^{bal}_{2,h} = -(\xi+\eta) \chi_2(\xi, \eta)  &\\
    \xi \mathfrak{a}^{bal}_{2,h} + (\xi + \eta)^2 \mathfrak{c}^{bal}_{2,h} + \eta \mathfrak{d}^{bal}_{2,h} = -\eta\chi_2(\xi, \eta)  &\\
     \eta^2 \mathfrak{a}^{bal}_{2,h} -(\xi+ \eta)^2 \mathfrak{b}^{bal}_{2,h} + \xi^2 \mathfrak{d}^{bal}_{2,h} =(\xi+\eta)(\xi + \frac{3}{2}\eta) \chi_2(\xi, \eta) .&  
    \end{array}
\right.
\end{equation*}
The expressions for balanced holomorphic symbols are
 \begin{align*}
&\mathfrak{a}^{bal}_{2,h}(\xi, \eta) = \dfrac{(2\xi^3- 9\xi^2\eta -16\xi\eta^2 -9\eta^3) \chi_2(\xi, \eta)}{2\xi (9\xi^2 + 14\xi \eta + 9 \eta^2)}, \\
&\mathfrak{b}^{bal}_{2,h}(\xi, \eta) = -\dfrac{(15\xi^3+ 31\xi^2\eta +25\xi\eta^2 +9 \eta^3)\chi_2(\xi, \eta)}{2 \xi(9\xi^2 + 14\xi \eta + 9 \eta^2)}, \\
&\mathfrak{c}^{bal}_{2,h}(\xi, \eta) = -\dfrac{ (\xi^2 + 4\xi \eta +3\eta^2)\chi_2(\xi, \eta)}{\xi 
(9\xi^2 + 14\xi \eta + 9 \eta^2)}, \\
&\mathfrak{d}^{bal}_{2,h}(\xi, \eta) = \dfrac{(3\xi^3+ 12\xi^2\eta +11\xi\eta^2 +6 \eta^3)\chi_2(\xi, \eta)}{2\xi(9\xi^2 + 14\xi \eta + 9 \eta^2)}.
\end{align*}   
The leading terms of these symbols are
\begin{equation*}
\mathfrak{a}^{bal}_{2,h} \approx -\frac{1}{2}\chi_{2}(\xi, \eta), \quad \mathfrak{b}^{bal}_{2,h} \approx -\frac{5}{4}\chi_{2}(\xi, \eta),  \quad \mathfrak{c}^{bal}_{2,h} \approx -\frac{1}{4\xi}\chi_{2}(\xi, \eta), \quad \mathfrak{d}^{bal}_{2,h} \approx \frac{1}{2}\chi_{2}(\xi, \eta).   
\end{equation*}

Similarly, the balanced mixed symbols $\mathfrak{a}^{bal}_{2,a}, \mathfrak{b}^{bal}_{2,a}, \mathfrak{c}^{bal}_{2,a}, \mathfrak{d}^{bal}_{2,a}$ solve the linear system: 
\begin{equation*}
\left\{
    \begin{array}{lr}
    (\zeta -\eta) \mathfrak{a}^{bal}_{2,a} - \zeta \mathfrak{b}^{bal}_{2,a} - \eta^2 \mathfrak{c}^{bal}_{2,a} =- \zeta \chi_2(\eta, \zeta)1_{\zeta<\eta}  &\\
    \eta \mathfrak{b}^{bal}_{2,a} + \zeta^2 \mathfrak{c}^{bal}_{2,a} + (\zeta-\eta) \mathfrak{d}^{bal}_{2,a} =  \frac{1}{2}\zeta \chi_2(\eta, \zeta)1_{\zeta<\eta}   &\\
    \eta \mathfrak{a}^{bal}_{2,a} - ( \zeta -\eta)^2 \mathfrak{c}^{bal}_{2,a} - \zeta \mathfrak{d}^{bal}_{2,a} = \zeta\chi_2(\eta, \zeta)1_{\zeta<\eta}  &\\
     \zeta^2 \mathfrak{a}^{bal}_{2,a} -( \zeta -\eta)^2 \mathfrak{b}^{bal}_{2,a} - \eta^2 \mathfrak{d}^{bal}_{2,a} =  \frac{1}{2}\zeta(\eta +\zeta) \chi_2(\eta, \zeta)1_{\zeta<\eta},&  
    \end{array}
\right.
\end{equation*}
Then solving the system, these balanced mixed symbols are given by
\begin{align*}
&\mathfrak{a}^{bal}_{2,a}(\eta, \zeta) = \dfrac{\zeta(8\eta^2-8\eta \zeta +9\zeta^2) \chi_2(\eta, \zeta)1_{\zeta<\eta}}{2\eta(4\eta^2 - 4\eta\zeta + 9 \zeta^2)}, \quad \mathfrak{b}^{bal}_{2,a}(\eta, \zeta) = \dfrac{\zeta(2\eta^2 +\eta \zeta + 9\zeta^2)\chi_2(\eta, \zeta)1_{\zeta<\eta}}{2\eta(4\eta^2 - 4\eta\zeta + 9 \zeta^2)}, \\
&\mathfrak{c}^{bal}_{2,a}(\eta, \zeta) = \dfrac{ 3\zeta^2 \chi_2(\eta, \zeta)1_{\zeta<\eta}}{\eta(4\eta^2 - 4\eta\zeta + 9 \zeta^2)}, \quad \mathfrak{d}^{bal}_{2,a}(\eta, \zeta) = -\dfrac{\zeta(2\eta^2-\eta \zeta + 2\zeta^2)\chi_2(\eta, \zeta)1_{\zeta<\eta}}{2\eta(4\eta^2 - 4\eta\zeta + 9 \zeta^2)}.
\end{align*}
The leading terms of these symbols are
\begin{equation*}
\mathfrak{a}^{bal}_{2,a} \approx \frac{1}{2}\chi_{2}(\eta, \zeta)1_{\zeta<\eta}, \quad \mathfrak{b}^{bal}_{2,a} \approx \frac{1}{2}\chi_{2}(\eta, \zeta)1_{\zeta<\eta},  \quad \mathfrak{c}^{bal}_{2,a} \approx \frac{1}{3\eta}\chi_{2}(\eta, \zeta)1_{\zeta<\eta}, \quad \mathfrak{d}^{bal}_{2,a} \approx \frac{1}{9}\chi_{2}(\eta, \zeta)1_{\zeta<\eta}.   
\end{equation*}
From the leading terms of these balanced holomorphic or mixed symbols, we get that using Lemma \ref{t:SymbolPara}, balanced quadratic normal form variables $(w_{2}^{bal}, r_{2}^{bal})$ are bounded in the sense that
\begin{equation*}
 \|(w_{2}^{bal}, r_{2}^{bal}) \|_{\mathcal{H}^\Half} \lesssim \CalAZ \|(w,r) \|_{\mathcal{H}^\Half}.
\end{equation*}

To conclude this subsection, we have shown the following result:
\begin{lemma} \label{t:wrbalNF}
Suppose that $(w,r)$ solve the linearized water wave system \eqref{ParadifferentialLinearEqn}, then the normal form variables $(w_{NF}^{bal}, r_{NF}^{bal}) : = (w,r) + (w_{1}^{bal}, r_{1}^{bal}) + (w_{2}^{bal}, r_{2}^{bal})$, where $(w_{1}^{bal}, r_{1}^{bal})$ are defined in \eqref{wrBalOne} and  $(w_{2}^{bal}, r_{2}^{bal})$ are defined in \eqref{wrBalTwo}, satisfy the norm equivalence:
\begin{equation*}
 \|(w_{NF}^{bal}, r_{NF}^{bal})-(w,r) \|_{\mathcal{H}^\Half} \lesssim \CalAZ \|(w,r) \|_{\mathcal{H}^\Half}.
\end{equation*}
Moreover, $(w_{NF}^{bal}, r_{NF}^{bal})$ solve \eqref{SourceParadifferential} with source terms $(G, K)$ having no non-perturbative balanced quadratic terms. 
\end{lemma}

\subsection{Quadratic normal form analysis for low-high source terms.} \label{s:PartialNormal}
In this subsection, we take into account the quadratic low-high inhomogeneous source terms of $(\mathcal{G}^\sharp, \mathcal{K}^\sharp)$ in \eqref{ParadifferentialLinearEqn}.
These low-high quadratic source terms are divided into three parts:   low-high quadratic source terms of $(\nP\mathcal{G}_1, \nP\mathcal{K}_1)$ where  linearized variables are at high frequencies,  the other low-high quadratic source terms of $(\nP\mathcal{G}_1, \nP\mathcal{K}_1)$ where  linearized variables are at low frequencies, and  low-high quadratic source terms of $(\nP\mathcal{G}_0, \nP\mathcal{K}_0)$.
We will construct normal form transformations to eliminate each part of these low-high quadratic source terms.

\subsubsection{Normal forms for low-high quadratic terms in $(\nP\mathcal{G}_1, \nP\mathcal{K}_1)$ part one}
Here we  show that there exist low-high quadratic normal
form variables $(w_1^{lh}, r_1^{lh})$ that can remove the part of low-high quadratic terms in $(\nP\mathcal{G}_1, \nP\mathcal{K}_1)$ where linearized variables are at high frequencies, namely, 
\begin{align*}
&\partial_t w_1^{lh}+T_{1-\bar{Y}} \partial_\alpha r_1^{lh} + \text{cubic and higher terms} \\
=& -\Half T_{T_{1-Y}\bar{R}_\alpha}w+\frac{1}{2}T_{T_{1-\bar{Y}}R_\alpha}w  + G, \\
&\partial_t r_1^{lh} + i T_{(1-Y)J^{-\Half}} \partial_\alpha^2 w_1^{lh}+ \text{cubic and higher terms} \\
=&  iT_{J^{-\frac{3}{2}}}T_{\bar{\W}_\alpha} w_\alpha -iT_{J^{-\Half}(1-Y)^2}\partial_\alpha T_{\W_\alpha}w +K,
\end{align*}
where $(G,K)$ are perturbative source terms that satisfy \eqref{HalfGK}.

We choose normal form transformations as the sum of bilinear forms of the following type:
\begin{equation*} 
\begin{aligned}
w^{lh}_1 &= B^{lh}_{h,1}(T_{1-Y}\W, w) + C^{lh}_{h,1}(T_{J^{\frac{3}{2}}(1-\bar{Y})^2}R, r)+ B^{lh}_{a,1}(T_{1-\bar{Y}}\bar{\W}, w) + C^{lh}_{a,1}( T_{J^{\Half}}\bar{R}, r),\\
r^{lh}_1 &= A^{lh}_{h,1}(T_{1-Y}\W, r) + D^{lh}_{h,1}(R, w)+ A^{lh}_{a,1}(T_{1-\bar{Y}}\bar{\W}, r) + D^{lh}_{a,1}(T_{(1+\bar{\W})(1-Y)}\bar{R}, w).
\end{aligned}
\end{equation*}
Plugging these bilinear forms into the system, 
\begin{align*}
&\partial_t w_1^{lh}+T_{1-\bar{Y}} \partial_\alpha r_1^{lh} + \text{cubic and higher terms} \\
=& -T_{1-\bar{Y}}B^{lh}_{h,1}(R_\alpha, w) -i T_{1-\bar{Y}}C^{lh}_{ h,1}(R, w_{\alpha \alpha}) + T_{1-\bar{Y}}\partial_\alpha D^{lh}_{h,1}(R, w)\\
&+ T_{J^{-1}}\partial_\alpha A^{lh}_{h,1}(\W, r) - T_{J^{-1}} B^{lh}_{h,1}(\W, r_\alpha) -i T_{J^{-1}} C^{lh}_{ h,1}(\W_{\alpha \alpha}, r)\\
& -T_{1-Y}B^{lh}_{a,1}(\bar{R}_\alpha, w) -i T_{1-Y}C^{lh}_{a,1}(\bar{R}, w_{\alpha \alpha}) + T_{1-Y}\partial_\alpha D^{lh}_{a,1}(\bar{R}, w)\\
&+ T_{(1-\bar{Y})^2}\partial_\alpha A^{lh}_{a,1}(\bar{\W}, r) - T_{(1-\bar{Y})^2} B^{lh}_{a,1}(\bar{\W}, r_\alpha) +i T_{(1-\bar{Y})^2} C^{lh}_{a,1}(\bar{\W}_{\alpha \alpha}, r),\\
&\partial_t r_1^{lh} + i T_{(1-Y)J^{-\Half}} \partial_\alpha^2 w_1^{lh}+ \text{cubic and higher terms} \\
=& -T_{1-\bar{Y}} A^{lh}_{h,1}(R_\alpha, r) + i\partial_\alpha^2 T_{1-\bar{Y}} C^{lh}_{h,1}(R, r) - T_{1-\bar{Y}} D^{lh}_{h,1}(R, r_\alpha)\\
& -iT_{J^{-\Half}(1-Y)^2} A^{lh}_{h,1}(\W, w_{\alpha \alpha}) + i T_{J^{-\Half}(1-Y)^2} \partial_\alpha^2 B^{lh}_{h,1}(\W, w) - iT_{J^{-\Half}(1-Y)^2} D^{lh}_{h,1}(\W_{\alpha \alpha}, w) \\
& -T_{1-Y} A^{lh}_{a,1}(\bar{R}_\alpha, r) + i\partial_\alpha^2 T_{1-Y} C^{lh}_{a,1}(\bar{R}, r) - T_{1-Y} D^{lh}_{a,1}(\bar{R}, r_\alpha)\\
& -iT_{J^{-\frac{3}{2}}} A^{lh}_{a,1}(\bar{\W}, w_{\alpha \alpha}) + i T_{J^{-\frac{3}{2}}} \partial_\alpha^2 B^{lh}_{a,1}(\bar{\W}, w) + iT_{J^{-\frac{3}{2}}} D^{lh}_{a,1}(\bar{\W}_{\alpha \alpha}, w).
\end{align*}
 The low-high holomorphic symbols $\mathfrak{a}^{lh}_{h,1}, \mathfrak{b}^{lh}_{h,1}, \mathfrak{c}^{lh}_{h,1}, \mathfrak{d}^{lh}_{h,1}$ for bilinear forms $A^{lh}_{h,1}, B^{lh}_{h,1}, C^{lh}_{h,1}, D^{lh}_{h,1}$ solve the linear system
\begin{equation*}
\left\{
    \begin{array}{lr}
    (\xi+\eta) \mathfrak{a}^{lh}_{h,1} - \eta \mathfrak{b}^{lh}_{h,1} + \xi^2 \mathfrak{c}^{lh}_{h,1} = 0 &\\
    \xi \mathfrak{b}^{lh}_{h,1} - \eta^2 \mathfrak{c}^{lh}_{h,1} - (\xi+\eta) \mathfrak{d}^{lh}_{h,1} = -\Half \xi \chi_1(\xi, \eta)  &\\
    \xi \mathfrak{a}^{lh}_{h,1} + (\xi + \eta)^2 \mathfrak{c}^{lh}_{h,1} + \eta \mathfrak{d}^{lh}_{h,1} = 0  &\\
     \eta^2 \mathfrak{a}^{lh}_{h,1} -(\xi+ \eta)^2 \mathfrak{b}^{lh}_{h,1} + \xi^2 \mathfrak{d}^{lh}_{h,1} = \xi(\xi+\eta) \chi_1(\xi, \eta),&  
    \end{array}
\right.
\end{equation*}
where  the symbol $\chi_{1}(\xi, \eta)$ is defined in \eqref{ChiOnelh} to select the low-high frequencies.
The expressions for the holomorphic symbols are given by
\begin{align*}
&\mathfrak{a}^{lh}_{h,1}(\xi, \eta) = -\dfrac{(\xi^3 +3\xi^2\eta+5\xi \eta^2 + 3 \eta^3) \chi_1(\xi, \eta)}{(9\xi^2 + 14\xi \eta+9\eta^2)\eta}, \\
&\mathfrak{b}^{lh}_{h,1}(\xi, \eta) = -\dfrac{(3\xi^3 +15\xi^2 \eta + 16\xi\eta^2 +6\eta^3)\chi_1(\xi, \eta)}{2(9\xi^2 + 14\xi \eta+9\eta^2)\eta}, \\
&\mathfrak{c}^{lh}_{h,1}(\eta, \zeta) = \dfrac{ (2\xi^2 +5 \xi \eta +\eta^2) \chi_1(\xi, \eta)}{2(9\xi^2 + 14\xi \eta+9\eta^2)\eta}, \\
&\mathfrak{d}^{lh}_{h,1}(\xi, \eta) = -\dfrac{(3\xi^3 +3\xi^2 \eta+\xi \eta^2 + \eta^3)\chi_1(\xi, \eta)}{2(9\xi^2 + 14\xi \eta+9\eta^2)\eta}.
\end{align*}
The leading terms of these symbols are
\begin{equation*}
\mathfrak{a}^{lh}_{h,1} \approx -\frac{1}{3}\chi_{1}(\xi, \eta), \quad \mathfrak{b}^{lh}_{h,1} \approx -\frac{1}{3}\chi_{1}(\xi, \eta),  \quad \mathfrak{c}^{lh}_{h,1} \approx \frac{1}{18\eta}\chi_{1}(\xi, \eta), \quad \mathfrak{d}^{lh}_{h,1} \approx -\frac{1}{18}\chi_{1}(\xi, \eta).   
\end{equation*}

Similarly,  the low-high mixed symbols $\mathfrak{a}^{lh}_{a,1}, \mathfrak{b}^{lh}_{a,1}, \mathfrak{c}^{lh}_{a,1}, \mathfrak{d}^{lh}_{a,1}$ for bilinear forms $A^{lh}_{a,1}, B^{lh}_{a,1}, C^{lh}_{a,1}, D^{lh}_{a,1}$ solve the linear system
\begin{equation*}
\left\{
    \begin{array}{lr}
    (\zeta -\eta) \mathfrak{a}^{lh}_{a,1} - \zeta \mathfrak{b}^{lh}_{a,1} - \eta^2 \mathfrak{c}^{lh}_{a,1} =0  &\\
    \eta \mathfrak{b}^{lh}_{a,1} + \zeta^2 \mathfrak{c}^{lh}_{a,1} + (\zeta-\eta) \mathfrak{d}^{lh}_{a,1} =  \Half\eta \chi_1(\eta, \zeta)   &\\
    \eta \mathfrak{a}^{lh}_{a,1} - ( \zeta -\eta)^2 \mathfrak{c}^{lh}_{a,1} - \zeta \mathfrak{d}^{lh}_{a,1} = 0  &\\
     \zeta^2 \mathfrak{a}^{lh}_{a,1} -( \zeta -\eta)^2 \mathfrak{b}^{lh}_{a,1} - \eta^2 \mathfrak{d}^{lh}_{a,1} =  \eta \zeta \chi_1(\eta, \zeta) .&  
    \end{array}
\right.
\end{equation*}
The solutions of the system are
\begin{align*}
&\mathfrak{a}^{lh}_{a,1}(\eta, \zeta) = \dfrac{ \zeta (\eta^2 -\eta \zeta + 3\zeta^2) \chi_1(\eta, \zeta)}{(\eta - \zeta)(4\eta^2 - 4\eta\zeta + 9 \zeta^2)}, \quad \mathfrak{b}^{lh}_{a,1}(\eta, \zeta) = \dfrac{(2\eta^3 -5\eta^2 \zeta + 8\eta \zeta^2 -6 \zeta^3)\chi_1(\eta, \zeta)}{2(\eta -\zeta)(4\eta^2 - 4\eta\zeta + 9 \zeta^2)}, \\
&\mathfrak{c}^{lh}_{a,1}(\eta, \zeta) = \dfrac{ 3\zeta^2 \chi_1(\eta, \zeta)}{2(\eta -\zeta)(4\eta^2 - 4\eta\zeta + 9 \zeta^2)}, \quad \mathfrak{d}^{lh}_{a,1}(\eta, \zeta) = -\dfrac{(2\eta^3-\eta^2 \zeta+ 4\eta\zeta^2 +\zeta^3)\chi_1(\eta, \zeta)}{2(\eta - \zeta)(4\eta^2 - 4\eta\zeta + 9 \zeta^2)}.
\end{align*}
The leading terms of these symbols are
\begin{equation*}
\mathfrak{a}^{lh}_{a,1} \approx \frac{1}{3}\chi_{1}(\eta, \zeta), \quad \mathfrak{b}^{lh}_{a,1} \approx \frac{1}{3}\chi_{1}(\eta, \zeta),  \quad \mathfrak{c}^{lh}_{a,1} \approx -\frac{1}{18\zeta}\chi_{1}(\eta, \zeta), \quad \mathfrak{d}^{lh}_{a,1} \approx \frac{1}{18}\chi_{1}(\eta, \zeta).  
\end{equation*}
Hence, we have
\begin{align*}
&w^{lh}_1 = -\frac{1}{3}T_{T_{1-Y}\W}w + \frac{i}{18}T_{T_{J^{\frac{3}{2}}(1-\bar{Y})^2}R}\partial_\alpha^{-1}r+ \frac{1}{3}T_{T_{1-\bar{Y}}\bar{\W}}w -\frac{i}{18} T_{T_{J^{\frac{1}{2}}}\bar{R}} \partial_\alpha^{-1}r + \text{lower order terms}, \\
&r^{lh}_1 = -\frac{1}{3}T_{T_{1-Y}\W}r -\frac{1}{18}T_R w + \frac{1}{3}T_{T_{1-\bar{Y}}\bar{\W}}r + \frac{1}{18}T_{T_{(1+\bar{\W})(1-Y)}\bar{R}}w  + \text{lower order terms}.
\end{align*}
 Low-high quadratic normal form variables $(w_1^{lh}, r_1^{lh})$ are bounded in the sense that
\begin{equation*}
 \|(w_1^{lh}, r_1^{lh}) \|_{\mathcal{H}^\Half} \lesssim \CalAZ \|(w,r) \|_{\mathcal{H}^\Half}.
\end{equation*}

\subsubsection{Normal forms for low-high quadratic terms in $(\nP\mathcal{G}_1, \nP\mathcal{K}_1)$ part two}
Here we  show that there exist low-high quadratic normal
form variables $(w_{h,2}^{lh}, r_{h,2}^{lh})$ that can remove the part of low-high quadratic terms in $(\nP\mathcal{G}_1, \nP\mathcal{K}_1)$ where linearized variables are at low frequencies.
More precisely,
\begin{align*}
&\partial_t w_{h,2}^{lh}+T_{1-\bar{Y}} \partial_\alpha r_{h,2}^{lh} + \text{cubic and higher terms} \\
=& T_{w_\alpha}T_{1-\bar{Y}}R + T_{w}T_{1-\bar{Y}}R_\alpha   + G, \\
&\partial_t r_{h,2}^{lh} + i T_{(1-Y)J^{-\Half}} \partial_\alpha^2 w_{h,2}^{lh}+ \text{cubic and higher terms} \\
=& T_{r_\alpha}T_{1-\bar{Y}}R -\frac{3}{2}i T_{J^{-\Half}(1-Y)^2} \partial_\alpha T_{w_{\alpha}} \W -i T_{J^{-\Half}(1-Y)^2} T_w \W_{\alpha \alpha} +K,
\end{align*}
where $(G,K)$ are perturbative source terms that satisfy \eqref{HalfGK}.

We choose normal form transformations as the sum of bilinear forms of the following type:
\begin{equation*} 
\begin{aligned}
w^{lh}_{h,2} &= B^{lh}_{h,2}(w, T_{1-Y}\W) + C^{lh}_{h,2}(r, T_{ J^{\frac{3}{2}}(1-\bar{Y})^2}R),\\
r^{lh}_{h,2} &= A^{lh}_{h,2}(r,T_{ 1-Y}\W) + D^{lh}_{h,2}(w, R).
\end{aligned}
\end{equation*}

Plugging these bilinear forms into the system, we have
\begin{align*}
&\partial_t w_{h,2}^{lh}+T_{1-\bar{Y}} \partial_\alpha r_{h,2}^{lh} + \text{cubic and higher terms} \\
=& -T_{1-\bar{Y}}B^{lh}_{h,2}(w, R_\alpha) -i T_{1-\bar{Y}}C^{lh}_{h,2}(w_{\alpha \alpha}, R) + T_{1-\bar{Y}}\partial_\alpha D^{lh}_{h,2}(w, R)\\
&+ T_{J^{-1}}\partial_\alpha A^{lh}_{h,2}(r, \W) - T_{J^{-1}} B^{lh}_{h,2}( r_\alpha, \W) -i T_{J^{-1}} C^{lh}_{h,2}(r, \W_{\alpha \alpha}),\\
&\partial_t r_{h,2}^{lh} + i T_{(1-Y)J^{-\Half}} \partial_\alpha^2 w_{h,2}^{lh}+ \text{cubic and higher terms} \\
=& -T_{1-\bar{Y}} A^{lh}_{h,2}(r, R_\alpha) + i\partial_\alpha^2 T_{1-\bar{Y}} C^{lh}_{h,2}(r, R) - T_{1-\bar{Y}} D^{lh}_{h,2}( r_\alpha, R)\\
& -iT_{J^{-\Half}(1-Y)^2} A^{lh}_{h,2}( w_{\alpha \alpha},\W) + i T_{J^{-\Half}(1-Y)^2} \partial_\alpha^2 B^{lh}_{h,2}(w, \W) - iT_{J^{-\Half}(1-Y)^2} D^{lh}_{h,2}(w, \W_{ \alpha \alpha}).
\end{align*}
The low-high holomorphic symbols $\mathfrak{a}^{lh}_{h,2}(\xi, \eta), \mathfrak{b}^{lh}_{h,2}(\xi, \eta), \mathfrak{c}^{lh}_{h,2}(\xi, \eta), \mathfrak{d}^{lh}_{h,2}(\xi, \eta)$ for bilinear forms $A^{lh}_{h,2}, B^{lh}_{h,2}, C^{lh}_{h,2}, D^{lh}_{h,2}$ solve the linear system
\begin{equation*}
\left\{
    \begin{array}{lr}
    (\xi+\eta) \mathfrak{a}^{lh}_{h,2} - \xi \mathfrak{b}^{lh}_{h,2} + \eta^2 \mathfrak{c}^{lh}_{h,2} = 0 &\\
    \eta \mathfrak{b}^{lh}_{h,2} - \xi^2 \mathfrak{c}^{lh}_{h,2} - (\xi+\eta) \mathfrak{d}^{lh}_{h,2} = - (\xi + \eta) \chi_1(\xi, \eta)  &\\
    \eta\mathfrak{a}^{lh}_{h,2} + (\xi + \eta)^2 \mathfrak{c}^{lh}_{h,2} + \xi \mathfrak{d}^{lh}_{h,2} = -\xi \chi_1(\xi, \eta)   &\\
     \xi^2 \mathfrak{a}^{lh}_{h,2} -(\xi+ \eta)^2 \mathfrak{b}^{lh}_{h,2} + \eta^2 \mathfrak{d}^{lh}_{h,2} = (\frac{3}{2}\xi(\xi+\eta) + \eta^2) \chi_1(\xi, \eta),&  
    \end{array}
\right.
\end{equation*}
where  the symbol $\chi_{1}(\xi, \eta)$ is defined in \eqref{ChiOnelh} to select the low-high frequencies.
The expressions for the holomorphic symbols are given by
\begin{align*}
&\mathfrak{a}^{lh}_{h,2}(\xi, \eta) = -\dfrac{(9\xi^3 +10\xi^2\eta+3\xi \eta^2 -6 \eta^3) \chi_1(\xi, \eta)}{2\eta(9\xi^2 + 14\xi \eta+9\eta^2)}, \\
&\mathfrak{b}^{lh}_{h,2}(\xi, \eta) = -\dfrac{(9\xi^3 +19\xi^2 \eta + 19\xi\eta^2 +9\eta^3)\chi_1(\xi, \eta)}{2\eta(9\xi^2 + 14\xi \eta+9\eta^2)}, \\
&\mathfrak{c}^{lh}_{h,2}(\eta, \zeta) = -\dfrac{ 3(\xi+\eta)^2 \chi_1(\xi, \eta)}{\eta(9\xi^2 + 14\xi \eta+9\eta^2)}, \\
&\mathfrak{d}^{lh}_{h,2}(\xi, \eta) = \dfrac{3(2\xi^3 +5\xi^2 \eta+6\xi \eta^2 + 3\eta^3)\chi_1(\xi, \eta)}{2\eta(9\xi^2 + 14\xi \eta+9\eta^2)}.
\end{align*}
Since $|\xi| \ll |\eta |$,  the leading terms of these symbols are
\begin{equation*}
\mathfrak{a}^{lh}_{h,2} \approx \frac{1}{3}\chi_{1}(\xi, \eta), \quad \mathfrak{b}^{lh}_{h,2} \approx -\frac{1}{2}\chi_{1}(\xi, \eta),  \quad \mathfrak{c}^{lh}_{h,2} \approx -\frac{1}{3\eta}\chi_{1}(\xi, \eta), \quad \mathfrak{d}^{lh}_{h,2} \approx \frac{1}{2}\chi_{1}(\xi, \eta).   
\end{equation*}
Hence, we have
\begin{align*}
&w_{h,2}^{lh} = -\frac{1}{2}T_{w}T_{1-Y}\W - \frac{i}{3} T_rT_{J^{\frac{3}{2}}(1-\bar{Y})^2}\partial_\alpha^{-1}R + \text{lower order terms}, \\
&r_{h,2}^{lh} = \frac{1}{3}T_{r}T_{1-Y}\W +\Half T_w R  + \text{lower order terms}.
\end{align*}
Low-high quadratic normal form variables $(w_{h,2}^{lh}, r_{h,2}^{lh})$ satisfy the estimates 
\begin{equation*}
 \|(w_{h,2}^{lh}, r_{h,2}^{lh}) \|_{\mathcal{H}^\Half} \lesssim \CalAO \|(w,r) \|_{\mathcal{H}^\Half}.
\end{equation*}

\subsubsection{Normal forms for low-high quadratic terms in $(\nP\mathcal{G}_0, \nP\mathcal{K}_0)$}
Here we  show that there exist low-high quadratic normal
form variables $(w_{a,2}^{lh}, r_{a,2}^{lh})$ that can remove the part of low-high quadratic terms in $(\nP\mathcal{G}_0, \nP\mathcal{K}_0)$.
Precisely, we show that
\begin{align*}
&\partial_t w_{a,2}^{lh}+T_{1-\bar{Y}} \partial_\alpha r_{a,2}^{lh} + \text{cubic and higher terms} \\
=& T_{J^{-1}}T_{\bar{r}_\alpha}\W -  T_{(1-\bar{Y})^2(1+\W)}T_{\bar{w}_\alpha} R  + G, \\
&\partial_t r_{a,2}^{lh} + i T_{(1-Y)J^{-\Half}} \partial_\alpha^2 w_{a,2}^{lh}+ \text{cubic and higher terms} \\
=& T_{1-\bar{Y}}T_{\bar{r}_\alpha} R +\frac{i}{2} T_{J^{-\frac{3}{2}}} T_{\bar{w}_{\alpha \alpha}}\W - \frac{i}{2} T_{J^{-\frac{3}{2}}} T_{\bar{w}_{\alpha}}\W_\alpha+K,
\end{align*}
where $(G,K)$ are perturbative source terms that satisfy \eqref{HalfGK}.

We choose normal form transformations as the sum of bilinear forms of the following type:
\begin{equation*} 
\begin{aligned}
w^{lh}_{a,2} &=  B^{lh}_{a,2}\left(\bar{w}, T_{1-Y}\W \right) + C^{lh}_{a,2}(\bar{r}, T_{J^{-\Half}(1+\W)^2}R ),\\
r^{lh}_{a,2} &=  A^{lh}_{a,2}\left(\bar{r}, T_{1-Y}\W \right) + D^{lh}_{a,2}\left(\bar{w}, T_{(1+\W)(1- \bar{Y})}R \right).
\end{aligned}
\end{equation*}

Plugging these bilinear forms into the system, we get
\begin{align*}
&\partial_t w_{a,2}^{lh}+T_{1-\bar{Y}} \partial_\alpha r_{a,2}^{lh} + \text{cubic and higher terms} \\
=& -T_{(1-\bar{Y})^2(1+\W)}B^{lh}_{a,2}(\bar{w}, R_\alpha) +i T_{(1-\bar{Y})^2(1+\W)}C^{lh}_{a,2}(\bar{w}_{\alpha \alpha}, R) + T_{(1-\bar{Y})^2(1+\W)}\partial_\alpha D^{lh}_{a,2}( \bar{w}, R)\\
&+ T_{J^{-1}}\partial_\alpha A^{lh}_{a,2}(\bar{r}, \W) - T_{J^{-1}} B^{lh}_{a,2}( \bar{r}_\alpha, \W) -i T_{J^{-1}} C^{lh}_{a,2}(\bar{r}, \W_{\alpha \alpha}),\\
&\partial_t r_{a,2}^{lh} + i T_{(1-Y)J^{-\Half}} \partial_\alpha^2 w_{a,2}^{lh}+ \text{cubic and higher terms} \\
=& -T_{1-\bar{Y}} A^{lh}_{a,2}( \bar{r}, R_\alpha) + iT_{1-\bar{Y}}\partial_\alpha^2  C^{lh}_{a,2}(\bar{r}, R) - T_{1-\bar{Y}} D^{lh}_{a,2}(\bar{r}_\alpha, R)\\
& +iT_{J^{-\frac{3}{2}}} A^{lh}_{a,2}( \bar{w}_{\alpha \alpha}, \W) + i T_{J^{-\frac{3}{2}}} \partial_\alpha^2 B^{lh}_{a,2}(\bar{w}, \W) - iT_{J^{-\frac{3}{2}}} D^{lh}_{a,2}(\bar{w}, \W_{\alpha \alpha}).
\end{align*}

The low-high mixed symbols $\mathfrak{a}^{lh}_{a,2}(\eta, \zeta),\mathfrak{b}^{lh}_{a,2}(\eta, \zeta), \mathfrak{c}^{lh}_{a,2}(\eta, \zeta), \mathfrak{d}^{lh}_{a,2}(\eta, \zeta)$ for bilinear forms $A^{lh}_{a,2},$ $B^{lh}_{a,2}, C^{lh}_{a,2}, D^{lh}_{a,2}$ solve the linear system
\begin{equation*}
\left\{
    \begin{array}{lr}
    (\zeta -\eta) \mathfrak{a}^{lh}_{a,2} + \eta \mathfrak{b}^{lh}_{a,2} + \zeta^2 \mathfrak{c}^{lh}_{a,2} = -\eta \chi_1(\eta, \zeta)  &\\
    \zeta \mathfrak{b}^{lh}_{a,2} + \eta^2 \mathfrak{c}^{lh}_{a,2} - (\zeta-\eta) \mathfrak{d}^{lh}_{a,2} =  -\eta \chi_1(\eta, \zeta)   &\\
    \zeta \mathfrak{a}^{lh}_{a,2} + ( \zeta -\eta)^2 \mathfrak{c}^{lh}_{a,2} - \eta \mathfrak{d}^{lh}_{a,2} = \eta \chi_1(\eta, \zeta)  &\\
     \eta^2 \mathfrak{a}^{lh}_{a,2} +( \zeta -\eta)^2 \mathfrak{b}^{lh}_{a,2} - \zeta^2 \mathfrak{d}^{lh}_{a,2} =  \Half\eta(\eta + \zeta) \chi_1(\eta, \zeta) .&  
    \end{array}
\right.
\end{equation*}
The solutions of this system are
\begin{align*}
&\mathfrak{a}^{lh}_{a,2}(\eta, \zeta) = \dfrac{ 3 (2\eta^2 -\eta \zeta + 2\zeta^2) \chi_1(\eta, \zeta)}{2(4\eta^2 - 4\eta\zeta + 9 \zeta^2)}, \quad \mathfrak{b}^{lh}_{a,2}(\eta, \zeta) = -\dfrac{(2\eta^2 +\eta \zeta +9\zeta^2)\chi_1(\eta, \zeta)}{2(4\eta^2 - 4\eta\zeta + 9 \zeta^2)}, \\
&\mathfrak{c}^{lh}_{a,2}(\eta, \zeta) = -\dfrac{ 3\zeta \chi_1(\eta, \zeta)}{4\eta^2 - 4\eta\zeta + 9 \zeta^2}, \quad \mathfrak{d}^{lh}_{a,2}(\eta, \zeta) = -\dfrac{(8\eta^2-8\eta \zeta+ 9\zeta^2 )\chi_1(\eta, \zeta)}{2(4\eta^2 - 4\eta\zeta + 9 \zeta^2)}.
\end{align*}
Since $|\xi| \ll |\eta |$, the leading terms of these symbols are
\begin{equation*}
\mathfrak{a}^{lh}_{a,2} \approx \frac{1}{3}\chi_{1}(\eta, \zeta), \quad \mathfrak{b}^{lh}_{a,2} \approx -\frac{1}{2}\chi_{1}(\eta, \zeta),  \quad \mathfrak{c}^{lh}_{a,2} \approx -\frac{1}{3\zeta}\chi_{1}(\eta, \zeta), \quad \mathfrak{d}^{lh}_{a,2} \approx -\frac{1}{2}\chi_{1}(\eta, \zeta).   
\end{equation*}
Consequently,
\begin{align*}
&w_{a,2}^{lh} = -\frac{1}{2}T_{\bar{w}}T_{1-Y}\W - \frac{i}{3} T_{\bar{r}}T_{J^{-\frac{1}{2}}(1+\W)^2}\partial_\alpha^{-1}R + \text{lower order terms}, \\
&r_{a,2}^{lh} = \frac{1}{3}T_{\bar{r}}T_{1-Y}\W -\Half T_{\bar{w}} T_{(1+\W)(1-\bar{Y})}R  + \text{lower order terms}.
\end{align*}
Low-high quadratic normal form variables $(w_{a,2}^{lh}, r_{a,2}^{lh})$ satisfy the estimates 
\begin{equation*}
 \|(w_{a,2}^{lh}, r_{a,2}^{lh}) \|_{\mathcal{H}^\Half} \lesssim \CalAO \|(w,r) \|_{\mathcal{H}^\Half}.
\end{equation*}

\subsection{Cubic normal form analysis} \label{s:CubicNormal}
In this subsection, we compute the cubic normal form transformations to eliminate the non-perturbative cubic source terms of $(\mathcal{G}^\sharp, \mathcal{K}^\sharp)$ in \eqref{ParadifferentialLinearEqn} as well as those produced by quadratic normal form transformations.

We consider each part of the cubic source terms:
\begin{itemize}
\item The time derivative acts on the para-coefficients of the quadratic source terms such as $T_{1-Y}$, $T_{(1-Y)(1+\bar{\W})}$, and so on.
This results in an additional $T_{R_\alpha}$ or $T_{\bar{R}_\alpha}$.
In this scenario, we estimate $R_\alpha$ or $\bar{R}_\alpha$ in $L^2$, which introduces an $N_s$ factor.
 Furthermore, we estimate $w$ or $r$ in $L^\infty$ or Zygmund spaces, and use Sobolev embedding to further bound them by $\|(w,r) \|_{\mathcal{H}^\Half}$.
The cubic source terms of this type satisfy \eqref{GoodSourceTermBound} as desired.
\item Cubic and quartic terms in $(\nP\mathcal{G}_0, \nP\mathcal{K}_0)$ and $(\nP\mathcal{G}_1, \nP\mathcal{K}_1)$.
These terms include the cubic term $-\nP[T_{1-\bar{Y}}T_{\bar{w}}\bar{R}_\alpha T_{1-Y}\W]$ in $\mathcal{G}_{0,0}$, cubic terms $-T_w \nP[R\bar{Y}_\alpha]-\Pi(w, \nP[R\bar{Y}_\alpha])$ in $\nP\mathcal{G}_1$,  $iT_{1-Y}T_w a+ iT_{1-Y}\Pi(w,a)-iT_{aw}Y-i\Pi(aw,Y)$, and the cubic part of $i(T_{Lw}Y +\Pi(Lw, Y))$ in $\nP\mathcal{K}_1$.
These terms also satisfy  \eqref{GoodSourceTermBound} by estimating $w$ in $L^\infty$ and using Sobolev embeddings.
\item The sub-leading term of the time derivative acts on the main factors of balanced quadratic normal form variables, leading to additional cubic terms. 
Cubic terms of this type always satisfy \eqref{GoodSourceTermBound} as one can freely distribute derivatives for balanced paraproducts.
\item The sub-leading term of the time derivative acts on the main factors of low-high quadratic normal form variables, introducing non-perturbative cubic terms.
When this occurs at variables at low frequencies, the cubic terms satisfy  \eqref{GoodSourceTermBound}.
Hence, we only need to consider the case when the sub-leading term of the time derivative acts on  high-frequency variables of low-high quadratic normal forms.
\end{itemize}

Consequently, the non-perturbative cubic source terms are of the following type:
\begin{align*}
&\mathcal{G}_{[3]} = L^{llh}(T_{1-Y}\W, b, w_\alpha) +L^{llh}(T_{J^{\frac{3}{2}}(1-\bar{Y})^2}R, b, r) + L^{llh}(T_{1-\bar{Y}}\bar{\W}, b, w_\alpha) +L^{llh}(T_{J^\Half}\bar{R}, b, r) , \\
+& L^{llh}(T_{1-Y}\W, w, R_\alpha) +L^{llh}(T_{J^{\frac{3}{2}}(1-\bar{Y})^2}R, w, \W_\alpha) + L^{llh}(T_{1-\bar{Y}}\bar{\W}, w, R_\alpha) +L^{llh}(T_{J^\Half}\bar{R}, w, \W_\alpha)\\
+&  L^{llh}(w, b, T_{1-Y}\W_\alpha)+ L^{llh}(r, b, T_{J^{\frac{3}{2}}(1-\bar{Y})^2} R) +  L^{llh}(\bar{w}, b, T_{1-Y}\W_\alpha)+ L^{llh}(\bar{r}, b, T_{J^{-\frac{1}{2}}(1+\W)^2} R), \\
&\mathcal{K}_{[3]} = L^{llh}(T_{1-Y}\W, b, r_\alpha) +L^{llh}(R, b, w_\alpha) + L^{llh}(T_{1-\bar{Y}}\bar{\W}, b, r_\alpha) + L^{llh}(T_{(1+\bar{\W})(1-Y)}\bar{R}, b, w_\alpha)\\
+& L^{llh}(T_{1-Y}\W, w, \W_{\alpha \alpha}) +L^{llh}(R, w, R_\alpha) + L^{llh}(T_{1-\bar{Y}}\bar{\W}, w, \W_{\alpha \alpha}) + L^{llh}(T_{(1+\bar{\W})(1-Y)}\bar{R}, w, R_\alpha)\\
+&  L^{llh}(r, b, T_{1-Y}\W_\alpha)+ L^{llh}(w, b, R_\alpha) +  L^{llh}(\bar{r}, b, T_{1-Y}\W_\alpha)+ L^{llh}(\bar{w}, b, T_{(1+\W)(1-\bar{Y})}R_\alpha),
\end{align*}
where $L^{llh}$ are trilinear forms where each of its arguments has relatively low, low and high frequencies respectively.
These non-perturbative cubic source terms need $\Half$ less derivative to satisfy the estimate \eqref{GoodSourceTermBound}.

To give an outline of how to compute the symbols for these trilinear forms, we follow the idea of Section 3.2 of \cite{MR3667289} by using auxiliary variables.
Consider auxiliary variables $Z_{\pm}$ and $z_{\pm}$ defined by
\begin{equation*}
    Z_{\pm} := R \pm i |D|^\Half \W, \quad z_{\pm} := r \pm i |D|^\Half w.
\end{equation*}
 $Z_{\pm}$ and $z_{\pm}$ solve the scalar equations
\begin{equation*}
\partial_t Z_\pm = \pm i |D|^{\frac{3}{2}}Z_\pm + \text{lower order terms}, \quad \partial_t z_\pm = \pm i |D|^{\frac{3}{2}}z_\pm + \text{lower order terms}.
\end{equation*}
Setting the non-perturbative cubic source term
\begin{equation*}
    \mathcal{Z}_{[3]}^{\pm} = \mathcal{K}_{[3]} \pm i|D|^\Half \mathcal{G}_{[3]} = \sum_{\pm} L^1_{\pm \pm \pm}(Z_\pm, Z_\pm, z_\pm),
\end{equation*}
we will need to find the cubic normal form transformation 
\begin{align*}
 Z_{[3]}^\pm = \sum_{\pm} L^2_{\pm \pm \pm}(Z_\pm, Z_\pm, z_\pm)   
\end{align*}
such that $Z_{[3]}^\pm$ solve the equation
\begin{equation*}
    (\partial_t \mp i|D|^{\frac{3}{2}}) Z^\pm_{[3]} = \mathcal{Z}_{[3]}^\pm.
\end{equation*}
By arranging the frequencies of each variable and the outcome frequency, from small to large in absolute value, the frequencies $\xi_1, \xi_2, \xi_3, \xi_4$ satisfy \eqref{XiOnetoFour}, and the denominators of the symbols of trilinear forms $L^2_{\pm \pm \pm}$ satisfy the non-resonance condition
\begin{equation*}
  |\xi_1|^{\frac{3}{2}} \pm |\xi_2|^{\frac{3}{2}} \pm |\xi_3|^{\frac{3}{2}} \pm |\xi_4|^{\frac{3}{2}} \neq 0.
\end{equation*}
Similar to the discussion in Section \ref{s:QuarticEnergy}, the symbols of trilinear forms $L^2_{\pm \pm \pm}$ are bounded and the normal form transformation $Z^{[3]}_\pm$ satisfy
\begin{equation*}
    \|Z_{[3]}^\pm \|_{H^\Half}\lesssim_{\| Z\|_{H^s}} (1+\| Z\|_{C^{1+\epsilon}_{*}}^2) \|z\|_{H^\Half}.
\end{equation*}
Switching back to $(\W, R)$ and $(w,r)$ as in Section 3.2 of \cite{MR3667289} and adding appropriate para-coefficients as quasilinear corrections, we obtain the cubic normal form variables $(w_{[3]}, r_{[3]})$, which satisfy the estimate
\begin{equation*}
 \|(w_{[3]}, r_{[3]}) \|_{\mathcal{H}^\Half} \lesssim \CalAO \|(w,r) \|_{\mathcal{H}^\Half}.
\end{equation*}
The extra quartic terms produced by $(w_{[3]}, r_{[3]})$ satisfy \eqref{GoodSourceTermBound} because normal form transformations replace the leading term of time derivatives by sub-leading terms, and thus reduce $\Half$ derivative so that quartic terms become perturbative.

At the end of this section, we use  the paradifferential energy $E^{\Half, para}_{lin}$ for the linear paradifferential system constructed in Section \ref{s:HomoFlow} as well as all the paradifferential normal form transformations to finish the proof of Theorem \ref{t:LinearizedWellposed}.
Define normal form variables
\begin{equation*}
    (w_{NF}, r_{NF}): = (w,r)+ (w^{bal}_1, r^{bal}_1)+ (w^{bal}_2, r^{bal}_2) + (w^{lh}_1, r^{lh}_1)+ (w^{lh}_{h,2}, r^{lh}_{h,2}) + (w^{lh}_{a,2}, r^{lh}_{a,2}) + (w_{[3]}, r_{[3]}),
\end{equation*}
where each part of the normal form variables are constructed in the above subsections.
Setting
\begin{equation*}
   E^{\Half}_{lin}(w,r): = E^{\Half, para}_{lin}(w_{NF}, r_{NF}),
\end{equation*}
then the energy $E^{\Half}_{lin}$ satisfies the norm equivalence and the energy estimate in Theorem \ref{t:LinearizedWellposed}.

\section{The proof of local well-posedness} \label{s:Wellposed}
In this section, we finish the proof of the the main result of this article, namely, the low regularity well-posedness Theorem \ref{t:MainWellPosed}.
Recall that it was already shown by Alazard-Burq-Zuily \cite{MR2805065} that the system \eqref{e:WR} is locally well-posed in $\H^s$ for where $s> \frac{3}{2}$ (Note that the proof of Theorem \ref{t:Previous} uses Strichartz estimate, and does not hold in the periodic case).
In the following, we will just consider $s$ as a fixed Sobolev index $1<s\leq \frac{3}{2}$.

Our proof below follows the sequence of steps laid out  in the expository article of Ifrim-Tataru \cite{MR4557379}.
We will first establish the $\mathcal{H}^{\tau}$ bounds for  regular solutions, with $\tau >\frac{3}{2}$.
Then, we use these  regular solutions  to construct rough solutions in $\mathcal{H}^{s}$.
More precisely, regular solutions are obtained by truncating the rough initial data in frequency,  so that we get a continuous family of solutions, thereby estimating only a solution for  linearized equations at each step.
Finally, we prove the continuous dependence on the initial data in $\mathcal{H}^{s}$.

Note that we can make  use of the  scaling \eqref{Scaling}.
By choosing the scaling parameter $\lambda$ small enough, we can make the $\mathcal{H}^s$ norm of initial data and the control norm $\CalAO \ll 1$, at the price to turn the gravity $g$ into $\frac{g}{\lambda^2}$.
For simplicity, we only give an outline of the proof below.
We refer the interested readers to Section 7 of \cite{ai2023dimensional} for the details of the proof, especially the definition and the use of \textit{frequency envelopes}.

\begin{proof}[Outline of proof of Theorem \ref{t:MainWellPosed}]
 The  proof of the theorem is divided into the following three steps:
 
 \textbf{$(1)$ $\mathcal{H}^\tau$ bounds for regular solutions.}
 Suppose we have an $\mathcal{H}^\tau$ solution $(\W, R)$ that satisfies  the initial condition
\begin{equation*}
\|(\W_0, R_0)\|_{\mathcal{H}^s}\leq \mathcal{M}_0 \ll 1,
\end{equation*}
for some small constant $\mathcal{M}_0$.
In order to show that there exists a time $T = T(\mathcal{M}_0)$ such that the solution exists in $C([0,T]; \mathcal{H}^\tau)$ and satisfies the bounds
\begin{align}
&\|(\W, R)\|_{L^\infty(0,T;\mathcal{H}^s)}<  \mathcal{M}(\mathcal{M}_0), \label{SEstimate} \\
&\|(\W, R)\|_{L^\infty(0,T;\mathcal{H}^\tau)}\leq  \mathcal{C}(\mathcal{M}_0) \|(\W_0,R_0
)\|_{\mathcal{H}^\tau}, \label{TauEstimate}
\end{align}
for constants $\mathcal{C}$ and $\mathcal{M}$ that depend on $\mathcal{M}_0$,
we make the bootstrap assumption
\begin{equation*}
    \|(\W, R)\|_{L^\infty(0,T; \mathcal{H}^s)} <  2\mathcal{M}.
\end{equation*}
By the bootstrap assumption and Sobolev embedding, the control norms satisfy 
\begin{equation*}
    \CalAZ, \CalAO \lesssim \mathcal{M}.
\end{equation*}
 Using the paradifferential modified energy estimates Theorem \ref{t:MainEnergyEstimate} and Gronwall’s inequality, we get the energy estimate
\begin{equation*}
    \|(\W, R)(t)\|_{\mathcal{H}^s} <  e^{Ct}  \|(\W_0, R_0)\|_{\mathcal{H}^s} \lesssim e^{Ct}\mathcal{M}_0.
\end{equation*}
By choosing constant $\mathcal{M}$ large and time $T$ small enough, we get the bound \eqref{SEstimate}.
Similarly, applying Theorem \ref{t:MainEnergyEstimate} for $\tau$ and Gronwall’s inequality,
\begin{equation*}
\| (\W ,R)(t)\|_{\mathcal{H}^\tau} \lesssim e^{Ct} \| (\W_0 ,R_0)\|_{\mathcal{H}^\tau},
\end{equation*}
which gives the bound \eqref{TauEstimate}.
Therefore we have the $\mathcal{H}^\tau$ bounds for regular solutions up to time $T$.

\textbf{$(2)$  Construction of rough solutions, $(\W, R)\in \mathcal{H}^{s}$.}
We truncate the frequency the initial data $(W_{<k}(0), Q_{<k}(0))$ and $(\W_{<k}(0), R_{<k}(0))$ at $2^{k}$.
We can then establish the bound of regularized initial data using frequency envelopes.
The corresponding solutions will be regular, with a uniform lifespan bound as shown in Step (1).
Here $k$ can be viewed as a continuous parameter rather than a discrete parameter.
Then
\begin{equation*}
    (w^k, r^k) = (\partial_k W_{<k}, \partial_k Q_{<k} - R_{<k} \partial_k W_{<k} )
\end{equation*}
solve the corresponding linearized equations around $(\W_{<k}, R_{<k})$.
For the high-frequency part of the regularized solutions, we apply the energy estimates for the full equations Theorem \ref{t:MainEnergyEstimate}.
Next, using Theorem \ref{t:LinearizedWellposed} for the linearized variables $(w^k, r^k)$, one can establish the difference bound $(\W_{<k+1}-\W_{<k}, R_{<k+1}-R_{<k})$ in $\mathcal{H}^{s}$.
Summing up with respect to $k$, it follows that the sequence $(\W_{<k}, R_{<k})$ converges to a solution $(\W, R)$ with uniform $\mathcal{H}^{s}$ bound in time interval $[0,T]$.
This further shows that the solution is unique in the sense that it is the unique limit of regular solutions.

\textbf{$(3)$  Continuous dependence on the data for rough solutions.}
We consider an arbitrary sequence $(\W_j, R_j)(0)$ that converges to the initial data $(\W_0, R_0)$ in $\mathcal{H}^{s}$ topology.
Again by truncating the initial data at frequency $2^k$ and obtaining the corresponding regular solutions, we get the sequence of solutions $(\W_j^k, R_j^k)$, respectively $(\W^k, R^k)$.
Due to the continuous dependence for the regular solutions which is proved in  \cite{MR2805065}, we have for each $k$,
\begin{equation*}
(\W^k_j, R^k_j)- (\W^k, R^k) \rightarrow 0 \quad \text{in } \mathcal{H}^\tau, \quad \tau> \frac{3}{2}.
\end{equation*}
On the other hand, for the initial data, we let $k$ go to infinity,
\begin{equation*}
(\W^k_j, R^k_j)(0)- (\W_j, R_j)(0) \rightarrow 0 \quad \text{in } \mathcal{H}^s, \quad \text{uniformly in } j.
\end{equation*}
With the help of frequency envelopes, we get the uniform convergence for the solution:
\begin{equation*}
(\W^k_j, R^k_j)- (\W_j, R_j) \rightarrow 0 \quad \text{in } \mathcal{H}^s, \quad \text{uniformly in } j.
\end{equation*}
 We can again let $k$ go to infinity, and conclude that
 \begin{equation*}
  (\W_j, R_j)-(\W, R)\rightarrow 0 \quad \text{in } \mathcal{H}^{s}.
 \end{equation*}
 This shows the continuous dependence on the data for solutions in $\mathcal{H}^{s}$.

 \end{proof}

\appendix

\section{Paradifferential calculus and related estimates} \label{s:Norms}
In this section, we list the definition of norms and recall paraproducts and paradifferential estimates we have used in previous sections.
On one hand, the normal forms are not explicit bilinear forms; they are expressed using multi-linear Fourier multipliers.
On the other hand, in order to get the paradifferential modified energy estimate, we need to work at the paradifferential level to rebalance the fractional derivatives.
Many of these definitions and estimates are relatively standard.
They can be found in for instance \cite{ai2023dimensional,MR2931520, MR3260858,MR3585049} or the textbooks \cite{MR2768550, MR2418072}.

\subsection{Norms and function spaces}
We recall the  Littlewood-Paley frequency decomposition,
\begin{equation*}
    I = \sum_{k\in \mathbb{N}} P_k, 
\end{equation*}
where for each $k\geq 1$, the symbols of $P_k$ are smooth and  localized at $2^k$,  and $P_0$ selects the low frequency components $|\xi|\leq 1$.
\begin{enumerate}
\item Let $s\in \mathbb{R}$, and $p,q \in [1, \infty]$, the non-homogeneous Besov space $B^s_{p,q}(\mathbb{R})$ is defined as the space of all  tempered distributions $u$ such that
\begin{equation*}
\| u\|_{B^s_{p,q}} : = \left\|(2^{ks}\|P_k u \|_{L^p})_{k=0}^\infty \right\|_{l^q} < +\infty
\end{equation*}
\item When $p = q = \infty$, the Besov space $B^s_{\infty, \infty}$ becomes the \textit{Zygmund space} $C^s_{*}$.
When $p = q =2$, the Besov space $B^s_{2,2}$ becomes the \textit{Sobolev space} $H^s$.
\item Let $1\leq p_1 \leq p_2 \leq \infty$, $1\leq r_1 \leq r_2 \leq \infty$, then for any real number $s$,
\begin{equation*}
    B^s_{p_1, p_2}(\mathbb{R}) \hookrightarrow B^{s-(\frac{1}{p_1} - \frac{1}{p_2})}_{p_2, r_2}(\mathbb{R}).
\end{equation*}
As a special case when $p_1 = r_1 =2$ and $p_2 = r_2 = \infty$, 
\begin{equation*}
 H^{s+\frac{1}{2}}(\mathbb{R}) \hookrightarrow C^s_{*}(\mathbb{R}) \quad \forall s. 
\end{equation*}
The Sobolev space $H^{s+\frac{1}{2}}(\mathbb{R})$ can be embedded into the Zygmund space $C^s_{*}(\mathbb{R})$.
\item Let $k\in \mathbb{N}$, we let $W^{k,\infty}(\mathbb{R})$ be the space of all functions such that $\partial_x^j u \in L^\infty(\mathbb{R})$, $0\leq j \leq k$. 
For $\rho = k+ \sigma$ with $k\in \mathbb{N}$ and $\sigma \in (0,1)$, we denote $W^{\rho, \infty}(\mathbb{R})$  the
space of all function $u\in W^{k,\infty}(\mathbb{R})$ such that $\partial_x^k u$ is $\sigma$- H\"{o}lder continuous on $\mathbb{R}$. 
\item The Zygmund space $C^s_{*}(\mathbb{R})$ is just the the H\"{o}lder space $W^{s, \infty}(\mathbb{R})$ when $s\in (0,\infty)\backslash \mathbb{N}$.
One has the embedding properties
\begin{align*}
  &C_{*}^s(\mathbb{R}) \hookrightarrow L^\infty(\mathbb{R}), \quad s>0; \qquad L^\infty(\mathbb{R}) \hookrightarrow C_{*}^s, \quad s<0;\\
  &C_{*}^{s_1}(\mathbb{R})\hookrightarrow C_{*}^{s_2}(\mathbb{R}), \quad H^{s_1}(\mathbb{R})\hookrightarrow H^{s_2}(\mathbb{R}), \qquad s_1>s_2.
\end{align*}
\end{enumerate}

\subsection{ Paradifferential and  Moser type estimates}
\begin{definition}
\begin{enumerate}
\item Let $\rho\in [0,\infty)$, $m\in \mathbb{R}$. 
$\Gamma^m_\rho(\mathbb{R})$ denotes the space of locally bounded functions $a(x, \xi)$ on $\mathbb{R}\times (\mathbb{R}\backslash \{0\})$, which are $C^\infty$ with respect to $\xi$ for $\xi \neq 0$ and such that for all $k \in \mathbb{N}$ and $\xi \neq 0$, the function $x\mapsto \partial_\xi^k a(x,\xi)$ belongs to $W^{\rho,\infty}(\mathbb{R})$ and there exists a constant $C_k$ with
\begin{equation*}
\forall |\xi|\geq \frac{1}{2}, \quad \|\partial_\xi^k a(\cdot,\xi) \|_{W^{\rho,\infty}} \leq C_k (1+ |\xi|)^{m-k}.
\end{equation*}
Let $a\in \Gamma^m_\rho$,  we define the semi-norm
\begin{equation*}
M^m_{\rho}(a) = \sup_{k \leq \frac{3}{2}+\rho} \sup_{|\xi|\geq \frac{1}{2}} \|(1+ |\xi|)^{k-m}\partial_\xi^k a(\cdot,\xi)  \|_{W^{\rho,\infty}}.
\end{equation*}
\item Given $a\in \Gamma^m_\rho(\mathbb{R})$, let $C^\infty$ functions $\chi(\theta, \eta)$ and $\psi(\eta)$ be such that for some $0<\epsilon_1 < \epsilon_2<1$,
\begin{align*}
    &\chi(\theta, \eta) = 1,  \text{ if } |\theta| \leq \epsilon_1(1+ |\eta|), \qquad \chi(\theta, \eta) = 0,  \text{ if } |\theta| \geq \epsilon_2(1+ |\eta|),\\
    &\psi(\eta) = 0, \text{ if } |\eta|\leq \frac{1}{5}, \qquad \psi(\eta) =1, \text{ if } |\eta|\geq \frac{1}{4}.
\end{align*}
We define the paradifferential operator $T_a$ by
\begin{align*}
    \widehat{T_a u}(\xi) = \frac{1}{2\pi}\int \chi(\xi -\eta, \eta) \hat{a}(\xi-\eta, \eta)\psi(\eta)\hat{u}(\eta) d\eta,
\end{align*}
where $\hat{a}(\theta, \xi)$ is the Fourier transform of a with respect to the  variable x.
\item Let $m\in \mathbb{R}$, an operator  is said to be of order $m$ if, for all $s\in \mathbb{R}$, it is bounded from $H^s$ to $H^{s-m}$.
\end{enumerate}
\end{definition}

We recall the basic symbolic calculus for paradifferential operators in the following result.
\begin{lemma}
Let $m\in \mathbb{R}$ and $\rho\in [0, +\infty)$.
 If $a\in \Gamma^m_0$, then the paradifferential operator $T_a$ is of order m. Moreover, for all $s\in \mathbb{R}$, there exists a positive constant $K$ such that
\begin{equation*}
\|T_a\|_{H^s\rightarrow H^{s-m}} \leq K M^m_0(a). 
\end{equation*}
\end{lemma}

In Besov spaces, we have similar results for symbolic calculus.
\begin{lemma}[\hspace{1sp}\cite{MR3585049}]
Let $m, m^{'}, s\in \mathbb{R}$, $q\in [1,\infty]$ and $\rho\in [0,1]$.
\begin{enumerate}
\item If $a\in \Gamma^m_0$, then there exists a positive constant $K$ such that
\begin{equation*}
\|T_a\|_{B^s_{\infty, q}\rightarrow B^{s-m}_{\infty, q}} \leq KM^m_0(a).
\end{equation*}
\item If $a\in \Gamma^m_\rho$, and $b\in \Gamma^m_\rho$, then there exists a positive constant $K$ such that
\begin{equation}
  \|T_a T_b-T_{a b} \|_{B^s_{\infty, q} \rightarrow B^{s-m-m^{'}+\rho}_{\infty, q}} \leq K\left(M^m_\rho(a)M^{m^{'}}_0(b) + M^m_0(a)M^{m^{'}}_\rho(b)\right). \label{CompositionTwo}
\end{equation}
\end{enumerate}
In particular, when $q = \infty$, above symbolic calculus results hold for Zygmund spaces $C^s_{*}$.
\end{lemma}

When $a$ is only a function of $x$, $T_a u$ is the low-high paraproduct.
We then define 
\begin{equation*}
\Pi(a, u) := au -T_a u -T_u a
\end{equation*}
to be the balanced paraproduct.
For later use, we record below some estimates for paraproducts.

\begin{lemma} \label{t:ParaProductEst}
\begin{enumerate}
\item Let $\alpha, \beta \in \mathbb{R}$. 
If $\alpha+ \beta >0$, then
\begin{align}
&\|\Pi(a, u)\|_{H^{\alpha + \beta}(\mathbb{R})} \lesssim \| a\|_{C_{*}^\alpha(\mathbb{R})} \| u\|_{H^\beta(\mathbb{R})}, \label{HCHEstimate}\\
& \|\Pi(a, u)\|_{C_{*}^{\alpha + \beta}(\mathbb{R})} \lesssim \| a\|_{C_{*}^\alpha(\mathbb{R})} \| u\|_{C_{*}^\beta(\mathbb{R})}  \label{CCCEstimate}
\end{align}
\item Let $m > 0$ and $s\in \mathbb{R}$, then
\begin{align}
&\|T_{a} u \|_{H^{s-m}(\mathbb{R})} \lesssim \|a\|_{C_{*}^{-m}(\mathbb{R})} \| u \|_{H^s(\mathbb{R})} \label{HsCmStar}, \\
&\|T_{a} u \|_{H^{s}(\mathbb{R})} \lesssim \|a\|_{L^\infty(\mathbb{R})} \| u \|_{H^s(\mathbb{R})}, \label{HsLinfty}\\
&\|T_{a} u \|_{H^{s-m}(\mathbb{R})} \lesssim \|a\|_{H^{-m}(\mathbb{R})} \| u \|_{C_{*}^s(\mathbb{R})}, \label{HsHmCStar}\\
&\|T_{a} u \|_{C_{*}^{s-m}(\mathbb{R})} \lesssim \|a\|_{C_{*}^{-m}(\mathbb{R})} \| u \|_{C_{*}^s(\mathbb{R})}, \label{CsCmStar}\\
&\|T_{a} u \|_{C_{*}^{s}(\mathbb{R})} \lesssim \|a\|_{L^\infty(\mathbb{R})} \| u \|_{C_{*}^s(\mathbb{R})}. \label{CsLInfty}
\end{align}
\end{enumerate}
\end{lemma}

Using the above paraproducts estimates, we get the following result.
\begin{lemma}
\begin{enumerate}
\item If $s \geq 0$, then
\begin{align}
&\|uv\|_{H^s} \lesssim \|u\|_{H^s}\|v\|_{L^\infty}+ \|u\|_{L^\infty}\|v\|_{H^s},\label{HsProduct} \\
&\|uv\|_{C_{*}^s} \lesssim \|u\|_{C_{*}^s}\|v\|_{L^\infty}+ \|u\|_{L^\infty}\|v\|_{C_{*}^s}. \label{CsProduct}
\end{align}    
\item Let  a smooth function $F\in C^\infty(\mathbb{C}^N)$ satisfying $F(0) = 0$.
There exists a nondecreasing function $\mathcal{F}: \mathbb{R}_{+} \rightarrow \mathbb{R}_{+}$ such that,
\begin{align}
&\|F(u) \|_{H^s} \leq \mathcal{F}(\|u\|_{L^\infty}) \|u\|_{H^s},\quad s\geq 0,  \label{MoserOne}\\
&\|F(u) \|_{C_{*}^s} \leq \mathcal{F}(\|u\|_{L^\infty}) \|u\|_{C_{*}^s},\quad s>0. \label{MoserTwo}
\end{align}
\item Let $s_1 > s_2 > 0$, then
\begin{equation}
    \|uv\|_{C^{-s_2}_*}\lesssim \|u\|_{C^{s_1}_*} \|v\|_{C^{s_2}_*}. \label{CNegativeAlpha}
\end{equation}
\end{enumerate}
\end{lemma}

\begin{lemma}[Commutator estimates \cite{ai2023dimensional}] 
The following commutator estimates hold for $1< p< \infty$:
\begin{align}
 &\| [\nP, f] \langle D \rangle^{s_2} g\|_{W^{s_1, p}} \lesssim \| f\|_{C^{s_1+s_2}_{*}}\|g\|_{L^p}, \quad s_1, s_2\geq 0,  \label{CommutatorPBMO}\\
 &\| [\nP, f] \langle D \rangle^{s_2} g\|_{W^{s_1, p}} \lesssim \| f\|_{W^{s_1+s_2,p}}\|g\|_{C_{*}^0}, \quad s_1\geq 0,\; s_2> 0.  \label{CommutatorPBMOTwo}
\end{align}
\end{lemma}
Here we change the $BMO$ spaces to Zygmund spaces as in Lemma 2.1 of \cite{ai2023dimensional}.
We also change the homogeneous spaces to non-homogeneous spaces.
Although no assumptions are made on $f$ or $g$ in this lemma, later it will be applied to $f$ or $g$ which are either holomorphic or anti-holomorphic.

Finally, we need the following result of paralinearization in Besov spaces.
\begin{lemma}[Paralinearization \cite{MR2768550}] \label{t:Paralinear}
Let $s, \rho>0$, and $F(u)$ be a smooth function of $u$.
Assume that $\rho$ is not an integer.
Let $p, r_1, r_2 \in [1, \infty]$ and such that $r_2 \geq r_1$.
Let $r \in [1,\infty]$ be defined by $\frac{1}{r} = \min \{1, \frac{1}{r_1}+\frac{1}{r_2} \}$.
Then for any $u\in B^s_{p, r_1}\cap B^\rho_{\infty, r_2}$,
\begin{equation*}
    \|F(u)-F(0)-T_{F^{'}(u)}u\|_{B^{s+\rho}_{p,r}} \leq C(\|u\|_{L^\infty})\|u\|_{B^\rho_{\infty, r_2}}\|u\|_{B^s_{p. r_1}}.
\end{equation*}
\end{lemma}
We remark that this lemma also works for multivariable functions $F$. 
We simply need to replace $F^{'}$ by partial derivatives of $F$.
See for instance Lemma $3.26$ in ~\cite{MR2805065}.
We will apply this paralinearization result with $p=r=r_1 =r_2 = \infty$, so that this is an estimate in Zygmund spaces.

\subsection{Paradiffential estimates for bilinear forms}
In the following, we consider the estimates for the bilinear forms.
Let $\chi_1(\theta_1, \theta_2), \chi_2(\theta_1, \theta_2)$ be two non-negative smooth bump functions 
\begin{equation}
    \chi_1(\theta_1, \theta_2) = \left\{
\begin{aligned}
1, \text{ when } |\theta_1|\leq \frac{1}{20}|\theta_2|\\
0, \text{ when } |\theta_1|\geq \frac{1}{10}|\theta_2|,
\end{aligned}
\right.  \label{ChiOnelh}
\end{equation}
\begin{equation}
  \chi_2(\theta_1, \theta_2) = \left\{
\begin{aligned}
&1, \text{ when } \frac{1}{10}\leq \frac{|\theta_1|}{|\theta_2|} \leq 10\\
&0, \text{ when } |\theta_1|\leq \frac{1}{20}|\theta_2| \text{ or } |\theta_2|\leq \frac{1}{20}|\theta_1|,
\end{aligned}
\right.  \label{ChiTwohh}
\end{equation}
and such that $\chi_1(\theta_1, \theta_2) + \chi_1(\theta_2, \theta_1) + \chi_2(\theta_1, \theta_2) =1$.
For bilinear forms $B(u,v)$ with symbol $m(\xi, \eta)$, we  define  the paradifferential bilinear forms:
\begin{itemize}
\item Low-high part and balanced part of the holomorphic bilinear forms:
\begin{align*}
 \widehat{B_{lh}(u,v)}(\zeta) &= \int_{\zeta = \xi +\eta} \chi_1\left(\xi, \eta+\xi\right) m(\xi, \eta)\hat{u}(\xi)\hat{v}(\eta) d\xi,\\
  \widehat{B_{hh}(u,v)}(\zeta) &= \int_{\zeta = \xi +\eta} \chi_2\left(\xi, \eta+\xi\right) m(\xi, \eta)\hat{u}(\xi)\hat{v}(\eta) d\xi.
\end{align*}
\item Low-high part and balanced part of the mixed bilinear forms:
\begin{align*}
 \widehat{B_{lh}(u,v)}(\eta) &= 1_{\eta<0}\int_{\eta = \zeta-\xi} \chi_1\left(\xi, \zeta-\xi\right) m(\xi, \zeta)\bar{\hat{u}}(\xi)\hat{v}(\zeta) d\xi,\\
  \widehat{B_{hh}(u,v)}(\eta) &=1_{\eta<0}\int_{\eta = \zeta-\xi} \chi_2\left(\xi, \zeta-\xi\right) m(\xi, \zeta)\bar{\hat{u}}(\xi)\hat{v}(\zeta) d\xi.
\end{align*}
\end{itemize}

These represent low-high and balanced paradifferential parts of the bilinear forms $B(u,v)$, respectively $\nP B(\bar{u},v)$, restricted to the holomorphic class. 
We will always assume that bilinear symbols $m$ are homogeneous, and smooth away from $(0,0)$.

When the bilinear symbol $m$  is homogeneous, we have the following direct generalization of Lemma \ref{t:ParaProductEst} for bilinear forms, see Coifman-Meyer \cite{MR0518170}, Kenig-Stein \cite{MR1682725},  Muscalu \cite{MR2371442}, and Muscalu-Tao-Thiele \cite{MR1887641}.
\begin{lemma} \label{t:SymbolPara}
Let $B^\mu(f,g)$ be a homogeneous  bilinear form of order $\mu\geq 0$ as above. 
For the balanced bilinear forms, when $\alpha+\beta+\mu>0$, $\mu_1 + \mu_2 = \mu$
\begin{align}
&\|B^\mu_{hh}(f, g)\|_{H^{\alpha + \beta}} \lesssim \| f\|_{C_{*}^{\alpha + \mu_1}} \| g\|_{H^{\beta + \mu_2}}, \label{HHBilinearHCH}\\
& \|B^\mu_{hh}(f, g)\|_{C_{*}^{\alpha + \beta}} \lesssim \| f\|_{C_{*}^{\alpha+\mu_1}} \| g\|_{C_{*}^{\beta+\mu_2}}.  \label{HHBilinearCCC}
\end{align}

For the estimates of low-high bilinear forms,
\begin{align}
&\|B^\mu_{lh}(f, g) \|_{H^{s-m}} \lesssim \|f\|_{C_{*}^{-m}} \| g \|_{H^{s+\mu}} \label{BFHsCmStar}, \\
&\|B^\mu_{lh}(f, g) \|_{H^{s}} \lesssim \|f\|_{L^\infty(\mathbb{R})} \| g \|_{H^{s+\mu}}, \label{BFHsLinfty}\\
&\|B^\mu_{lh}(f, g) \|_{C_{*}^{s-m}} \lesssim \|f\|_{C_{*}^{-m}} \| g \|_{C_{*}^{s+\mu}}, \label{BFCsCmStar}\\
&\|B^\mu_{lh}(f, g) \|_{C_{*}^{s}} \lesssim \|f\|_{L^\infty} \| g \|_{C_{*}^{s+\mu}}. \label{BFCsLInfty}     
\end{align}
\end{lemma}

Below we have the following para-commutator, paraproduct, and para-associativity lemmas where $f$ and $g$ are functions.
\begin{lemma} \label{t:ParaCoefficient}
\begin{enumerate}
\item (Para-commutators)
Assume that $s_1, s_2<1$, $s_1 +s_2>0$ then
\begin{align}
& \|T_fT_g -T_gT_f\|_{H^s \rightarrow H^{s+s_1+s_2}} \lesssim \|f\|_{C^{s_1}_{*}}\|g\|_{C^{s_2}_{*}},\label{ParaCommutator} \\
&  \|T_fT_g -T_gT_f\|_{C_{*}^s \rightarrow C_{*}^{s+s_1+s_2}} \lesssim \|f\|_{C^{s_1}_{*}}\|g\|_{C^{s_2}_{*}},\label{ParaCommutatorTwo} \\
&  \|T_fT_g -T_gT_f\|_{C_{*}^s \rightarrow H_{*}^{s+s_1+s_2}} \lesssim \|f\|_{H^{s_1}}\|g\|_{C^{s_2}_{*}}.\label{ParaCommutatorThree} 
\end{align}
\item (Para-products)
Assume that $s_1, s_2<1$, $s_1 +s_2>0$ then 
\begin{align}
&\|T_fT_g -T_{fg}\|_{H^s \rightarrow H^{s+s_1+s_2}} \lesssim \|f\|_{C^{s_1}_{*}}\|g\|_{C^{s_2}_{*}},\label{ParaProducts} \\
&\|T_fT_g -T_{fg}\|_{C_{*}^s \rightarrow C_{*}^{s+s_1+s_2}} \lesssim \|f\|_{C^{s_1}_{*}}\|g\|_{C^{s_2}_{*}}.\label{ParaProductsTwo} 
\end{align}
 \item (Para-associativity) For $s+s_2 > 0$,  and $0<s_1<1$, 
\begin{align}
&\|T_f\Pi(v,u) - \Pi(v, T_f u)\|_{H^{s+s_1+s_2}} \lesssim \|f\|_{C^{s_1}_{*}}\|v\|_{C^{s_2}_{*}} \| u\|_{H^{s}},\label{ParaAssociateOne} \\
&\|T_f\Pi(v,u) - \Pi(v, T_f u)\|_{C_{*}^{s+s_1+s_2}} \lesssim \|f\|_{C^{s_1}_{*}}\|v\|_{C^{s_2}_{*}} \| u\|_{C_{*}^{s}}.\label{ParaAssociateThree}
\end{align}
\item For $0<s_1, s_2< 1$,  $s+s_2 \geq 0$ and $\bar{v} = \bar{P}\bar{v}$, 
\begin{align}
&\|T_fP(\bar{v}u) - P(\bar{v}T_f u)\|_{H^{s+s_1+s_2}} \lesssim \|f\|_{C^{s_1}_{*}}\|v\|_{C^{s_2}_{*}}\| u\|_{H^{s}}, \label{ParaAssociateTwo} \\
&\|T_fP(\bar{v}u) - P(\bar{v}T_f u)\|_{C_{*}^{s+s_1+s_2}} \lesssim \|f\|_{C^{s_1}_{*}}\|v\|_{C^{s_2}_{*}}\| u\|_{C_{*}^{s}}. \label{ParaAssociateFour}
\end{align}
\end{enumerate}
\end{lemma}
The proof of this lemma is identical to Lemma 2.4-2.7 in \cite{ai2023dimensional}.
We roughly sketch the idea of proof for $(1)$ and $(2)$.
By orthogonality, it suffices to consider the case where the frequency of $u$ is $\lambda$.
If both  $f$ and $g$ are at frequency $O(\lambda)$, the estimates are direct.
If $f$ is at frequency $O(\lambda)$, and  $g$ is at frequency $\ll \lambda$, then we rewrite the commutator,
\begin{equation*}
    [T_{f_\lambda}, T_{g_{\ll \lambda}}]u_{\lambda} = [T_{f_\lambda}, g_{\ll \lambda}]u_{\lambda} =  \lambda^{-1}L(f_{\lambda}, \partial_x g_{\ll \lambda}, u_{\lambda}), \quad (T_{f_\lambda}g_{\ll \lambda}- T_{f_\lambda g_{\ll \lambda}})u_\lambda = \lambda^{-1}L(f_{\lambda}, \partial_x g_{\ll \lambda}, u_{\lambda}),
\end{equation*}
where $L$ is a translation invariant trilinear form with integrable kernel.
These estimates follow by estimating the trilinear form.

Finally, we record here the para-Leibniz
rule.
Recall the definition of \textit{para-material derivative},
\begin{equation*}
    T_{D_t} := \partial_t + T_{b}\partial_\alpha.
\end{equation*}
We then consider the four versions of para-Leibniz errors.
The first two are unbalanced para-Leibniz errors.
\begin{align*}
    &E^p_L(u,v) = T_{D_t}T_u v - T_{T_{D_t}u}v -T_uT_{D_t}v,\\
    & \tilde{E}^p_L(u,v) = T_{D_t}T_u v - T_{D_{t} u}v -T_uT_{D_t}v.
\end{align*}
The other two are the balanced para-Leibniz errors.
\begin{align*}
    E^\pi_L(u,v) &= T_{D_t}\Pi(u,v) -\Pi(T_{D_t}u,v) -\Pi(u, T_{D_t}v), \\
    \tilde{E}^\pi_L(u,v) &= T_{D_t}\Pi(u,v) -\Pi(D_t u,v) -\Pi(u, T_{D_t}v).
\end{align*}
With the above notation, the Leibniz error can be bounded according to the following lemma.
\begin{lemma} \label{t:Leibniz}
\begin{enumerate}
\item For the unbalanced para-Leibniz error $E^p_L(u,v)$, we have the bounds
\begin{align}
    &\|E^p_{L}(u,v)\|_{H^s} \lesssim \CalAO \|u\|_{C^{\Half-\sigma}_{*}}\|v\|_{H^{s+\sigma}}, \quad \sigma >0, \label{UnbalancedParaLeibnizOne}\\
    &\|E^p_{L}(u,v)\|_{H^s} \lesssim \CalAO \|u\|_{H^{\frac{1}{2}-\sigma}}\|v\|_{C_{*}^{s+\sigma}}, \quad \sigma >0, \label{UnbalancedParaLeibnizTwo} \\
    &\|E^p_{L}(u,v)\|_{C_{*}^s} \lesssim \CalAO \|u\|_{C_{*}^{\frac{1}{2}-\sigma}}\|v\|_{C_{*}^{s+\sigma}}, \quad \sigma >0. \label{UnbalancedParaLeibnizThree}
\end{align}
In the case $\sigma=0$ the same bounds \eqref{UnbalancedParaLeibnizOne} and \eqref{UnbalancedParaLeibnizTwo} hold for $\tilde{E}^p_L(u,v)$ with $\sigma =0$.
\item For the balanced para-Leibniz error  $E^\pi_L(u,v)$, we have the estimates
\begin{align}
    &\|E^\pi_L(u,v) \|_{H^s} \lesssim \CalAO \|u\|_{C^{\frac{1}{2}-\sigma}_{*}}\|v\|_{H^{s+\sigma}}, \qquad \sigma \in \mathbb{R},\quad s> 0, 
    \label{BalancedLeibniz}\\
     &\|E^\pi_L(u,v) \|_{C_{*}^s} \lesssim \CalAO \|u\|_{C^{\frac{1}{2}-\sigma}_{*}}\|v\|_{C_{*}^{s+\sigma}}, \qquad \sigma \in \mathbb{R},\quad s> 0. 
    \label{BalancedLeibnizTwo}
\end{align}
In the case $\sigma =0$, we also have the same bound for $\tilde{E}^\pi_L(u,v)$.
\end{enumerate} 
\end{lemma}
The proof of this lemma is almost the same as Lemma $3.6$ in \cite{ai2023dimensional}, we only sketch the proof here.
For the unbalanced para-Leibniz error, when $\sigma >0$, 
\begin{equation*}
 E^p_{L}(u,v) = (T_b T_{\partial_\alpha u}v -T_{T_b \partial_\alpha u}v) -[T_b, T_u]\partial_\alpha v.  
\end{equation*}
For the first difference, $b$ cannot be at the low frequency, so that
\begin{equation*}
T_b T_{\partial_\alpha u}v -T_{T_b \partial_\alpha u}v =  \sum_{j\leq l<k} \partial_\alpha u_j b_l v_k. 
\end{equation*}
When $\sigma = 0$, we write
\begin{equation*}
 E^p_{L}(u,v) = (T_b T_{\partial_\alpha u}v -T_{b \partial_\alpha u}v) -[T_b, T_u]\partial_\alpha v.  
\end{equation*}
Then the bound for the difference follows by using \eqref{BCOneStar}.
The bound for the commutator term follows from the para-commutator estimates \eqref{ParaCommutator}, \eqref{ParaCommutatorTwo}.
As for the balanced para-Leibniz error, we can write
\begin{equation*}
E^\pi_L(u,v) = T_b \Pi(\partial_\alpha u, v)- \Pi(T_b\partial_\alpha u, v)+T_b \Pi(u, \partial_\alpha v)-\Pi(u, T_b \partial_\alpha v).
\end{equation*}
The estimates are the consequence of applying para-associativity bounds \eqref{ParaAssociateOne}, \eqref{ParaAssociateThree}.

\bibliography{ww}
\bibliographystyle{plain}
\end{document}